\documentclass[10pt,a4paper,final]{article}
\usepackage[utf8x]{inputenc}
\usepackage{amsmath}
\usepackage{amsfonts}
\usepackage{amsthm}
\usepackage{amssymb}
\usepackage{graphicx}
\usepackage[margin=0.8in]{geometry}
\usepackage[colorlinks=true]{hyperref}
\usepackage[capitalise]{cleveref}

\allowdisplaybreaks

\usepackage{mathtools}
\usepackage{apptools}
\usepackage{xcolor}
\usepackage{enumitem}
\usepackage{mathabx}
\usepackage{cases}
\usepackage{stmaryrd}
\usepackage[framemethod=tikz]{mdframed}
\usepackage{textgreek}
\usepackage{mathdots}

\newtheorem{theorem}{Theorem}
\newtheorem{lemma}[theorem]{Lemma}
\newtheorem{proposition}[theorem]{Proposition}
\newtheorem{corollary}[theorem]{Corollary}

\theoremstyle{definition}

\newtheorem{remark}[theorem]{Remark}

\usepackage{chngcntr}
\usepackage{apptools}
\AtAppendix{\counterwithin{theorem}{section}}

\setlist[enumerate]{label=$\rm{(\roman*)}$,leftmargin=\parindent,wide, labelwidth=!, labelindent=0pt}

\numberwithin{equation}{section}
\numberwithin{theorem}{section}
\numberwithin{table}{section}
\numberwithin{figure}{section}

\newcommand{\norm}[1]{\left\lVert #1 \right\rVert}
\newcommand{\abs}[1]{\left\lvert #1 \right\rvert}
\newcommand{\scal}[1]{\left\langle #1 \right\rangle}
\newcommand{\seq}[1]{\left\lbrace #1 \right\rbrace}

\newcommand{\sR}{\mathbb{R}}
\newcommand{\sL}{\mathbb{L}}

\newcommand{\sH}{\mathcal{H}}
\newcommand{\sG}{\mathcal{G}}

\newcommand{\Fix}{\mathsf{Fix}}

\newcommand{\zer}{\textsf{Zer}}

\newcommand{\mysum}{\displaystyle \sum\limits}

\newcommand{\proj}{\mathsf{P}}
\newcommand{\dist}{\mathsf{d}}
\newcommand{\Id}{\mathrm{Id}}
\newcommand{\bO}{\mathcal{O}}

\newcommand{\BP}{Banach-Picard }
\newcommand{\KM}{Krasnosel’ski\u{\i}-Mann }

\newcommand{\E}{\mathcal{E}}

\newcommand{\tcr}{\textcolor{red}}

\title{Tikhonov regularization of monotone operator flows not only ensures strong convergence of the trajectories but also speeds up the vanishing of the residuals}
\author{Radu Ioan Bo\c{t}
	\footnote{Faculty of Mathematics, University of Vienna, Oskar-Morgenstern-Platz 1, 1090 Vienna, Austria, email: \url{radu.bot@univie.ac.at}. Research partially supported by FWF (Austrian Science Fund), projects W 1260 and P 34922-N.}
\and Dang-Khoa Nguyen
	\footnote{Faculty of Mathematics and Computer Science, University of Science, Ho Chi Minh City, Vietnam, email: \url{ndkhoa@hcmus.edu.vn}.}
	\footnote{Vietnam National University, Ho Chi Minh City, Vietnam.  Research funded by the Vietnam National University,  Ho Chi Minh City (VNU-HCM) under grant number T.C2025-18-04.}}

\begin{document}

\maketitle	

\begin{abstract}
In the framework of real Hilbert spaces, we investigate first-order dynamical systems governed by monotone and continuous operators.  It has been established that for these systems, only the ergodic trajectory converges to a zero of the operator. A notable example is the counterclockwise $\pi/2$-rotation operator on $\mathbb{R}^2$, which illustrates that,  in general, trajectory convergence cannot be expected.  However, trajectory convergence is assured for operators with the stronger property of cocoercivity.  For this class of operators,  the trajectory's velocity and the opertor values along the trajectory converge in norm to zero at a rate of $o(\frac{1}{\sqrt{t}})$ as $t \rightarrow + \infty$.

In this paper, we demonstrate that when the monotone operator flow is augmented with a Tikhonov regularization term, the resulting trajectory converges strongly to the element of the set of zeros with minimal norm.  In addition,  rates of convergence in norm for the trajectory's velocity and the operator along the trajectory can be derived in terms of the regularization function.  In some particular cases,  these rates of convergence can outperform the ones of the coercive operator flows and can be as fast as $\bO(\frac{1}{t})$ as $t \rightarrow + \infty$.  In this way, we emphasize a surprising acceleration feature of the Tikhonov regularization.  Additionally, we explore these properties for monotone operator flows that incorporate time rescaling and an anchor point. For a specific choice of the Tikhonov regularization function, these flows are closely linked to second-order dynamical systems with a vanishing damping term. The convergence and convergence rate results we achieve for these systems complement recent findings for the Fast Optimistic Gradient Descent Ascent (OGDA) dynamics, leading to surprising outcomes.

When the monotone operator is defined as the identity minus a nonexpansive operator, the monotone equations transform into a fixed point problem. In such cases, explicitly discretizing the system with Tikhonov regularization, enhanced by an anchor point, leads to the Halpern fixed point iteration. Through convergence analysis, which can be viewed as the discrete counterpart of the continuous time analysis, we identify two regimes for the regularization sequence which ensure that the generated sequence of iterates converges strongly to the fixed point nearest to the anchor point.  Furthermore, we establish a general theoretical framework that provides convergence rates for the vanishing of the discrete velocity and the fixed point residual.  For certain regularization sequences, we derive specific convergence rates that align with those observed in continuous time.
\end{abstract}	

\vspace{1ex}


\tableofcontents

\section{Introduction}\label{sec1}

Let $\sH$ be a real Hilbert space and $M \colon \sH \to \sH$ a continuous and monotone operator.  We say that $M \colon \sH \to \sH$ is \emph{monotone} if
\begin{equation*}
\scal{M \left( x \right) - M \left( y \right) , x - y} \geq 0 \quad \forall x, y \in \sH.
\end{equation*}
Being monotone and continuous,  $M$ is also maximally monotone (\cite{Bauschke-Combettes:book}),  which means that there is no other (set-valued) monotone operator with the graph in $\sH \times \sH$ being a strict supset of the graph of $M$.  Therefore, the set $\zer M\coloneq \{x \in \sH: M(x) = 0 \}$ is convex and closed.

By assuming throughout this work that $\zer M$ is nonempty, we are interested in the investigation of continuous time models that solve the following equation.
\begin{mdframed}
\begin{equation}
\label{intro:mono}
\textrm{ Find } x \in \sH \textrm{ such that } M \left( x \right) = 0.
\end{equation}
\end{mdframed}

A broad class of problems arising in applications can be reduced to the solving of a monotone equation in real Hilbert spaces such as \eqref{intro:mono}.  This ranges from classical frameworks in optimization,  mechanics or partial differential equations,  such as saddle-point problems, variational inequalities,  game theory and Nash equilibrium problems,  to modern applications such as signal processing and image reconstruction, machine learning and data science.  For $x \in \sH$, we denote by $\proj_{\zer M}(x)$ its projection onto $\zer M$ and by ${\dist}_{\zer M}(x) \coloneq \norm{x - \proj_{\zer M}(x)}$ its distance to the set $\zer M$.

If the operator $M \coloneq \nabla f$ is the (not necessarily Lipschitz continuous) gradient of a convex and continuously differentiable function $f \colon \sH \to \sR$,  then it is monotone and continuous and the solutions of \eqref{intro:mono} are exactly the optimal solutions of the optimization problem
\begin{equation}
\label{intro:opti}
\min_{x \in \sH} f \left( x \right).
\end{equation}
In this case, we attach to \eqref{intro:opti} the gradient flow
\begin{equation}
\label{ds:sd}
\dot{x} \left( t \right) + \nabla f \left( x \left( t \right) \right) = 0,
\end{equation}
with initial condition $x \left( t_{0} \right) \coloneq x_{0} \in \sH$. Whenever $\arg\min f \neq \emptyset$,  i.e. the set of minimizers of $f$ is not empty,  the generated trajectory converges weakly to an element of $\arg\min f$ and $f(x(t)) - \min f = o \left( \frac{1}{t} \right)$ as $t \rightarrow +\infty$, where $\min f$ denotes the minimal value of $f$.

Inspired by the gradient flow,  we attach to \eqref{intro:mono} the following dynamics
\begin{equation}
	\label{ds:simp}
\dot{x} \left( t \right) + M \left( x \left( t \right) \right) = 0,
\end{equation}
with starting time $t_{0} \geq 0$ and initial condition $x \left( t_{0} \right) \coloneq x_{0} \in \sH$. We assume that it admits a unique strong global solution $x : [t_{0}, +\infty) \rightarrow \sH$ (see \cite[Appendix]{Brezis}) with the property that $t \mapsto M(x(t))$ is absolutely continuous on every compact interval $[t_{0},T]$.  If $M \colon \sH \to \sH$ is Lipschitz continous, then this holds for every $x_0 \in \sH$.  It has been long recognized (see, for instance,  \cite{Baillon-Brezis}) that only the ergodic trajectory $t  \mapsto  \frac{1}{t-t_{0}} \int_{t_{0}}^t x(s) ds$ converges weakly to a zero of $M$, while the trajectory $t \mapsto x(t)$ may not converge in general. To see this,  one can consider the example (\cite{Cominetti-Peypouquet-Sorin}) of the counterclockwise \(\pi/2\)-rotation operator on \(\mathbb{R}^2\), $M \colon \sR^{2} \to \sR^{2} ,  M(x_1,x_2) = \left( - x_2 , x_1 \right)$,  which has $\left( 0 , 0 \right)$ as its unique zero.  For $t_{0}=0$ and $(x_1(0), x_2(0)) \coloneq (r_0 \cos (\theta_0),  r_0 \sin (\theta_0))$,  for $r_0 >0$ and $\theta_0 \in \sR$,  the solution trajectory of  \eqref{ds:simp},  that is $\left( x_1(t), x_2(t) \right) = \left( r_0 \cos (\theta_0+t), r_0 \sin (\theta_0+t) \right)$ travels with a counterclockwise orbit on the circle with radius $r0$ centred on the origin,  so it is bounded but has no limit point  as $t \rightarrow +\infty$.  On the other hand,  the ergodic trajectory $t \mapsto \frac{1}{t} \int_{0}^t (x_1(s), x_2(s)) ds$ converges to $(0,0)$ as $t \rightarrow +\infty$.

\subsection{The cocoercive case}

The situation changes drastically when $M \colon \sH \to \sH$ is a $\rho$-cocoercive operator for some $\rho > 0$, that is
\begin{equation*}
	\scal{M \left( x \right) - M \left( y \right) , x - y} \geq \rho \norm{M \left( x \right) - M \left( y \right)}^{2} \quad \forall x, y \in \sH .
\end{equation*}
A $\rho$-cocoercive operator is obviously monotone and $\frac{1}{\rho}$-Lipschitz continuous.  In this case,  the solution trajectory $t \mapsto x(t)$ of \eqref{ds:simp} converges weakly to a zero of $M$, and it holds that $\norm{\dot x(t)} = \norm{M(x(t))} = o \left( \frac{1}{\sqrt{t}}  \right)$ as $t \rightarrow +\infty$. According to \cite[Theorem 11]{Attouch-Bot-Nguyen},  this results holds true even for the trajectory of a cocoercive operator flow with small perturbations.  For completeness we provide the proof of the statement for the system without perturbations in Theorem \ref{thm:cocoercive} in the Appendix \tcr{(}see, also,  \cite{Bot-Csetnek}).

In the following, we will mention some typical examples of cocoercive operators, which obviously also serve as examples of monotone operators.

\paragraph{Minimizing a convex and differentiable function with a Lipschitz continuous gradient.}

For every convex and continuously differentiable function $f \colon \sH \to \sR$ it follows from the Baillon-Haddad Theorem (see \cite[Corollary 18.17]{Bauschke-Combettes:book}) that its gradient $\nabla f \colon \sH \to \sH$ is $\frac{1}{L}$-cocoercive if and only if it is $L$-Lipschitz continuous.  

\paragraph{The \emph{Yosida approximation} of a maximally monotone operator.}

The \emph{Yosida approximation} ${}^{\rho} \! A \colon \sH \to \sH$ with parameter $\rho > 0$ of a set-valued maximally monotone operator $A \colon \sH \rightrightarrows \sH$ is the single-valued operator defined as
\begin{equation*}
{}^{\rho} \! A \colon \sH \rightarrow \sH,   \quad {}^{\rho} \! A \coloneq \frac{1}{\rho} \left( \Id - J_{\rho A} \right) ,
\end{equation*}
where  $J_{\rho A} \coloneq \left( \Id + \rho A \right) ^{-1} \colon \sH \to \sH$ denotes the \emph{resolvent} of $A$ with parameter $\rho > 0$.  The operator ${}^{\rho} \! A$ is $\rho$-cocoercive (\cite[Corollary 23.11]{Bauschke-Combettes:book}) and it holds  $\zer {}^{\rho} \! A = \zer A \coloneq \{x \in \sH : 0 \in A(x)\}$.  This emphasizes the strong connection between monotone  inclusions and equations governed by cocoercive operators.  Therefore, we can use \eqref{intro:mono} to solve a more general problem, namely to find the zero of a maximally monotone operator.

\paragraph{Fixed point problem.}

An operator $T \colon \sH \rightarrow \sH$ is nonexpansive,  meaning it is $1$-Lipschitz continuous,  if and only if $\Id -T$ is $\frac{1}{2}$-cocoercive,  where $\Id \colon \sH \to \sH$ denotes the identity mapping on $\sH$.  This emphasizes the strong connection between fixed point problems and equations governed by cocoercive operators.  

We want to note that even in the case of a cocoercive operator, one cannot generally expect the solution trajectory of \eqref{ds:simp} to strongly converge to a zero of this operator.  A counterexample in this sense was provided in \cite{Attouch-Baillon} for a dynamical system governed by the gradient of a convex and continuously differentiable function with a Lipschitz continuous gradient.

\subsection{The general monotone case}

In spite of being ubiquitous in applications, we will not focus in this paper on continuous time models for \eqref{intro:mono} driven by cocoercive operators, but rather by monotone and continuous ones, which constitute a larger class of operators.  Operators such as
\begin{equation*}
M \colon \sH \times \sG \rightarrow \sH \times \sG,  \quad M(x,y) \coloneq (\nabla_x \Phi(x,y), -\nabla_y \Phi(x,y)),
\end{equation*}
where $\sG$ is another real Hilbert space and $\Phi : \sH \times \sG \rightarrow \sR$ is a convex-concave function with a continuous gradient, are monotone and continuous but generally not cocoercive,  as illustrated by the example of a nontrivial bilinear operator.  For such operators, \eqref{intro:mono} is equivalent to the problem of finding the saddle points of the minimax problem
\begin{align*}
    \min_{x\in\sH}\max_{y\in\sG} \Phi(x,y),
\end{align*}
which is of great relevance in game theory and in convex optimization in the context of linear equality constraints.

To approach the solution of \eqref{intro:mono} from a continuous time perspective,  we start again from its special case related to optimization problems.  Relying  on the previous works \cite{Su-Boyd-Candes, Attouch-Chbani-Peypouquet-Redont},  the second-order continuous time model 
\begin{equation}
	\label{ds:sd:inertial}
\ddot x (t) + \frac{\alpha}{t}\dot{x} \left( t \right) + \beta \nabla^2 f(x(t)) \dot x(t) + \nabla f \left( x \left( t \right) \right) = 0,
\end{equation}
with $\alpha \geq 3$ and $\beta \geq 0$, initial time $t_{0} >0$ and initial conditions $x \left( t_{0} \right) \coloneq x_{0} \in \sH$ and $\dot x \left( t_{0} \right) \coloneq \dot x_{0} \in \sH$,  was introduced in \cite{Attouch-Peypouquet-Redont} in the context of minimizing a convex and twice differentiable function $f : \sH \rightarrow \sR$ with $\arg\min f \neq \emptyset$. The primary aim was to accelerate the convergence of the function values along the generated solution trajectory.  Indeed, if $\alpha \geq 3$ and $\beta \geq 0$, then $f(x(t)) - \min f = \bO(\frac{1}{t^2})$ as $t \rightarrow +\infty$, while if $\beta >0$,  then $\int_{t_{0}}^{+\infty} t^2 \|\nabla f(x(t))\|^2dt < +\infty$. If $\alpha >3$, then the trajectory converges weakly to an element of $\arg\min f$ and $f(x(t)) - \min f = o(\frac{1}{t^2})$ as $t \rightarrow +\infty$.

Recently, in \cite{Bot-Csetnek-Nguyen},  in the context of a  monotone equation such as \eqref{intro:mono}, a second-order dynamical system has been proposed that combines a vanishing damping term with the time derivative of the operator along the trajectory, which can be seen as an analogous of the Hessian-driven damping term in \eqref{ds:sd:inertial}.  This model exhibits fast convergence rates rates of convergence in norm for the trajectory's velocity and the operator along the trajectory,  outperforming those of the cocoercive operator flow.  This phenomenon is analogous to the one described above for optimization, despite the fact that the operator does not have to originate from a potential. The precise formulation of this so-called Fast Optimistic Gradient Descent Ascent (OGDA) dynamics is
\begin{equation}	
\label{ds:fast-OGDA:intro}
\ddot{x} \left( t \right) + \dfrac{\alpha}{t} \dot{x} \left( t \right) + \beta \left( t \right) \dfrac{d}{dt} M \left( x \left( t \right) \right) + \dfrac{1}{2} \left( \dot{\beta} \left( t \right) + \dfrac{\alpha \beta \left( t \right)}{t} \right) M \left( x \left(t \right) \right) = 0,
\end{equation}
with initial time $t_{0} >0$,  $\alpha \geq 2$,  $\beta:[t_{0}, +\infty) \rightarrow \sR_{++}$ a continuous differentiable and nondecreasing function, and initial conditions $x ( t_{0}) = x_{0} \in \sH$ and $\dot{x} \left( t_{0} \right) = \dot{x}_{0} \in \sH$.

It has been shown that if
\begin{equation*}
\sup_{t \geq t_{0}} \dfrac{t \dot{\beta} \left( t \right)}{\beta \left( t \right)} \leq \alpha - 2,
\end{equation*}
then 
$$\norm{\dot{x} \left( t \right)} = \bO \left(\dfrac{1}{t} \right) \quad \mbox{and} \quad \norm{M \left( x \left( t \right) \right)} = \bO \left(\dfrac{1}{t \beta \left( t \right)}\right) \quad \mbox{as} \ t \rightarrow +\infty,$$
and if
\begin{equation*}
\sup_{t \geq t_{0}} \dfrac{t \dot{\beta} \left( t \right)}{\beta \left( t \right)} < \alpha - 2,
\end{equation*}
then
$$\norm{\dot{x} \left( t \right)} = o \left(\dfrac{1}{t} \right) \quad \mbox{and} \quad \norm{M \left( x \left( t \right) \right)} = o \left(\dfrac{1}{t \beta \left( t \right)}\right) \quad \mbox{as} \ t \rightarrow +\infty,$$
and $x(t)$ converges weakly to a zero of $M$ as $t \rightarrow +\infty$. Explicit and implicit numerical algorithms,  introduced as discrete time counterparts of \eqref{ds:fast-OGDA:intro} and sharing its convergence properties have also been provided in \cite{Bot-Csetnek-Nguyen, Bot-Nguyen}.  An extension with a more general damping term has been studied in \cite{Bot-Hulett-Nguyen}.

\subsection{Continuous models}

In this paper, we consider the Tikhonov regularization of the dynamical system \eqref{ds:simp}
\begin{equation}
\label{ds:Tikhonov:intro}
	\dot{x} \left( t \right) + M \left( x \left( t \right) \right) + \varepsilon \left( t \right) x \left( t \right) = 0,
\end{equation}
with initial condition $x \left( t_{0} \right) \coloneq x_{0} \in \sH$ and continuously differentiable regularization function $\varepsilon \colon \left[ t_{0} , + \infty \right) \to \sR_{++}$ satisfying
\begin{equation*}
\lim_{t \to + \infty} \varepsilon \left( t \right) = 0
\quad \textrm{ and } \quad
\int_{t_{0}}^{+ \infty} \varepsilon \left( t \right) dt = + \infty.
\end{equation*}

It is well-known (\cite{Attouch-Cominetti,Cominetti-Peypouquet-Sorin}) that the Tikhonov regularization ensures that the trajectory generated by the system converges strongly to the element of minimum norm in $\zer M$.  We will demonstrate that, in addition,  rates of convergence in norm for the trajectory's velocity and the operator along the trajectory can be derived in terms of $\tcr{\varepsilon(\cdot)}$.  In some particular cases, these rates of convergence can be can be as fast as $\bO(\frac{1}{t})$ as $t \rightarrow +\infty$.  In this way, we emphasize a surprising acceleration feature of the Tikhonov regularization.

Further,  by using time rescaling techniques,  we translate the strong convergence and the convergence rate statements to the more general dynamic
\begin{equation}
	\label{ds:anchor:intro}
\dot{x} \left( t \right) + \beta \left( t\right) M \left( x \left( t \right) \right) + \delta \left( t \right) (x \left( t \right) -v) = 0 ,
\end{equation}
with initial time $t_{0} \geq 0$,  $\beta, \delta \colon \left[ t_{0} , + \infty \right) \to \sR_{++}$ continuously differentiable functions satifying
\begin{equation*}
\lim_{t \to + \infty} \dfrac{\delta \left( t \right)}{\beta \left( t \right)} = 0 , \quad
\int_{t_{0}}^{+ \infty} \delta \left( t \right) dt = + \infty ,
\quad \textrm{ and } \quad
\int_{t_{0}}^{+ \infty} \beta \left( t \right) dt = + \infty .
\end{equation*}
initial conditions $x \left( t_{0} \right) \coloneq x_{0} \in \sH$ and anchor point $v \in \sH$.  This dynamics,  called the continuous model of the Halpern method \cite{Halpern},  was  studied by Suh, Park, and Ryu in \cite{Suh-Park-Ryu} in finite-dimensional spaces, in the case $\beta(\cdot) \equiv 1$ and $\varepsilon \left( t \right) \coloneq \frac{\alpha}{t^{q}}$ for every $t \geq t_{0} >0$, where $\alpha > 0$ and $0 < q \leq 1$.  By relying on a very intricate Lyapunov analysis, the convergence of the trajectory and convergence rate for the norm of the operator along the trajectory has been provided. These results turn out to be particular cases of our more general approach,  which is fundamentally different.  In addition, our approach provides a rate of convergence for the velocity.  In light of the above considerations, it is evident that the Halpern method is nothing else than a particular instance of the classical Tikhonov regularization approach. The benefit of the time rescaling technique (\cite{Attouch-Bot-Nguyen,  Attouch-Bot-Nguyen:closed-loop}) is that there is no need to develop a Lyapunov analysis for every single system,  as it allows retrieving the convergence statements from the previously known results of other related dynamics.

Finally,  we show the connection between a particular first-order dynamical system with Tikhonov regularization and anchor point of the type \eqref{ds:anchor:intro} and the second order dynamical system
\begin{equation}	
\label{ds:second-order:intro}
\ddot{x} \left( t \right) + \dfrac{\alpha}{t} \dot{x} \left( t \right) + \beta \left( t \right) \dfrac{d}{dt} M \left( x \left( t \right) \right) + \left( \dot{\beta} \left( t \right) + \dfrac{\beta \left( t \right)}{t} \right) M \left( x \left(t \right) \right) = 0,
\end{equation}
with initial time $t_{0} >0$,  $\alpha \geq 2$,  $\beta:[t_{0}, +\infty) \rightarrow \sR_{++}$ a continuously differentiable function,  and initial conditions $x ( t_{0}) = x_{0} \in \sH$ and $\dot{x} \left( t_{0} \right) = \dot{x}_{0} \in \sH$. The dynamics \eqref{ds:second-order:intro} inherits from the first order dynamical system with Tikhonov regularization the strong convergence property of the trajectory and the fast rates of convergence in norm for the trajectory's velocity and the operator along the trajectory.  In addition,  it coincides with \eqref{ds:fast-OGDA:intro} for $\beta(t) = \beta_0 t^{\alpha-2}$ with $\alpha \geq 2$ and $\beta_0 >0$, and thus reveals the strong convergence of the trajectory for Fast OGDA in a setting where, according to \cite{Bot-Csetnek-Nguyen},  nothing could be said about the behaviour of the trajectory.

\subsection{Fixed point problems and discretization}

In the case when $M = \Id-T$,  for $T : \sH \rightarrow \sH$ a nonexpansve ($1$-Lipschitz continuous) operator, the explicit discretization of \eqref{ds:Tikhonov:intro} with anchor point $v \in \sH$ (or of the dynamic \eqref{ds:anchor:intro} in case $\beta \equiv 1$) leads to the Halpern fixed point iteration \cite{Halpern}.  By performing a convergence analysis which can be seen as a the discrete counterpart of the one provided for the continuous time system, we derive two different regimes for the regularization sequence,  both of which guarantee that the generated sequence of iterates strongly converges to the fixed point of $T$ (or the zero of $M$) that is closest to $v$. In addition, we provide a general theortical framework that allows to derive convergence rates for the vanishing of the discrete velocity and of the fixed point residual. Again, for some instances of the regularization sequence, we are able to derive concrete convergence rates that actually match those in continuous time.  Extensions to fixed point problems governed by averaged operators are given, which allow us to improve some statements concerning the asymptotic properties of some algorithms from the literature (\cite{Bot-Csetnek-Meier}, \cite{Kim}).

\section{First-order dynamical system with Tikhonov regularization}\label{sec2}

The starting point of our study is a first-order monotone operator flow with Tikhonov regularization.

\begin{mdframed}
Let $t_{0} \geq 0$ and $\varepsilon \colon \left[ t_{0} , + \infty \right) \to \sR_{++}$ be given.
We consider on $\left[ t_{0} , + \infty \right)$ the following dynamical system
\begin{equation}
\label{ds:Tikhonov}
\dot{x} \left( t \right) + M \left( x \left( t \right) \right) + \varepsilon \left( t \right) x \left( t \right) = 0 ,
\end{equation}
with initial condition $x \left( t_{0} \right) \coloneq x_{0} \in \sH$.
\end{mdframed}

For the \emph{regularization function} $\varepsilon(\cdot)$, the following conditions are imposed. These conditions are known to be necessary for the strong convergence of the trajectory to the minimum norm solution, as discussed in \cite{Attouch-Cominetti,Cominetti-Peypouquet-Sorin}.

\begin{mdframed}
The regularization function $\varepsilon \colon \left[ t_{0} , + \infty \right) \to \sR_{++}$ is required to be continuously differentiable and to satisfy
\begin{equation}
\label{cond:eps}
\lim_{t \to + \infty} \varepsilon \left( t \right) = 0
\quad \textrm{ and } \quad
\int_{t_{0}}^{+ \infty} \varepsilon \left( t \right) dt = + \infty.
\end{equation}
The first-order dynamical system is assumed to admit a unique strong global solution $x : [t_{0}, +\infty) \rightarrow \sH$, meaning that it is absolutely continuous on every compact interval $[t_{0},T]$ (see \cite[Appendix]{Brezis}) and it satisfies \eqref{ds:Tikhonov} almost everywhere, with the property that $t \mapsto M(x(t))$ is absolutely continuous on every compact interval $[t_{0},T]$.
\end{mdframed}

\begin{remark}\label{remark21}
\begin{enumerate}
\item 
The first condition in \eqref{cond:eps} is essential and inherently expected, as we anticipate the trajectory to converge towards a zero of the operator $M$. If we were to assume that $\varepsilon \left( t \right) \equiv \varepsilon_{0} > 0$, then the trajectory $t \mapsto x(t)$ would indeed exhibit strong convergence, albeit towards the unique zero of the strongly monotone operator $M + \varepsilon_{0} \Id$.
\item 
\label{remark21:M0}
In case $M \equiv 0$, the equation
\begin{equation*}
\dot{x} \left( t \right) + \varepsilon \left( t \right) x \left( t \right) = 0
\end{equation*}
with initial condition $x \left( t_{0} \right) \coloneq x_{0} \in \sH$, has as its unique solution
\begin{equation*}
x \left( t \right) = x_{0} \exp \left( - \int_{t_{0}}^{t} \varepsilon \left( u \right) du \right) .
\end{equation*}
We observe that $x(t)$ converges to $0$ as $t \rightarrow +\infty$, representing the element of $\zer M$ with the minimum norm, if and only if $\int_{t_{0}}^{+\infty} \varepsilon(t) dt = +\infty$. This property holds irrespective of the initial data $x_{0}$ and underscores the indispensability of the second condition in \eqref{cond:eps}.
\item
If $M \colon \sH \rightarrow \sH$ is Lipschitz continuous, then  (see, for instance, \cite[Proposition 6.2.1]{Haraux}
) \eqref{ds:Tikhonov} has for every $x_0 \in \sH$ a unique strong global solution $x  \colon \left[ t_{0} , + \infty \right) \rightarrow \sH$. In this case, $t \mapsto M(x(t))$ is absolutely continuous on every compact interval $[t_{0},T]$ and therefore almost everywhere differentiable on $[t_{0},+\infty)$.
\end{enumerate}
\end{remark}

\subsection{Convergence analysis and rates of convergence}\label{subsec21}

Following Remark \ref{remark21} \ref{remark21:M0}, for $\varepsilon \colon \left[ t_{0} , + \infty \right) \to \sR_{++}$, we define 
\begin{equation}
\label{defi:gamma}
\gamma_{\varepsilon} \colon \left[ t_{0} , + \infty \right) \to \sR_{++}, \quad \gamma_{\varepsilon} \left( t \right) \coloneq \exp \left( \int_{t_{0}}^{t} \varepsilon \left( u \right) du \right),
\end{equation}
so that for every $t \geq t_{0}$ it holds
\begin{align}
\dot{\gamma}_{\varepsilon} \left( t \right) 	& = \varepsilon \left( t \right) \gamma_{\varepsilon} \left( t \right) , \label{gamma:der} \\
\ddot{\gamma}_{\varepsilon} \left( t \right) 	& = \left[ \dot{\varepsilon} \left( t \right) + \varepsilon^{2} \left( t \right) \right] \gamma_{\varepsilon} \left( t \right) . \label{gamma:second}
\end{align}
This function will naturally appear several times in the analysis. 
The conditions \eqref{cond:eps} imposed on $\varepsilon$ ensure that $\gamma_{\varepsilon}$ is strictly increasing, thus $\gamma_{\varepsilon} \left( t \right) \geq \gamma_{\varepsilon} \left( t_{0} \right) = 1$ for every $t \geq t_{0}$, and that $\lim_{t \to + \infty} \gamma_{\varepsilon} \left( t \right) = + \infty$.

Given a trajectory solution $x \colon \left[ t_{0} , + \infty \right) \to \sH$ of \eqref{ds:Tikhonov} and $\xi \in \zer M$, we define the functions
\begin{align}\label{defi:phi}
\varphi_{\xi} \colon \left[ t_{0} , + \infty \right) \to \sR_{+}, \quad \varphi_{\xi} \left( t \right) \coloneq \dfrac{1}{2} \norm{x \left( t \right) - \xi}^{2}, 
\end{align}
and 
\begin{align}\label{defi:g}
\psi \colon \left[ t_{0} , + \infty \right) \to \sR_{+}, \quad \psi \left( t \right)  \coloneq  \dfrac{1}{2} \norm{M \left( x \left( t \right) \right) + \varepsilon \left( t \right) x \left( t \right)}^{2}.
\end{align}
These two objects are pivotal in our convergence analysis. Let us begin by establishing some initial estimates for them. In some of these estimates, we will utilize the quantity
\begin{equation}
\label{defi:D0}
D_{0} \left( x_{0},  v,  \xi \right) \coloneq \max \seq{\norm{x_{0} - \xi} , \norm{v - \xi}},
\end{equation}
which is defined for a triple $\left( x_{0} , v , \xi \right) \in \sH \times \sH \times \sH$.

\begin{lemma}
\label{lem:bnd}
Let $x \colon \left[ t_{0} , + \infty \right) \to \sH$ be a trajectory solution of the first-order dynamical system with Tikhonov regularization \eqref{ds:Tikhonov} and $\xi \in \zer M$. Then, the following statements are true:
\begin{enumerate}
\item 	
\label{lem:bnd:dphi}
for almost every $t \geq t_{0}$, it holds
\begin{equation}
\label{d:dphi}
\dot{\varphi_{\xi}} \left( t \right) + 2\varepsilon \left( t \right) \varphi_{\xi} \left( t \right) \leq - \varepsilon \left( t \right) \scal{x \left( t \right) - \xi , \xi} \leq \varepsilon \left( t \right) \varphi_{\xi} \left( t \right) +  \dfrac{\varepsilon \left( t \right)}{2} \norm{\xi}^{2};
\end{equation}

\item 
\label{lem:bnd:traj-bnd}
the trajectory is bounded, in particular, for every $t \geq t_{0}$, it holds
\begin{equation}
\label{Tikh:traj-bnd}
\norm{x \left( t \right) - \xi} \leq D_{0} \left( x_{0} , 0 , \xi \right) 
\quad \textrm{ and } \quad 
\norm{x \left( t \right)} \leq 2D_{0} \left( x_{0} , 0 , \xi \right) ;
\end{equation}

\item 
for almost every $t \geq t_{0}$, it holds
\begin{equation}
\label{Tikh:est}
\dot{\psi} \left( t \right) + 2 \varepsilon \left( t \right) \psi \left( t \right) \leq  - \dfrac{1}{2} \dot{\varepsilon} \left( t \right) \left( \dfrac{d}{dt} \norm{x \left( t \right)}^{2} \right) = - \dot{\varepsilon}(t) \langle \dot x(t), x(t) \rangle;
\end{equation}

\item 
for almost every $t \geq t_{0}$, it holds
\begin{equation}
\label{Tikh:crit}
\dfrac{d}{dt} \big( \gamma_{\varepsilon}^{2} \left( t \right) \psi \left( t \right) \big) 
= \gamma_{\varepsilon}^{2} \left( t \right) \dot{\psi} \left( t \right) + 2 \varepsilon \left( t \right) \gamma_{\varepsilon}^{2} \left( t \right) \psi \left( t \right) 
\leq - \dfrac{1}{2} \dot{\varepsilon} \left( t \right) \gamma_{\varepsilon}^{2} \left( t \right) \left( \dfrac{d}{dt} \norm{x \left( t \right)}^{2} \right) = - \dot{\varepsilon}(t) \gamma_{\varepsilon}^{2} \left( t \right) \langle \dot x(t), x(t) \rangle.
\end{equation}
\end{enumerate}
\end{lemma}

\begin{proof}
\begin{enumerate}
\item 
For almost every $t \geq t_{0}$, when we take the time derivative of $\varphi_{\xi}$, we obtain
\begin{align}	
\dot{\varphi_{\xi}} \left( t \right) = \scal{x \left( t \right) - \xi , \dot{x} \left( t \right)} 
& = - \scal{x \left( t \right) - \xi , M \left( x \left( t \right) \right)} - \varepsilon \left( t \right) \scal{x \left( t \right) - \xi , x \left( t \right)} . \nonumber \\
& = - \scal{x \left( t \right) - \xi , M \left( x \left( t \right) \right)} - \varepsilon \left( t \right) \scal{x \left( t \right) - \xi , \xi} - 2 \varepsilon \left( t \right) \varphi_{\xi} \left( t \right) \nonumber \\
& \leq - \varepsilon \left( t \right) \scal{x \left( t \right) - \xi , \xi} - 2 \varepsilon \left( t \right) \varphi_{\xi} \left( t \right) , \label{d:phi}
\end{align}
where in the last inequality we used that $\scal{x \left( t \right) - \xi , M \left( x \left( t \right) \right)} \geq 0$, which holds due to the monotonicity of $M$. The second inequality in \eqref{d:dphi} follows from
\begin{align*}
- \varepsilon \left( t \right) \scal{x \left( t \right) - \xi , \xi} = \dfrac{\varepsilon \left( t \right)}{2} \left( \norm{x \left( t \right) - \xi}^{2} + \norm{\xi}^{2} - \norm{x \left( t \right)}^{2} \right) \leq \varepsilon \left( t \right) \varphi_{\xi} \left( t \right) + \dfrac{\varepsilon \left( t \right)}{2} \norm{\xi}^{2}.
\end{align*}

\item 
By invoking \eqref{d:dphi}, it yields for almost every $t \geq t_{0}$
\begin{equation*}
\dfrac{d}{dt} \big( \gamma_{\varepsilon} \left( t \right) \varphi_{\xi} \left( t \right) \big) = \gamma_{\varepsilon} \left( t \right) \dot{\varphi_{\xi}} \left( t \right) + \varepsilon \left( t \right) \gamma_{\varepsilon} \left( t \right) \varphi_{\xi} \left( t \right) \leq \dfrac{\varepsilon \left( t \right) \gamma_{\varepsilon} \left( t \right)}{2} \norm{\xi}^{2}  = \dot{\gamma_{\varepsilon}} \left( t \right) \dfrac{\norm{\xi}^{2}}{2} .
\end{equation*}
After integration from $t_{0}$ to $t$, we obtain for every $t \geq t_{0}$
\begin{equation*}
\gamma_{\varepsilon} \left( t \right) \varphi_{\xi} \left( t \right) - \gamma_{\varepsilon} \left( t_{0} \right) \varphi_{\xi} \left( t_{0} \right) \leq \left( \gamma_{\varepsilon} \left( t \right) - \gamma_{\varepsilon} \left( t_{0} \right) \right) \dfrac{\norm{\xi}^{2}}{2},
\end{equation*}
which is equivalent to
\begin{equation*}
\norm{x \left( t \right) - \xi}^{2} \leq \dfrac{1}{\gamma_{\varepsilon} \left( t \right)} \norm{x_{0} - \xi}^{2} + \left( 1 - \dfrac{1}{\gamma_{\varepsilon} \left( t \right)} \right) \norm{\xi}^{2}.
\end{equation*}
The first statement in \eqref{Tikh:traj-bnd} is a direct consequence, as $\gamma_{\varepsilon} \left( t \right) \geq \gamma_{\varepsilon} \left( t_{0} \right) = 1$ for every $t \geq t_{0}$.
The second statement in \eqref{Tikh:traj-bnd} follows by the triangle inequality.

\item 
Differentiating the function $\psi$ with respect to $t$ gives for almost every $t \geq t_{0}$
\begin{align}
\dot{\psi} \left( t \right)
& = \scal{M \left( x \left( t \right) \right) + \varepsilon \left( t \right) x \left( t \right) , \dfrac{d}{dt} \big( M \left( x \left( t \right) \right) + \varepsilon \left( t \right) x \left( t \right) \big)} \nonumber \\
& = - \scal{\dot{x} \left( t \right) , \dfrac{d}{dt} M \left( x \left( t \right) \right) + \dot{\varepsilon} \left( t \right) x \left( t \right) + \varepsilon \left( t \right) \dot{x} \left( t \right)} \nonumber \\
& = - \scal{\dot{x} \left( t \right) , \dfrac{d}{dt} M \left( x \left( t \right) \right)} - \dot{\varepsilon} \left( t \right) \scal{\dot{x} \left( t \right) , x \left( t \right)} - \varepsilon \left( t \right) \norm{\dot{x} \left( t \right)}^{2} \nonumber \\
& \leq - \dot{\varepsilon} \left( t \right) \scal{\dot{x} \left( t \right) , x \left( t \right)} - \varepsilon \left( t \right) \norm{\dot{x} \left( t \right)}^{2} \label{Tikh:pre-est} \\
& = - \dfrac{1}{2} \dot{\varepsilon} \left( t \right) \left( \dfrac{d}{dt} \norm{x \left( t \right)}^{2} \right) - 2 \varepsilon \left( t \right) \psi \left( t \right), \nonumber
\end{align}
where inequality \eqref{Tikh:pre-est} is a consequence of the monotonicity of $M$.

\item The statement follows from \eqref{Tikh:est}, by using that  $\dot \gamma_{\varepsilon}(t) = \gamma_{\varepsilon}(t) \varepsilon (t)$ for every $t \geq t_{0}$.
\qedhere
\end{enumerate}
\end{proof}

Studying the function $\varphi$ not only reveals the boundedness of the trajectory,  but also allows us to deduce its strong convergence, provided that $\norm{M \left( x (t) \right)} \to 0$ as $t \to + \infty$.

\begin{proposition}\label{prop:Tikh:converge}
Let $x \colon \left[ t_{0} , + \infty \right) \to \sH$ be the trajectory solution of the first-order dynamical system with Tikhonov regularization \eqref{ds:Tikhonov}. If $\|M( x ( t))\| \to 0$ as $t \to + \infty$, then $x(t)$ converges strongly to the minimum norm solution $\xi_{*} \coloneq \proj_{\zer M} \left( 0 \right)$ as $t \to + \infty$.
\end{proposition}

\begin{proof}
Invoking inequality \eqref{d:dphi} for $\xi_{*} \coloneq \proj_{\zer M} \left( 0 \right)$, we have for almost every $t \geq t_{0}$
\begin{equation*}
\dot{\varphi}_{\xi_{*}} \left( t \right) + \varepsilon \left( t \right) \varphi_{\xi_{*}} \left( t \right) \leq \varepsilon \left( t \right) \scal{x \left( t \right) - \xi_{*} , 0 - \xi_{*}} .
\end{equation*}
By using Lemma \ref{lem:limsup} and the boundedness of $x(\cdot)$, from here we obtain that 
\begin{equation}\label{ineq:strongconv}
\limsup_{t \rightarrow +\infty} \varphi_{\xi_{*}} \left( t \right) \leq \limsup_{t \rightarrow +\infty} \scal{x \left( t \right) - \xi_{*} , 0 - \xi_{*}}.
\end{equation}

We are going to show that, under the additional assumption $\|M \left( x \left( t \right) \right)\| \to 0$ as $t \to + \infty$,  every weak cluster point of $x(\cdot)$ belongs to $\zer M$. Such cluster point always exists thanks to the boundedness of the trajectory.  If $\widehat{x}$ is arbitrary weak cluster point of $x(\cdot)$, then there exists a sequence $\seq{x \left( t_{k} \right)}_{k \geq 0}$ such that 
$x \left( t_{k} \right)$ converges weakly to $\widehat{x}$ as $k \to + \infty$. On the other hand, $M \left( x \left( t_{k} \right) \right)$ converges strongly to $0$ as $k \to + \infty$. Since $M$ is monotone and continuous, it is maximally monotone (see, for instance, \cite[Corollary 20.28]{Bauschke-Combettes:book}). Therefore, its graph is sequentially closed in the weak $\times$ strong topology of $\sH \times \sH$ (see \cite[Proposition 20.38]{Bauschke-Combettes:book}). This means that $M \left( \widehat{x} \right) = 0$, which enables us to conclude from \eqref{ineq:strongconv}, by using the variational characterization of the projection, that $\limsup_{t \rightarrow +\infty} \varphi_{\xi_{*}} \left( t \right) \leq 0$, which is nothing else than $x(t) \rightarrow \xi_{*}$ as $t \rightarrow +\infty$.
\end{proof}

In the following,  we will introduce general conditions concerning the regularization function that enable the derivation of convergence rates for the operator norm and the velocity.  This is where the second function $\psi$ comes into play.

Let us begin with an observation. For every $t \geq t_{0}$, if
\begin{equation}
\label{gamma:eq}
\dot{\varepsilon} \left( t \right) + \varepsilon^{2} \left( t \right) = 0,
\end{equation}
then, according to \eqref{gamma:second}, we must have
\begin{equation*}
\ddot{\gamma}_{\varepsilon} \left( t \right) = \left[ \dot{\varepsilon} \left( t \right) + \varepsilon^{2} \left( t \right) \right] \gamma_{\varepsilon} \left( t \right) = 0 .
\end{equation*}
This leads to the following equation
\begin{equation*}
\dot{\gamma}_{\varepsilon} \left( t \right) = \varepsilon \left( t \right) \gamma_{\varepsilon} \left( t \right) = \varepsilon \left( t_{0} \right) .
\end{equation*}
In this case, we can derive from \eqref{Tikh:crit} that for almost every $t \geq t_0$
\begin{equation}
\label{Tikh:crit:est}
\dfrac{d}{dt} \left( \gamma_{\varepsilon}^{2} \left( t \right) \psi \left( t \right) \right) 
\leq - \dfrac{1}{2} \dot{\varepsilon} \left( t \right) \gamma_{\varepsilon}^{2} \left( t \right) \left( \dfrac{d}{dt} \norm{x \left( t \right)}^{2} \right) 
= \dfrac{1}{2} \varepsilon^{2} \left( t \right) \gamma_{\varepsilon}^{2} \left( t \right) \left( \dfrac{d}{dt} \norm{x \left( t \right)}^{2} \right) 
= \dfrac{1}{2} \varepsilon^{2} \left( t_{0} \right) \left( \dfrac{d}{dt} \norm{x \left( t \right)}^{2} \right) .
\end{equation}

We record this special case here for $\varepsilon \left( t \right) \coloneq \frac{1}{t}$, which is obviously a solution of \eqref{gamma:eq}.  

\begin{theorem}[the case $\varepsilon \left( t \right) \coloneq \frac{1}{t}$]
\label{thm:Tikh:q=1:alpha=1}
Let $t_{0} > 0$, and $x \colon \left[ t_{0} , + \infty \right) \to \sH$ be the trajectory solution of the dynamical system
\begin{equation}
\label{ds:q=1:alpha=1}
\dot{x} \left( t \right) + M \left( x \left( t \right) \right) + \dfrac{1}{t} x \left( t \right) = 0,
\end{equation}
with initial condition $x \left( t_{0} \right) \coloneq x_{0} \in \sH$.  Then, the following statements are true:
\begin{enumerate}	
\item 
as $t \to + \infty$ it holds
\begin{align*}
\norm{\dot{x} \left( t \right)} = \bO \left( \dfrac{1}{t} \right) 
\quad \textrm{ and } \quad
\norm{M \left( x \left( t \right) \right)} = \bO \left( \dfrac{1}{t} \right);
\end{align*}
\item the trajectory $x(t)$ converges strongly to the minimum norm solution $\xi_{*} \coloneq \proj_{\zer M} \left( 0 \right)$ as $t \to + \infty$.
\end{enumerate}
\end{theorem}
\begin{proof}
Let $t \geq t_{0}$ and $\xi \in \zer M$. By integration from $t_{0}$ to $t$, we obtain
\begin{equation*}
\gamma_{\varepsilon}^{2} \left( t \right) \psi \left( t \right) - \psi \left( t_{0} \right) \leq \dfrac{1}{2} \varepsilon^{2} \left( t_{0} \right) \left( \norm{x \left( t \right)}^{2} - \norm{x_{0}}^{2} \right),
\end{equation*}
which implies
\begin{align*}
\gamma_{\varepsilon}^{2}(t) \norm{M \left( x \left( t \right) \right) + \varepsilon \left( t \right) x \left( t \right)}^{2} 
& \leq \norm{M \left( x_{0} \right) + \varepsilon \left( t_{0} \right) x_{0}}^{2} + \varepsilon^{2} \left( t_{0} \right) \left( \norm{x \left( t \right)}^{2} - \norm{x_{0}}^{2} \right) \nonumber \\
& = \norm{M \left( x_{0} \right)}^{2} + 2 \varepsilon \left( t_{0} \right) \scal{M \left( x_{0} \right),x_{0}} + \varepsilon^{2} \left( t_{0} \right) \norm{x \left( t \right)}^{2} \nonumber \\
& \leq \norm{M \left( x_{0} \right)}^{2} + 2 \varepsilon \left( t_{0} \right) \norm{M \left( x_{0} \right)} \norm{x_{0}} + \varepsilon^{2} \left( t_{0} \right) \norm{x \left( t \right)}^{2} \nonumber \\
& \leq \norm{M \left( x_{0} \right)}^{2} + 4 \varepsilon \left( t_{0} \right) \norm{M \left( x_{0} \right)} D_{0} \left( x_{0} , 0 , \xi \right) + 4 \varepsilon^{2} \left( t_{0} \right) D_{0}^{2} \left( x_{0} , 0 , \xi \right) \nonumber \\
& = \left( \norm{M \left( x_{0} \right)} + 2 \varepsilon \left( t_{0} \right) D_{0} \left( x_{0} , 0 , \xi \right) \right) ^{2} .
\end{align*}
This proves statement (i). The strong convergence of the trajectory is a consequence of Proposition \ref{prop:Tikh:converge}.
\end{proof}

\begin{remark}
If one wants to consider dynamics with starting time $t_0=0$, one can choose $\varepsilon \left( t \right) \coloneq \frac{1}{t + c}$, for some constant $c > 0$, which are also solutions of \eqref{gamma:eq}.
\end{remark}

We now turn our attention to the general case of a regularization function.

\begin{proposition}
\label{thm:Tikh:pre}
Let $x \colon \left[ t_{0} , + \infty \right) \to \sH$ be the trajectory solution of the first-order dynamical system with Tikhonov regularization \eqref{ds:Tikhonov}.  Then, it holds as $t \to +\infty$
\begin{equation*}
\norm{\dot{x} \left( t \right)} = \bO \left( \dfrac{1}{\gamma_{\varepsilon} \left( t \right)} \right) + \bO \left( \dfrac{1}{\gamma_{\varepsilon} \left( t \right)} \int_{t_{0}}^{t} \abs{\dot{\varepsilon} \left( u \right)} \gamma_{\varepsilon} \left( u \right) du \right)
\end{equation*}
and, consequently,
\begin{equation*}
\norm{M \left( x \left( t \right) \right)}  = \bO \left( \dfrac{1}{\gamma_{\varepsilon} \left( t \right)} \right) + \bO \left( \dfrac{1}{\gamma_{\varepsilon} \left( t \right)} \int_{t_{0}}^{t} \abs{\dot{\varepsilon} \left( u \right)} \gamma_{\varepsilon} \left( u \right) du \right) + \bO \left( \varepsilon \left( t \right) \right).
\end{equation*}
\end{proposition}
\begin{proof}
Let $t \geq t_{0}$ and $\xi \in \zer M$. By integrating \eqref{Tikh:crit} from $t_{0}$ to $t$, we obtain
\begin{align*}
\gamma_{\varepsilon}^{2} \left( t \right) \psi \left( t \right) -\gamma_{\varepsilon}^{2} \left( t_{0} \right) \psi \left( t_{0} \right) 
& \leq \int_{t_{0}}^{t} \left[ - \dot{\varepsilon} \left( u \right) \right] \gamma_{\varepsilon}^{2} \left( u \right) \langle \dot{x} \left( u \right) , x \left( u \right) \rangle du \nonumber \\
& \leq \int_{t_{0}}^{t} {\abs{\dot{\varepsilon} \left( u \right)}} \gamma_{\varepsilon}^{2} \left( u \right) \norm{\dot{x} \left( u \right)} \norm{x \left( u \right)} du \nonumber \\
& \leq 2D_{0} \left( x_{0} , 0 , \xi \right) \int_{t_{0}}^{t} {\abs{\dot{\varepsilon} \left( u \right)}} \gamma_{\varepsilon}^{2} \left( u \right) \norm{\dot{x} \left( u \right)} du,
\end{align*}	
where the second inequality follows from the Cauchy-Schwarz inequality.
Then,
\begin{align*}
\dfrac{1}{2} \gamma_{\varepsilon}^{2} \left( t \right) \norm{M \left( x \left( t \right) \right) + \varepsilon \left( t \right) x \left( t \right)}^{2} 
& \leq \dfrac{1}{2} \norm{M \left( x_{0} \right) + \varepsilon \left( t_{0} \right) x_{0}}^{2} \nonumber \\
& \quad + 2D_{0} \left( x_{0} , 0 , \xi \right) \int_{t_{0}}^{t} {\abs{\dot{\varepsilon} \left( u \right)}} \gamma_{\varepsilon}^{2} \left( u \right) \norm{M \left( x \left( u \right) \right) + \varepsilon \left( u \right) x \left( u \right)} du .
\end{align*}
According to \cref{lem:Ou-Iang}, we obtain from here that
\begin{equation*}
\gamma_{\varepsilon} \left( t \right) \norm{M \left( x \left( t \right) \right) + \varepsilon \left( t \right) x \left( t \right)} \leq \norm{M \left( x_{0} \right) + \varepsilon \left( t_{0} \right) x_{0}} + 2D_{0} \left( x_{0} , 0 , \xi \right) \int_{t_{0}}^{t} {\abs{\dot{\varepsilon} \left( u \right)}} \gamma_{\varepsilon} \left( u \right) du ,
\end{equation*}
which, after similar arguments as in the previous theorem, leads to
\begin{align*}
\norm{M \left( x \left( t \right) \right) + \varepsilon \left( t \right) x \left( t \right)} 
& \leq \dfrac{1}{\gamma_{\varepsilon} \left( t \right)} \left( \norm{M \left( x_{0} \right) + \varepsilon \left( t_{0} \right) x_{0}} + 2D_{0} \left( x_{0} , 0 , \xi \right) \int_{t_{0}}^{t} {\abs{\dot{\varepsilon} \left( u \right)}} \gamma_{\varepsilon} \left( u \right) du \right) \nonumber \\
& \leq \dfrac{1}{\gamma_{\varepsilon} \left( t \right)} \left( \norm{M \left( x_{0} \right)} + 2D_{0} \left( x_{0} , 0 , \xi \right) \left( \varepsilon \left( t_{0} \right) + \int_{t_{0}}^{t} {\abs{\dot{\varepsilon} \left( u \right)}} \gamma_{\varepsilon} \left( u \right) du \right) \right)
\end{align*}
It is from here that the conclusion follows.
\end{proof}

In the following theorem, we present additional conditions on the regularization function that ensure $\norm{M \left( x(t) \right)} \to 0$ as $t \to +\infty$, leading to the strong convergence of $x(\cdot)$ to the solution of \eqref{intro:mono} with minimal norm.

\begin{theorem}
\label{thm:Tikh}
Let $x \colon \left[ t_{0} , + \infty \right) \to \sH$ be the trajectory solution of the first-order dynamical system with Tikhonov regularization \eqref{ds:Tikhonov}.   Regarding the regularization function $\varepsilon(\cdot)$, assume that, in addition to \eqref{cond:eps}, it satisfies either
\begin{equation}
\label{eps:BV}
\int_{t_{0}}^{+ \infty} \abs{\dot{\varepsilon} \left( t \right)} dt < + \infty.
\end{equation}
or
\begin{equation}
\label{eps:lim}
\lim\limits_{t \to + \infty} \dfrac{\abs{\dot{\varepsilon} \left( t \right)}}{\varepsilon \left( t \right)} = 0.
\end{equation}
Then, $\norm{M \left( x(t) \right)} \to 0$ as $t \to + \infty$, and therefore the trajectory $x(t)$ converges strongly to the minimum norm solution $\xi_{*} \coloneq \proj_{\zer M} \left( 0 \right)$ as $t \to + \infty$.
\end{theorem}
\begin{proof}
In the first scenario,  the view of \Cref{prop:Tikh:converge} and \Cref{thm:Tikh:pre},  we only need to guarantee that $\frac{1}{\gamma_{\varepsilon} \left( t \right)} \int_{t_{0}}^{t} \abs{\dot{\varepsilon} \left( u \right)} \gamma_{\varepsilon} \left( u \right) du \to 0$ as $t \to + \infty$.
This can be done via \eqref{eps:BV} in the virtue of \Cref{lem:eint}.

In the second scenario,  let $\xi \in \zer M$ and  $t_{1} \geq t_{0}$ such that $t \mapsto \frac{\abs{\dot{\varepsilon} \left( t \right)}}{\varepsilon(t)}$ is bounded on $[t_{1}, +\infty)$. According to \eqref{Tikh:est} and by using twice the Cauchy-Schwarz inequality and then Lemma \ref{lem:bnd} \ref{lem:bnd:traj-bnd}, we have for almost every $t \geq t_{1}$
\begin{align*}
\dot \psi(t) + \varepsilon(t) \psi(t) \leq & \ -\varepsilon(t) \psi(t) + 
\abs{\dot{\varepsilon} \left( t \right)} \norm{\dot{x} \left( t \right)} \norm{x \left( t \right)} = -\frac{1}{2}\varepsilon(t) \norm{\dot x(t)}^{2} +  \abs{\dot{\varepsilon} \left( t \right)} \norm{\dot{x} \left( t \right)} \norm{x \left( t \right)}\\
\leq & \ \frac{1}{2} \frac{\abs{\dot{\varepsilon} \left( t \right)}^2}{\varepsilon(t)}\norm{x \left( t \right)}^{2} \leq 2 D_{0}^2 \left( x_{0} , 0 , \xi \right) \frac{\abs{\dot{\varepsilon} \left( t \right)}^2}{\varepsilon(t)} = \varepsilon(t) \left(2 D_{0}^2 \left( x_{0} , 0 , \xi \right) \frac{\abs{\dot{\varepsilon} \left( t \right)}^2}{\varepsilon(t)^2} \right).
\end{align*}
By using \cref{lem:limsup}, we obtain $\lim_{t \to + \infty} \psi(t) = 0$.
\end{proof}

\begin{remark}\label{rem:strongconvergence}

\begin{enumerate}
\item 
Condition \eqref{eps:BV} was introduced by Cominetti, Peypouquet, and Sorin in \cite[Theorem 9]{Cominetti-Peypouquet-Sorin}. The technique used to prove $\lim_{t \to +\infty} \psi(t) = 0$ aligns with the approach used in this paper.

\item In \cite{Israel-Reich} (also see \cite[Proposition 5]{Cominetti-Peypouquet-Sorin}), the strong convergence of the trajectory to the minimal norm solution has been demonstrated under assumptions regarding the regularization function which, in addition to \eqref{cond:eps}, entail that ${\varepsilon(\cdot)}$ is decreasing and satisfy the condition
\begin{equation}
\label{eps:lim-2}
\lim_{t \to + \infty} \dfrac{\abs{\dot{\varepsilon} \left( t \right)}}{\varepsilon^{2} \left( t \right)} = 0.
\end{equation}
Condition \eqref{eps:lim} in Theorem \ref{thm:Tikh} represents a notable improvement in this context. Specifically, it accommodates regularization functions of the form $\varepsilon \left( t \right) = \frac{\alpha}{t^{q}}$, where $0 < q \leq 1$ and $\alpha > 0$. This condition allows for the critical choice $q \coloneq 1$, for which \eqref{eps:lim-2} is not fulfilled.
\end{enumerate}
\end{remark}

In the following, we impose additional conditions on the regularization function to obtain explicit convergence rates in terms of $\varepsilon(\cdot)$ and $\gamma_\varepsilon(\cdot)$.
In particular, we will show that as $t \to + \infty$ it holds
\begin{equation}
\label{Tikh:eps:rates:pre}
\norm{\dot{x} \left( t \right)} = 
\bO \left( \max \left\lbrace \varepsilon \left( t \right) , \dfrac{1}{\gamma_{\varepsilon} \left( t \right)} \right\rbrace \right) 
\quad \textrm{ and } \quad
\norm{M \left( x \left( t \right) \right)} = 
\bO \left( \max \left\lbrace \varepsilon \left( t \right) , \dfrac{1}{\gamma_{\varepsilon} \left( t \right)} \right\rbrace \right) .
\end{equation}
For the sake of brevity, we denote
\begin{equation*}
\varepsilon_{i} \coloneq \inf_{t \geq t_{0}} \dfrac{d}{dt} \left[ \dfrac{1}{\varepsilon \left( t \right)} \right] 
= \inf_{t \geq t_{0}} \frac{- \dot{\varepsilon} \left( t \right)}{\varepsilon^{2} \left( t \right)}
\quad \textrm{ and } \quad
\varepsilon_{s} \coloneq \sup_{t \geq t_{0}} \dfrac{d}{dt} \left[ \dfrac{1}{\varepsilon \left( t \right)} \right] 
= \sup_{t \geq t_{0}} \frac{- \dot{\varepsilon} \left( t \right)}{\varepsilon^{2} \left( t \right)} .
\end{equation*}
One may notice that in the critical case  considered in Theorem \ref{thm:Tikh:q=1:alpha=1}
\begin{equation*}
\varepsilon_{i} = \varepsilon_{s} = 1 ,
\end{equation*}
and it is indeed the case that the two quantities within the  Big-$\bO$ rates in \eqref{Tikh:eps:rates:pre} are equal.

\begin{theorem}
\label{thm:Tikh:eps}
Let $x \colon \left[ t_{0} , + \infty \right) \to \sH$ be the trajectory solution of the first-order dynamical system with Tikhonov regularization \eqref{ds:Tikhonov}.
Suppose that the regularization function $\varepsilon(\cdot)$, in addition to \eqref{cond:eps}, satisfies 
\begin{equation}
\label{eps:inf-sup}
1 < \varepsilon_{i}
\quad \textrm{ or } \quad
-1 < \varepsilon_{i} \leq \varepsilon_{s} < 1
\quad \textrm{ or } \quad
\varepsilon_{i} = \varepsilon_{s} = 1 .
\end{equation}
Then, the following statements are true:
\begin{enumerate}
\item 
\label{thm:Tikh:eps:rate}
as $t \to + \infty$ it holds
\begin{align}
\label{Tikh:eps:rates}
\norm{\dot{x} \left( t \right)} & = \begin{cases}
\bO \left( \dfrac{1}{\gamma_{\varepsilon} \left( t \right)} \right) & \textrm{ if } 1 < \varepsilon_{i}  \textrm{ or } \varepsilon_{i} = \varepsilon_{s} = 1 , \\
 \bO \left( \varepsilon \left( t \right) \right) 					& \textrm{ if } -1 < \varepsilon_{i} \leq \varepsilon_{s} < 1 ,
\end{cases}
\\
\label{Tikh:eps:rates1}
\norm{M \left( x \left( t \right) \right)} & = \begin{cases}
\bO \left( \dfrac{1}{\gamma_{\varepsilon} \left( t \right)} \right) & \textrm{ if } 1 < \varepsilon_{i} \textrm{ or } \varepsilon_{i} = \varepsilon_{s} = 1 , \\
\bO \left( \varepsilon \left( t \right) \right) 					& \textrm{ if } -1 < \varepsilon_{i} \leq \varepsilon_{s} < 1  ;
\end{cases} 
\end{align}

\item the trajectory $x(t)$ converges strongly to the minimum norm solution $\xi_{*} \coloneq \proj_{\zer M} \left( 0 \right)$ as $t \to + \infty$.
\end{enumerate}
\end{theorem}
\begin{proof}
For every $t \geq t_{0}$, integration by parts yields
\begin{align}
\label{Tikh:ineg-gen}
\int_{t_{0}}^{t} \left[ - \dot{\varepsilon} \left( u \right) \right] \gamma_{\varepsilon} \left( u \right) du 
& = - \varepsilon \left( u \right) \gamma_{\varepsilon} \left( u \right) \bigg\vert_{r=t_{0}}^{r=t} + \int_{t_{0}}^{t} \varepsilon \left( u \right) \dot{\gamma}_{\varepsilon} \left( u \right) du  \nonumber \\
& =  {\varepsilon \left( t_{0} \right) - \varepsilon \left( t \right) \gamma_{\varepsilon} \left( t \right)} + \int_{t_{0}}^{t} \varepsilon^{2} \left( u \right) \gamma_{\varepsilon} \left( u \right) du .
\end{align}	
We will use this estimate to derive an upper bound for $\int_{t_{0}}^{t} \abs{\dot{\varepsilon} \left( u \right)} \gamma_{\varepsilon} \left( u \right) du$.
Let us consider the two scenarios separately.
\begin{itemize}[wide]
\item 
\underline{The case $1 < \varepsilon_{i}$}.
For every $t \geq t_{0}$, we have
\begin{equation}
\label{Tikh:eps:cond:inf}
1 < \varepsilon_{i}  \leq \dfrac{- \dot{\varepsilon} \left( t \right)}{\varepsilon^{2} \left( t \right)}.
\end{equation}
{It} follows from \eqref{Tikh:ineg-gen} that for every $t \geq t_0$
\begin{align*}
\int_{t_{0}}^{t} \abs{\dot{\varepsilon} \left( u \right)} \gamma_{\varepsilon} \left( u \right) du 
= \int_{t_{0}}^{t} \left[ - \dot{\varepsilon} \left( u \right) \right] \gamma_{\varepsilon} \left( u \right) du
& \leq {\varepsilon \left( t_{0} \right) - \varepsilon \left( t \right) \gamma_{\varepsilon} \left( t \right)} + \dfrac{1}{\varepsilon_{i}} \int_{t_{0}}^{t} \left[ - \dot{\varepsilon} \left( u \right) \right] \gamma_{\varepsilon} \left( u \right) du \nonumber \\
& = \varepsilon \left( t_{0} \right) - \varepsilon \left( t \right) \gamma_{\varepsilon} \left( t \right) + \dfrac{1}{\varepsilon_{i}} \int_{t_{0}}^{t} \abs{\dot{\varepsilon} \left( u \right)} \gamma_{\varepsilon} \left( u \right) du ,
\end{align*}
which is equivalent to
\begin{equation}
\label{Tikh:eps:est:inf}	
\int_{t_{0}}^{t} \abs{\dot{\varepsilon} \left( u \right)} \gamma_{\varepsilon} \left( u \right) du
\leq \dfrac{\varepsilon_{i}}{\varepsilon_{i} - 1} \left(  \varepsilon \left( t_{0} \right) - \varepsilon \left( t \right) \gamma_{\varepsilon} \left( t \right) \right).
\end{equation}
Inequality \eqref{Tikh:eps:cond:inf} also implies $\ddot{\gamma_\varepsilon} (t) = \left[ \dot{\varepsilon} \left( t \right) + \varepsilon^{2} \left( t \right) \right] \gamma_{\varepsilon} \left( t \right) \leq 0$, hence $\dot{\gamma_\varepsilon} \left( t \right) = \varepsilon \left( t \right) \gamma_\varepsilon \left( t \right) \leq \dot{\gamma_\varepsilon} \left( t_{0} \right) = \varepsilon \left( t_{0} \right)$.
Combining this estimate with \eqref{Tikh:eps:est:inf} and Proposition \ref{thm:Tikh:pre} we get the rates given in the first case in  \eqref{Tikh:eps:rates} and \eqref{Tikh:eps:rates1}.

\item
\underline{The case {$-1 < \varepsilon_{i} \leq \varepsilon_{s} < 1$}}.
For every $t \geq t_{0}$
\begin{equation}
\label{Tikh:eps:cond:sup}
-1 < \varepsilon_{i} \leq \dfrac{- \dot{\varepsilon} \left( t \right)}{\varepsilon^{2} \left( t \right)} \leq {\varepsilon_{s} < 1} .
\end{equation}
We denote  $\varepsilon_{\max} := \max \left\lbrace \abs{\varepsilon_{i}} , \abs{\varepsilon_{s}} \right\rbrace < 1$, thus
\begin{equation*}
\abs{\dot{\varepsilon} \left( t \right)} \leq \varepsilon_{\max} \varepsilon^{2} \left( t \right) \quad \forall t \geq t_0.
\end{equation*}
Therefore, we can derive from \eqref{Tikh:ineg-gen} that for every $t \geq t_0$
\begin{align*}
\int_{t_{0}}^{t} \abs{\dot{\varepsilon} \left( u \right)} \gamma_{\varepsilon} \left( u \right) du
\geq \int_{t_{0}}^{t} \left[ - \dot{\varepsilon} \left( u \right) \right] \gamma_{\varepsilon} \left( u \right) du 
& \geq {\varepsilon \left( t_{0} \right) - \varepsilon \left( t \right) \gamma_{\varepsilon} \left( t \right)} + \dfrac{1}{\varepsilon_{\max}} \int_{t_{0}}^{t} \abs{\dot{\varepsilon} \left( u \right)} \gamma_{\varepsilon} \left( u \right) du .
\end{align*}	
which implies
\begin{equation}
\label{Tikh:eps:est:sup}
\int_{t_{0}}^{t} \abs{\dot{\varepsilon} \left( u \right)}  \gamma_{\varepsilon} \left( u \right) du
\leq \dfrac{\varepsilon_{\max}}{1 - \varepsilon_{\max}} \left( \varepsilon \left( t \right) \gamma_{\varepsilon} \left( t \right) - \varepsilon \left( t_{0} \right) \right) .
\end{equation}
Similar to the previous case,  \eqref{Tikh:eps:cond:sup} leads to $\ddot{\gamma_\varepsilon} (t) = \left[ \dot{\varepsilon} \left( t \right) + \varepsilon^{2} \left( t \right) \right] \gamma_{\varepsilon} \left( t \right) \geq 0$ hence $\dot{\gamma_{\varepsilon}} \left( t \right) = \varepsilon \left( t \right) \gamma_{\varepsilon} \left( t \right) \geq \dot{\gamma_{\varepsilon}} \left( t_{0} \right) = \varepsilon \left( t_{0} \right)$.
The rates given in the second case in \eqref{Tikh:eps:rates} and  \eqref{Tikh:eps:rates1} follow from the estimate above together with \eqref{Tikh:eps:est:sup} and Proposition \ref{thm:Tikh:pre}.
\end{itemize}

Since $\gamma_{\varepsilon} \left( t \right) \to + \infty$ and $\varepsilon \left( t \right) \to 0$,  we have $\norm{M \left( x (t) \right)} \to 0$ as $t \to + \infty$,  which leads to the strong convergence of the trajectory according to \Cref{prop:Tikh:converge}.
\end{proof}

As {a first} application of the theoretical considerations above, we consider dynamics generated by \eqref{ds:Tikhonov} for the regularization function
\begin{equation*}
\varepsilon \colon \left[ t_{0} , +\infty \right) \to \sR, \quad \varepsilon \left( t \right) \coloneq \frac{\alpha}{t}, \quad \textrm{ with } t_{0} > 0 \textrm{ and } \alpha > 0,
\end{equation*}
which evidently satisfies \eqref{cond:eps}.  We have $\gamma_{\varepsilon} \left( t \right) = \left( \frac{t}{t_{0}} \right) ^{\alpha}$ for every $t \geq t_{0}$, and moreover $\varepsilon_i = \varepsilon_s = \frac{1}{\alpha}$. The cases $0 < \alpha < 1$ { and $\alpha > 1$} correspond to the {two} scenarios in \Cref{thm:Tikh:eps}, while the critical case $\alpha=1$ has been considered in Theorem \ref{thm:Tikh:q=1:alpha=1}.

\begin{corollary}[the case $\varepsilon \left( t \right) \coloneq \frac{\alpha}{t}$]
\label{thm:Tikh:q=1}
Let $t_{0} > 0$, $\alpha > 0$, and $x \colon \left[ t_{0} , + \infty \right) \to \sH$ be the trajectory solution of the dynamical system
\begin{equation}
\label{ds:q=1}
\dot{x} \left( t \right) + M \left( x \left( t \right) \right) + \dfrac{\alpha}{t} x \left( t \right) = 0 ,
\end{equation}
with initial condition $x \left( t_{0} \right) \coloneq x_{0} \in \sH$. 
Then, the following statements are true:
\begin{enumerate}
\item 
as $t \to + \infty$ it holds
\begin{align*}
\norm{\dot{x} \left( t \right)} = \bO \left( \dfrac{1}{t^{\min \left\lbrace \alpha , 1 \right\rbrace}} \right)
\quad \textrm{ and } \quad
\norm{M \left( x \left( t \right) \right)} = \bO \left( \dfrac{1}{t^{\min \left\lbrace \alpha , 1 \right\rbrace}} \right);
\end{align*}

\item the trajectory $x(t)$ converges strongly to the minimum norm solution $\xi_{*} \coloneq \proj_{\zer M} \left( 0 \right)$ as $t \to + \infty$.
\end{enumerate}
\end{corollary}

A choice for the regularization function closely related to the previous one is
\begin{equation*}
\varepsilon \colon \left[ t_{0} , +\infty \right) \to \sR, \quad \varepsilon \left( t \right) \coloneq \frac{\alpha}{t^{q}}, \quad \textrm{ with } t_{0} > 0, \quad  0 < q < 1, \textrm{ and } \alpha > 0,
\end{equation*}
which satisfies \eqref{cond:eps},  too. 
Furthermore, in this case we have $\gamma_{\varepsilon} \left( t \right) = \exp \left( \frac{\alpha}{1-q} \left( t^{1-q} - t_{0}^{1-q} \right) \right)$ for every $t \geq t_{0}$.
In addition, $0 \leq \varepsilon_s = \sup_{t \geq t_{0}} \frac{- \dot{\varepsilon} \left( t \right)}{\varepsilon^{2} \left( t \right)} = \sup_{t \geq t_{0}}  \frac{q}{\alpha t^{1-q}} = \frac{q}{\alpha t_{0}^{1-q}}$,  hence the statements follow from  \Cref{thm:Tikh:eps}.

\begin{corollary}	[the case $\varepsilon \left( t \right) \coloneq \frac{\alpha}{t^q}, 0 < q < 1$]
\label{thm:Tikh:0<q<1}
Let $t_{0} > 0, 0 < q < 1$, $\alpha >0$ such that $\alpha t_{0}^{1-q} > {q}$, and $x \colon \left[ t_{0} , + \infty \right) \to \sH$ be the trajectory solution of the dynamical system
\begin{equation}
\label{ds:0<q<1}
\dot{x} \left( t \right) + M \left( x \left( t \right) \right) + \dfrac{\alpha}{t^{q}} x \left( t \right) = 0,
\end{equation}
with initial condition $x \left( t_{0} \right) \coloneq x_{0} \in \sH$. Then, the following statements are true:
\begin{enumerate}	
\item 
as $t \to + \infty$ it holds
\begin{align*}
\norm{\dot{x} \left( t \right)} = \bO \left( \dfrac{1}{t^{q}} \right)
\quad \textrm{ and } \quad
\norm{M \left( x \left( t \right) \right)} = \bO \left( \dfrac{1}{t^{q}} \right);
\end{align*}

\item 
the trajectory $x(t)$ converges strongly to the minimum norm solution $\xi_{*} \coloneq \proj_{\zer M} \left( 0 \right)$ as $t \to + \infty$.
\end{enumerate}
\end{corollary}

The above theorems suggest that the best rates of convergence occur when they are $\mathcal{O}(\frac{1}{t})$ as $t \rightarrow +\infty$, achieved when $q=1$ and $\alpha \geq 1$. This represents a clear improvement concerning the convergence rates compared to those of the unregularized dynamical system, even in cases where the operator is cocoercive (see Theorem \ref{thm:cocoercive}). Should one aim to enhance the rates of convergence for both the velocity and the operator norm by opting for $q > 1$, it is important to note, as highlighted in Remark \ref{remark21} (ii), that strong convergence of the trajectory to the minimum norm solution cannot be expected.

For $t \geq t_{0} >1$, we consider the regularization function $\varepsilon \colon \left[ t_{0}, +\infty \right) \to \sR,  \varepsilon(t) = \frac{1}{t \log(t)}$. 
Notice that 
\begin{equation*}
\varepsilon_i = \inf_{t \geq t_{0}} \dfrac{d}{dt} \left[ \dfrac{1}{\varepsilon \left( t \right)} \right] 
= \inf_{t \geq t_{0}} \dfrac{d}{dt} \left( t \log \left( t \right) \right) 
= \inf_{t \geq t_{0}} \left( \log \left( t \right) + 1 \right)
= \log \left( t_{0} \right) + 1 > 1 .
\end{equation*}
In addition, condition \eqref{eps:lim} as well as condition \eqref{eps:BV} is satisfied.  
For every $t \geq t_{0}$ it holds
\begin{equation*}
\gamma_{\varepsilon} \left( t \right) = \exp \left( \int_{t_{0}}^{t} \frac{1}{s \log(s)} ds \right) = \exp \left( \log(\log(t)) - \log(\log(t_{0})) \right) = \frac{\log(t)}{\log(t_{0})},
\end{equation*}
thus \Cref{thm:Tikh:eps} leads to the following result.

\begin{corollary}
	[the case $\varepsilon \left( t \right) \coloneq \frac{1}{t \log \left( t \right)}$]
\label{thm:Tikh:loglog}
Let $t_{0} > 1$ and $x \colon \left[ t_{0} , + \infty \right) \to \sH$ be the trajectory solution of the dynamical system
\begin{equation}
\label{ds:loglog}
\dot{x} \left( t \right) + M \left( x \left( t \right) \right) + \dfrac{1}{t \log(t)} x \left( t \right) = 0 ,
\end{equation}
with initial condition $x \left( t_{0} \right) \coloneq x_{0} \in \sH$. Then, the following statements are true:
\begin{enumerate}	
\item 
as $t \to + \infty$ it holds
\begin{align*}
\norm{\dot{x} \left( t \right)} = \bO \left( \dfrac{1}{\log \left( t \right)} \right)
\quad \textrm{ and } \quad
\norm{M \left( x \left( t \right) \right)} = \bO \left( \dfrac{1}{\log \left( t \right)} \right) .
\end{align*}

\item the trajectory $x(t)$ converges strongly to the minimum norm solution $\xi_{*} \coloneq \proj_{\zer M} \left( 0 \right)$ as $t \to + \infty$.
\end{enumerate}
\end{corollary}

We have seen that even if the regularization function converges to zero faster than {$\bO \left( \frac{1}{t} \right)$}, while satisfying the standing assumptions \eqref{cond:eps},  the rates of convergence for the  velocity of the trajectory and the operator along the trajectory cannot be expected to be faster than $\mathcal{O}(\frac{1}{t})$ as $t \rightarrow +\infty$.

\subsection{First-order dynamical system with Tikhonov regularization and time rescaling}\label{sec23}

In this subsection, we will extend the previous analysis to first-order dynamical systems that incorporate Tikhonov regularization and time rescaling, expressed through the presence of a time-dependent coefficient of the operator. With the introduction and analysis of this continuous dynamic, we take the first step in our attempt to demonstrate the close connection between first-order dynamical systems with Tikhonov regularization and second-order dynamical systems with a vanishing damping term associated with the monotone equation \eqref{intro:mono}.

\begin{mdframed}
Let $s_{0} \geq 0$ and $\beta, \delta \colon \left[ s_{0} , + \infty \right) \to \sR_{++}$ be given functions.
We consider on $\left[ s_{0} , + \infty \right)$ the following dynamical system
\begin{equation}
\label{ds:rescaling}
\dot{z} \left( s \right) + \beta \left( s \right) M \left( z \left( s \right) \right) + \delta \left( s \right) z \left( s \right) = 0 ,
\end{equation}
with initial condition $z \left( s_{0} \right) \coloneq z_{0} \in \sH$.
\end{mdframed}

For the time rescaling function $\beta$ and the regularization function $\delta$, we consider the following assumptions.
\begin{mdframed}
The functions $\beta, \delta \colon \left[ s_{0} , + \infty \right) \to \sR_{++}$ are required to be continuously differentiable and to satisfy
\begin{equation}
\label{cond:beta-delta}
\lim_{s \to + \infty} \dfrac{\delta \left( s \right)}{\beta \left( s \right)} = 0 , \quad
\int_{s_{0}}^{+ \infty} \delta \left( s \right) ds = + \infty ,
\quad \textrm{ and } \quad
\int_{s_{0}}^{+ \infty} \beta \left( s \right) ds = + \infty .
\end{equation}
The first-order dynamical system is assumed to admit a unique strong global solution $z : [s_{0}, +\infty) \rightarrow \sH$ with the property that $s \mapsto M(x(s))$ is absolutely continuous on every compact interval $[s_{0},S]$.
\end{mdframed}

The following proposition will allows us to transfer the statements obtained for \eqref{ds:Tikhonov} in the previous section to \eqref{ds:rescaling}.
\begin{proposition}
\label{prop:time-rescaling}
Let $s_{0}, t_{0} \geq 0$ and $\beta, \delta \colon \left[ s_{0} , + \infty \right) \to \sR_{++}$, $\varepsilon \colon \left[ t_{0} , + \infty \right) \to \sR_{++}$ be continuously differentiable functions.
Suppose that the time rescaling function satisfies
\begin{equation*}
\int_{s_{0}}^{+ \infty} \beta \left( s \right) ds = + \infty .
\end{equation*}
Let $x \colon \left[ t_{0} , + \infty \right) \to \sH$ be the trajectory solution of
\begin{equation}
\label{ds:eq:Tikhonov}
\dot{x} \left( t \right) + M \left( x \left( t \right) \right) + \varepsilon \left( t \right) x \left( t \right) = 0 \quad \textrm{ with } \quad x \left( t_{0} \right) \coloneq z_{0} \in \sH ,
\end{equation}
and $z \colon \left[ s_{0} , + \infty \right) \to \sH$ be the trajectory solution of
\begin{equation}
\label{ds:eq:rescaling}
\dot{z} \left( s \right) + \beta \left( s \right) M \left( z \left( s \right) \right) + \delta \left( s \right) z \left( s \right) = 0 \quad \textrm{ with } \quad z \left( s_{0} \right) \coloneq z_{0} \in \sH .
\end{equation}
Then the two trajectories are equivalent, subject to a time rescaling process. Precisely, there exist bijective functions $\tau \colon \left[ s_{0} , + \infty \right) \to \left[ t_{0} , + \infty \right)$ and $\sigma \colon \left[ t_{0} , + \infty \right) \to \left[ s_{0} , + \infty \right)$ such that $\sigma = \tau^{-1}$ and
\begin{enumerate}
\item if $x(\cdot)$ is the trajectory solution of \eqref{ds:eq:Tikhonov}, then $z(s) \coloneq x \left( \tau \left( s \right) \right)$ defines the trajectory solution of \eqref{ds:eq:rescaling} with $\delta(s) \coloneq \varepsilon \left( \tau \left( s \right) \right) \beta \left( s \right)$ for every $s \geq s_{0}$;

\item if $z(\cdot)$  is the trajectory solution of \eqref{ds:eq:rescaling}, then $x \left( t \right) \coloneq z \left( \sigma \left( t \right) \right)$ defines the trajectory solution of \eqref{ds:eq:Tikhonov} with $\varepsilon(t) \coloneq \frac{\delta \left( \sigma \left( t \right) \right)}{\beta \left( \sigma \left( t \right) \right)}$ for every $t \geq t_{0}$.
\end{enumerate}
Moreover,
\begin{equation*}
\lim_{t \to + \infty} \varepsilon \left( t \right) = 0  \Leftrightarrow \lim_{s \to + \infty} \dfrac{\delta \left( s \right)}{\beta \left( s \right)} = 0 \quad
\textrm{ and } \quad
\int_{t_{0}}^{+ \infty} \varepsilon \left( t \right) dt = + \infty \Leftrightarrow \int_{s_{0}}^{+ \infty} \delta \left( s \right) ds = + \infty.
\end{equation*}
\end{proposition}
\begin{proof}
We define $\tau \colon \left[ s_{0} , + \infty \right) \to \left[ t_{0} , + \infty \right)$ for every $s \geq s_{0}$ as
\begin{equation}
\label{defi:tau}
\tau \left( s \right) = t_{0} + \int_{s_{0}}^{s} \beta \left( u \right) du .
\end{equation}
Since $\beta$ is strictly positive and continuous, $\tau$ is a monotonically increasing and thus injective. Moreover, it holds $\tau \left( s_{0} \right) = t_{0}$ and $\lim_{s \to + \infty} \tau \left( s \right) = + \infty$, meaning that $\sigma$ is also surjective and therefore, bijective. For clarity, we will denote its inverse $\tau^{-1}$ by $\sigma \colon \left[ t_{0} , + \infty \right) \to \left[ s_{0} , + \infty \right)$.

\item[\underline{``$x(t) \Mapsto z(s)$''}.] {Let} $x(\cdot)$ be the trajectory solution of \eqref{ds:eq:Tikhonov}. We perform a change of time variable, setting $t \coloneq \tau(s)$ in the dynamics \eqref{ds:eq:Tikhonov}, and define for every $s \geq s_{0}$
\begin{equation}
	\label{defi:x-tau}
z \left( s \right) \coloneq x \left( \tau \left( s \right) \right) .
\end{equation}
Then by the chain rule, we have for almost every $s \geq s_{0}$
\begin{equation*}
\dot{z} \left( s \right) = \dot{x} \left( \tau \left( s \right) \right) \dot{\tau} \left( s \right) = \dot{x} \left( \tau \left( s \right) \right) \beta \left( s \right) .
\end{equation*}
At the time $t = \tau \left( s \right)$, multiplying both sides of \eqref{ds:eq:Tikhonov} by $\beta \left( s \right) > 0$ yields for almost every $s \geq s_{0}$
\begin{equation*}
\dot{z} \left( s \right) + \beta \left( s \right) M \left( z \left( s \right) \right) + \varepsilon \left( \tau \left( s \right) \right) \beta \left( s \right) z \left( s \right) = 0 .
\end{equation*}
This is equivalent to \eqref{ds:eq:rescaling} if for every $s \geq s_{0}$ we choose
\begin{equation}
\label{rescaling:delta-eps-beta}
\delta \left( s \right) \coloneq \varepsilon \left( \tau \left( s \right) \right) \beta \left( s \right) .
\end{equation}

Suppose that $\lim_{t \to + \infty} \varepsilon \left( t \right) = 0$ and $\int_{t_{0}}^{+ \infty} \varepsilon \left( t \right) dt = + \infty$.
It follows from \eqref{rescaling:delta-eps-beta} and the property that $\lim_{s \to + \infty} \tau \left( s \right) = + \infty$ that
\begin{equation*}
	\lim_{s \to + \infty} \dfrac{\delta \left( s \right)}{\beta \left( s \right)} = \lim_{s \to + \infty} \varepsilon \left( \tau \left( s \right) \right) = 0 .
\end{equation*}
Furthermore, for every $s \geq s_{0}$, we have
\begin{equation}
\label{rescaling:eq:exp}
\int_{s_{0}}^{s} \delta \left( u \right) du
= \int_{s_{0}}^{s} \varepsilon \left( \tau \left( u \right) \right) \beta \left( u \right) du 
= \int_{s_{0}}^{s} \varepsilon \left( \tau \left( u \right) \right) \dot{\tau} \left( u \right) du
= \int_{\tau \left( s_{0} \right)}^{\tau \left( s \right)} \varepsilon \left( r \right) dr
= \int_{t_{0}}^{\tau \left( s \right)} \varepsilon \left( r \right) dr .
\end{equation}
Taking the limit as $s$ approaches $+ \infty$, we obtain $\int_{s_{0}}^{+ \infty} \delta \left( u \right) du = + \infty$.

\item[\underline{``$x(t) \Mapsfrom z(s)$''}.]	
{Let} $z(\cdot)$ be the trajectory solution of \eqref{ds:eq:rescaling}. We substitute $s \coloneq \sigma \left( t \right)$ in the dynamic \eqref{ds:eq:rescaling}.
By definition, for every $t \geq t_{0}$ it holds
\begin{equation*}
\tau \left( \sigma \left( t \right) \right) = t
\end{equation*}
and differentiation with respect to $t$ gives
\begin{equation}
\label{defi:d-sigma}
\dot{\tau} \left( \sigma \left( t \right) \right) \dot{\sigma} \left( t \right) = \beta \left( \sigma \left( t \right) \right) \dot{\sigma} \left( t \right) = 1 \Leftrightarrow \dot{\sigma} \left( t \right) = \dfrac{1}{\beta \left( \sigma \left( t \right) \right)} .
\end{equation}
Therefore, if we set for every $t \geq t_{0}$
\begin{equation*}
x \left( t \right) \coloneq z \left( \sigma \left( t \right) \right) ,
\end{equation*}
by the chain rule it yields for almost every $t \geq t_{0}$
\begin{equation*}
\dot{x} \left( t \right) = \dot{z} \left( \sigma \left( t \right) \right) \dot{\sigma} \left( t \right) = \dot{z} \left( \sigma \left( t \right) \right) \cdot \dfrac{1}{\beta \left( \sigma \left( t \right) \right)} .
\end{equation*}
Consequently, at the time $s = \sigma \left( t \right)$, dividing both sides of equation \eqref{ds:eq:rescaling} by $\beta \left( \sigma \left( t \right) \right) > 0$, we deduce that for almost every $t \geq t_{0}$
\begin{equation}
\dot{x} \left( t \right) + M \left( x \left( t \right) \right) + \dfrac{\delta \left( \sigma \left( t \right) \right)}{\beta \left( \sigma \left( t \right) \right)} x \left( t \right) = 0.
\end{equation}
By setting for every $t \geq t_{0}$
\begin{equation}
\label{defi:eps-sigma}
\varepsilon \left( t \right) \coloneq \dfrac{\delta \left( \sigma \left( t \right) \right)}{\beta \left( \sigma \left( t \right) \right)},
\end{equation}
this is nothing else than \eqref{ds:eq:Tikhonov}.

Suppose that $\lim_{s \to + \infty} \frac{\delta \left( s \right)}{\beta \left( s \right)} = 0$ and $\int_{s_{0}}^{+ \infty} \delta \left( s \right) ds = + \infty$. By following previous arguments, we have
\begin{equation*}
\lim_{t \to + \infty} \varepsilon \left( t \right) = \lim_{t \to + \infty} \dfrac{\delta \left( \sigma \left( t \right) \right)}{\beta \left( \sigma \left( t \right) \right)} = 0 .
\end{equation*}
Furthermore, for every $s \geq s_{0}$ it holds 
\begin{equation*}
\int_{t_{0}}^{t} \varepsilon \left( r \right) dr = \int_{t_{0}}^{t} \dfrac{\delta \left( \sigma \left( r \right) \right)}{\beta \left( \sigma \left( r \right) \right)} dr = \int_{t_{0}}^{t} \delta \left( \sigma \left( r \right) \right) \dot{\sigma} \left( r \right) dr = \int_{\sigma \left( t_{0} \right)}^{\sigma \left( t \right)} \delta \left( u \right) du = \int_{s_{0}}^{\sigma \left( t \right)} \delta \left( u \right) du ,
\end{equation*}
thus, by taking the limit as $t$ approaches $+ \infty$, we conclude  $\int_{t_{0}}^{+ \infty} \varepsilon \left( t \right) dt = + \infty$.
\end{proof}

\cref{prop:time-rescaling} allows us to transfer the convergence and convergence rate results stated in \cref{thm:Tikh,thm:Tikh:eps}, respectively, to the dynamical system \eqref{ds:rescaling}.  It would be useful to mention that if $\varepsilon(\cdot)$ satisfies \eqref{defi:eps-sigma}, then for every $t \geq t_{0}$ we have
\begin{align}
\label{der:eps-sigma}
\dot{\varepsilon} \left( t \right) & =  \dfrac{1}{\beta \left( \sigma (t) \right)} \left[ \dfrac{d}{ds} \left( \dfrac{\delta}{\beta} \right) \right] \left( \sigma (t) \right) 
= \dfrac{\delta \left( \sigma (t) \right)}{\beta^{2} \left( \sigma (t) \right)} \left( \dfrac{\dot{\delta} \left( \sigma (t) \right)}{\delta \left( \sigma (t) \right)} - \dfrac{\dot{\beta} \left( \sigma (t) \right)}{\beta \left( \sigma (t) \right)} \right) .
\end{align}

\begin{theorem}
\label{thm:rescaling}
Let $z \colon \left[ s_{0} , + \infty \right) \to \sH$ be the trajectory solution of the dynamical system \eqref{ds:rescaling}. Regarding $\beta$ and $\delta$, assume that, in addition to \eqref{cond:beta-delta}, they satisfy either
\begin{equation}
\label{rescaling:BV}
\int_{s_{0}}^{+ \infty} \abs{\left[ \dfrac{d}{ds} \left( \dfrac{\delta}{\beta} \right) \right] \left( s \right)} ds < + \infty
\end{equation}
or
\begin{equation}
\label{rescaling:lim}
\lim\limits_{s \to + \infty} \dfrac{1}{\beta \left( s \right)} \abs{\dfrac{\dot{\delta} \left( s \right)}{\delta \left( s \right)} - \dfrac{\dot{\beta} \left( s \right)}{\beta \left( s \right)}} = 0.
\end{equation}
Then, the trajectory $z(s)$ converges strongly to the minimum norm solution $\xi_{*} \coloneq \proj_{\zer M} \left( 0 \right)$ as $s \to + \infty$.
\end{theorem}
\begin{proof}
Given the trajectory solution $z(\cdot)$ of \eqref{ds:rescaling} and
\begin{equation*}
\tau(s) = t_{0} + \int_{s_{0}}^{s} \beta \left( u \right) du  \quad \forall s \geq s_{0},
\end{equation*}
where $t_{0} \geq 0$, then, according to \cref{prop:time-rescaling}, $x(\cdot) \coloneq z \left( \tau^{-1}(\cdot) \right)$ is the trajectory solution of \eqref{ds:Tikhonov} with
\begin{equation*}
\varepsilon(t) \coloneq \dfrac{\delta\left( \tau^{-1}(t) \right)}{\beta\left( \tau^{-1}(t) \right)} \quad \forall t \geq t_{0} .
\end{equation*}

According to \eqref{der:eps-sigma}, for every $t \geq t_{0}$ we have
\begin{equation*}
\dfrac{\abs{\dot{\varepsilon} \left( t \right)}}{\varepsilon \left( t \right)}
= \dfrac{1}{\beta\left( \tau^{-1}(t) \right)} \abs{\left[ \dfrac{d}{ds} \left( \dfrac{\delta}{\beta} \right) \right] \left( \tau^{-1}(t) \right)} \frac{\beta\left( \tau^{-1}(t) \right)}{\delta\left( \tau^{-1}(t) \right)} = \dfrac{1}{\beta \left( \tau^{-1}(t) \right)}\abs{\dfrac{\dot{\delta} \left( \tau^{-1}(t) \right)}{\delta \left( \tau^{-1}(t) \right)} - \dfrac{\dot{\beta} \left( \tau^{-1}(t) \right)}{\beta \left( \tau^{-1}(t) \right)}} .
\end{equation*}
Therefore, if \eqref{rescaling:lim} is satisfied, then
\begin{equation*}
\lim_{t \to +\infty} \dfrac{\abs{\dot{\varepsilon} \left( t \right)}}{\varepsilon \left( t \right)} = 0.
\end{equation*}

On the other hand, by using a substitution argument,
\begin{equation*}
\int_{t_{0}}^{+\infty} \abs{\dot{\varepsilon} \left( t \right)} dt = \int_{t_{0}}^{+\infty} \dfrac{1}{\beta\left( \tau^{-1}(t) \right)} \abs{\left[ \dfrac{d}{ds} \left( \dfrac{\delta}{\beta} \right) \right] \left( \tau^{-1}(t) \right)} dt = \int_{s_{0}}^{+ \infty} \abs{\left[ \dfrac{d}{ds} \left( \dfrac{\delta}{\beta} \right) \right] \left( s \right)} ds.
\end{equation*}
Therefore, if \eqref{rescaling:BV} is satisfied, then
\begin{equation*}
\int_{t_{0}}^{+\infty} \abs{\dot{\varepsilon} \left( t \right)} dt  < +\infty.
\end{equation*}

Invoking Theorem \ref{thm:Tikh}, the trajectory $x(t) = z \left( \tau^{-1}(t) \right)$ converges strongly to the minimum norm solution $\xi_{*} \coloneq \proj_{\zer M} \left( 0 \right)$ as $t \to + \infty$. This is nothing else than the trajectory $z(s)$ converges strongly to  $\xi_{*}$ as $s \to + \infty$.
\end{proof}

To obtain convergence rates, we need to derive conditions similar to those in \eqref{eps:inf-sup}.
This can be done using the formula \eqref{der:eps-sigma} above.  We first notice that for every $s \geq s_{0}$
\begin{equation*}
\dfrac{- \frac{1}{\beta \left( s \right)} \left[ \frac{d}{ds} \left( \frac{\delta}{\beta} \right) \right] \left( s \right)}{\frac{\delta^{2} \left( s \right)}{\beta^{2} \left( s \right)}}
= - \dfrac{\beta^{2} \left( s \right)}{\delta^{2} \left( s \right)} \dfrac{\dot{\delta} \left( s \right) \beta \left( s \right) - \delta \left( s \right) \dot{\beta} \left( s \right)}{\beta^{3} \left( s \right)} = \dfrac{1}{\delta \left( s \right)} \left( \dfrac{\dot{\beta} \left( s \right)}{\beta \left( s \right)} - \dfrac{\dot{\delta} \left( s \right)}{\delta \left( s \right)} \right) ,
\end{equation*}
and thus define
\begin{equation*}
\delta_{i} \coloneq \inf_{s \geq s_{0}} \dfrac{1}{\delta \left( s \right)} \left( \dfrac{\dot{\beta} \left( s \right)}{\beta \left( s \right)} - \dfrac{\dot{\delta} \left( s \right)}{\delta \left( s \right)} \right)
\quad \textrm{ and } \quad
\delta_{s} \coloneq \sup_{s \geq s_{0}} \dfrac{1}{\delta \left( s \right)} \left( \dfrac{\dot{\beta} \left( s \right)}{\beta \left( s \right)} - \dfrac{\dot{\delta} \left( s \right)}{\delta \left( s \right)} \right) .
\end{equation*}

\begin{theorem}
\label{thm:rescaling:rates}
Let $z \colon \left[ s_{0} , + \infty \right) \to \sH$ be the trajectory solution of the dynamical system \eqref{ds:rescaling}. 
Then, the following statements are true:
\begin{enumerate}	
\item 
as $s \to + \infty$ it holds
\begin{align}
\label{rescaling:eps:rates}
\norm{\dot{z} \left( s \right)} & = \begin{dcases}
\bO \left( \dfrac{1}{\gamma_{\delta} \left( t \right)} \right) & \textrm{ if } 1 < \delta_{i}  \textrm{ or } \delta_{i} = \delta_{s} = 1 , \\
\bO \left( \dfrac{\delta \left( s \right)}{\beta \left( s \right)} \right) 					& \textrm{ if } -1 < \delta_{i} \leq \delta_{s} < 1 ,
\end{dcases}
\\
\label{rescaling:eps:rates1}
\norm{M \left( z \left( s \right) \right)} & = \begin{dcases}
\bO \left( \dfrac{1}{\gamma_{\delta} \left( t \right)} \right) & \textrm{ if } 1 < \delta_{i} \textrm{ or } \delta_{i} = \delta_{s} = 1 , \\
\bO \left( \dfrac{\delta \left( s \right)}{\beta \left( s \right)} \right) 					& \textrm{ if } -1 < \delta_{i} \leq \delta_{s} < 1 ;
\end{dcases} 
\end{align}

\item 
the trajectory $z(s)$ converges strongly to the minimum norm solution $\xi_{*} \coloneq \proj_{\zer M} \left( 0 \right)$ as $s \to + \infty$.
\end{enumerate}
\end{theorem}

\begin{proof}
For $t_{0} \geq 0$ and
\begin{equation*}
\tau(s) = t_{0} + \int_{s_{0}}^{s} \beta \left( u \right) du  \quad \forall s \geq s_{0},
\end{equation*}
$x(\cdot) \coloneq z \left( \tau^{-1}(\cdot) \right)$ is the trajectory solution of \eqref{ds:Tikhonov} with $\varepsilon(\cdot) \coloneq \frac{\delta \left( \tau^{-1}(\cdot) \right)}{\beta \left( \tau^{-1}(\cdot) \right)}$. Since for every $t \geq t_{0}$
\begin{align}
\label{der:inverse:eps-sigma}
\dfrac{d}{dt} \left[ \dfrac{1}{\varepsilon \left( t \right)} \right] 
= \dfrac{- \dot{\varepsilon} \left( t \right)}{\varepsilon^{2} \left( t \right)} 
& = - \dfrac{1}{\beta\left( \tau^{-1}(t) \right)} \left( \left[ \dfrac{d}{ds} \left( \dfrac{\delta}{\beta} \right) \right] \left( \tau^{-1}(t) \right) \right) \frac{\beta^{2} \left( \tau^{-1}(t) \right)}{\delta^{2} \left( \tau^{-1}(t) \right)} \nonumber \\
& = \dfrac{1}{\delta \left( \tau^{-1}(t) \right)} \left( \dfrac{\dot{\beta} \left( \tau^{-1}(t) \right)}{\beta \left( \tau^{-1}(t) \right)} - \dfrac{\dot{\delta} \left( \tau^{-1}(t) \right)}{\delta \left( \tau^{-1}(t) \right)} \right) 
\end{align}
and, therefore,
\begin{equation*}
\delta_{i} 
= \inf_{s \geq s_{0}} \dfrac{1}{\delta \left( s \right)} \left( \dfrac{\dot{\beta} \left( s \right)}{\beta \left( s \right)} - \dfrac{\dot{\delta} \left( s \right)}{\delta \left( s \right)} \right) 
= \inf_{t \geq t_{0}} \dfrac{- \dot{\varepsilon} (t)}{\varepsilon^{2}(t)} 
= \varepsilon_{i} ,
\end{equation*}
the convergence rates given in \cref{thm:Tikh:eps} for \eqref{ds:Tikhonov} can be translated to \eqref{ds:rescaling}. To this end, we have only to write {the estimates in \eqref{Tikh:eps:rates} and \eqref{Tikh:eps:rates1}} for $t:=\tau(s)$ and (almost) every $s \geq s_{0}$, and to take into account that for every $s \geq s_{0}$
\begin{equation*}
\gamma_{\varepsilon} \left( \tau(s) \right) = \exp \left( \int_{t_{0}}^{\tau \left( s \right)} \varepsilon \left( r \right) dr \right) = \exp \left( \int_{s_{0}}^{s} \delta \left( u \right) du \right) = \gamma_{\delta} \left( s \right) ,
\end{equation*}
and for almost every $s \geq s_{0}$
\begin{equation*}
\beta \left( s \right) \norm{\dot{x} \left( \tau \left( s \right) \right)} = \norm{\dot{x} \left( \tau \left( s \right) \right) \dot{\tau} \left( s \right)} = \norm{\dot{z} \left( s \right)} .
\end{equation*}
\end{proof}
To illustrate the two theoretical results of this subsection, we consider the case where, for $\alpha >0$,
\begin{equation*}
\delta \left( s \right) \coloneq \dfrac{\alpha}{s} \quad \forall s \geq s_{0}.
\end{equation*}
With this choice of regularization function, we will later establish a connection between the first-order dynamical system and a second-order system with vanishing damping. For this instance, we have for every $s \geq s_{0}$
\begin{equation}
\label{beta:alpha:small:der}
\dfrac{1}{\delta \left( s \right)} \left( \dfrac{\dot{\beta} \left( s \right)}{\beta \left( s \right)} - \dfrac{\dot{\delta} \left( s \right)}{\delta \left( s \right)} \right) = \dfrac{1}{\alpha} \left( \dfrac{s \dot{\beta} \left( s \right)}{\beta \left( s \right)} + 1 \right) .
\end{equation}
This formula allows us to simplify the conditions imposed in \cref{thm:rescaling:rates}. To this end, we denote
\begin{equation*}
\beta_{i} \coloneq \inf_{s \geq s_{0}} \dfrac{s \dot{\beta} \left( s \right)}{\beta \left( s \right)}
\quad \textrm{ and } \quad
\beta_{s} \coloneq \sup_{s \geq s_{0}} \dfrac{s \dot{\beta} \left( s \right)}{\beta \left( s \right)} .
\end{equation*}

\begin{corollary}
\label{thm:beta:alpha:small}
Let $s_{0} > 0$, $\alpha > 0$, $\beta \colon \left[ s_{0}, + \infty \right) \to \sR_{++}$ a continuously differentiable function and $z \colon \left[ s_{0} , + \infty \right)$ $\to \sH$ the trajectory solution of the dynamical system
\begin{equation}
\label{ds:beta:alpha:small}
\dot{z} \left( s \right) + \beta \left( s \right) M \left( z \left( s \right) \right) + \dfrac{\alpha}{s} z \left( s \right) = 0 ,
\end{equation}
with initial condition $z \left( s_{0} \right) \coloneq z_{0} \in \sH$.
Suppose that
\begin{equation}
\label{beta:alpha:small:cond}
\lim_{s \to + \infty} \dfrac{1}{s \beta \left( s \right)} = 0
\quad \textrm{ and } \quad
\int_{s_{0}}^{+ \infty} \beta \left( s \right) ds = + \infty .
\end{equation}
Then, the following statements are true:
\begin{enumerate}	
\item 
as $s \to + \infty$ it holds
\begin{align}
\norm{\dot{z} \left( s \right)} = \begin{dcases}
\bO \left( \dfrac{1}{s^{\alpha-1}} \right) & \textrm{ if } \alpha - 2 < \beta_{i}  \textrm{ or } \beta_{i} = \beta_{s} = \alpha - 2 , \\
\bO \left( \dfrac{1}{s \beta \left( s \right)} \right) 					& \textrm{ if } -1 - \alpha < \beta_{i} \leq \beta_{s} < \alpha - 2  ,
\end{dcases}
\\
\norm{M \left( z \left( s \right) \right)} = \begin{dcases}
\bO \left( \dfrac{1}{s^{\alpha-1}} \right) & \textrm{ if } \alpha - 2  < \beta_{i}  \textrm{ or } \beta_{i} = \beta_{s} = \alpha - 2 , \\
\bO \left( \dfrac{1}{s \beta \left( s \right)} \right) 					& \textrm{ if } -1 - \alpha < \beta_{i} \leq \beta_{s} < \alpha - 2 ;
\end{dcases} 
\end{align}

\item 
the trajectory $z(s)$ converges strongly to the minimum norm solution $\xi_{*} \coloneq \proj_{\zer M} \left( 0 \right)$ as $s \to + \infty$.
\end{enumerate}
\end{corollary}

The conditions in \eqref{beta:alpha:small:cond} are fulfilled, for instance, for
\begin{equation*}
\beta(s) = s^p, \quad \mbox{ where } p > -1 .
\end{equation*}
In this case, $\frac{s \dot{\beta} \left( s \right)}{\beta \left( s \right)} = p$ for every $s \geq s_{0}$, and therefore $\beta_i = \beta_s = p$.  \Cref{thm:beta:alpha:small} provides, consequently, a 
convergence rate of order $\bO \left( \dfrac{1}{s^{\min \left\lbrace \alpha-1,p+1 \right\rbrace}} \right)$.

\subsection{First-order dynamical system with anchor point}
\label{sec:anchor}

Enhancing the dynamical systems \eqref{ds:rescaling} with an anchor point represents the final step in our attempt to establish a connection between first-order dynamical systems with Tikhonov regularization and second-order dynamical systems with a vanishing damping term.

\begin{mdframed}
Let $s_{0} \geq 0$ and $\beta, \delta \colon \left[ s_{0} , + \infty \right) \to \sR_{++}$ be given functions.
We consider on $\left[ s_{0} , + \infty \right)$ the following dynamical system
\begin{equation}
	\label{ds:anchor}
\dot{y} \left( s \right) + \beta \left( s \right) M \left( y \left( s \right) \right) + \delta \left( s \right) (y \left( s \right) -v) = 0 ,
\end{equation}
with initial condition $y \left( s_{0} \right) \coloneq y_{0} \in \sH$ and anchor point $v \in \sH$.
\end{mdframed}

The convergence and convergence rate results derived for \eqref{ds:rescaling} can be transferred to \eqref{ds:anchor}. In particular, the rates of convergence are not affected by the presence of the anchor point. The only difference is the limit point to which the trajectory converges strongly, now being the closest point to $v$ in the solution set $\zer M$.

Define $z \colon \left[ s_{0} , + \infty \right) \to \sH$ as
\begin{equation}
\label{Halpern:x=y-v}
z \left( s \right) \coloneq y \left( s \right) - v \quad \forall s \geq s_{0} .
\end{equation}
The main idea relies on the observation that $y(\cdot)$ is the trajectory solution of \eqref{ds:anchor} if and only if $z(\cdot)$ is the trajectory solution of
\begin{equation}
\label{Halpern:ds}
\dot{z} \left( s \right) + \beta \left( s \right) M^{v} \left( z \left( s \right) \right) + \delta \left( s \right) z \left( s \right) = 0 ,
\end{equation}
where $M^{v} \colon \sH \to \sH$ is the monotone operator defined as
\begin{equation}
\label{Halpern:Au}
M^{v} \left( z \right) = M \left( z + v \right) \quad \forall z \in \sH .
\end{equation}

The rates of convergence for the velocity and the operator norm follow from Theorem \ref{thm:rescaling:rates}, since $\dot{z} \left( s \right) = \dot{y} \left( s \right)$ and $M^{v} \left( z \left( s \right) \right) = M \left( y \left( s \right) \right)$ for every $s \geq s_{0}$.  If $\beta$ and $\delta$, in addition to \eqref{cond:beta-delta}, satisfy either {\eqref{rescaling:BV} or \eqref{rescaling:lim}}, then the trajectory $y(s)$ converges strongly to  $\proj_{\zer M} \left( v \right)$. Indeed, according to Theorem \ref{thm:rescaling}, $z(s) = y(s)-v$ converges strongly to $\proj_{\zer M^{v}} \left( 0 \right) = \proj_{\zer M} \left( v \right) - v$ (see \cite[Proposition 29.1]{Bauschke-Combettes:book}) as $s \rightarrow +\infty$. In other words, $y(s)$ converges strongly to $ \proj_{\zer M} \left( v \right)$ as $s \rightarrow +\infty$.

\begin{remark}\label{paper:ryu}
The dynamic \eqref{ds:anchor} was also studied by Suh, Park, and Ryu in \cite{Suh-Park-Ryu} in finite-dimensional spaces, in the case $\beta(s) \coloneq1$ and $\varepsilon \left( s \right) \coloneq \frac{\alpha}{s^{q}}$ for every $s \geq s_{0} >0$, where $\alpha > 0$ and $0 < q \leq 1$. By relying on a very intricate Lyapunov analysis, the convergence of the trajectory and convergence rate for the operator norm has been provided. These results are particular cases of our more general approach, see Theorem \ref{thm:Tikh:0<q<1} and  Theorem \ref{thm:Tikh:q=1},  which is fundamentally different, as it relies on the study of the two very simple functions $\varphi$ and $\psi$, defined as in \eqref{defi:phi} and \eqref{defi:g}, respectively. In addition, our approach provides a rate of convergence for the velocity. The dynamics \eqref{ds:anchor} is referred to as the continuous model of the Halpern method. In light of the above considerations, it is evident that the Halpern method is nothing else than a particular instance of the classical Tikhonov regularization approach.
\end{remark}

\section{\!\!From first-order dynamics with Tikhonov regularization to second-order dynamics with vanishing damping term}\label{sec3}

In this section, we will demonstrate that, when applied to monotone equations of the type \eqref{intro:mono}, first-order dynamics with a specific Tikhonov regularization term are intimately related to first-order dynamics with a vanishing damping term. The latter have been proved in \cite{Bot-Csetnek-Nguyen} to accelerate the convergence rate of the operator norm with respect to the first-order monotone operator flow, while preserving the weak convergence of the trajectory to a zero of the operator.

\begin{mdframed}
Let $s_{0} >0$, $\alpha > 1$ and $\beta : \left[ s_{0} , + \infty \right) \rightarrow \sR_{++}$ a continuously differentiable function. We consider on $\left[ s_{0} , + \infty \right)$ the dynamical system
\begin{equation}	
\label{ds:second-order}
\ddot{z} \left( s \right) + \dfrac{\alpha}{s} \dot{z} \left( s \right) + \beta \left( s \right) \dfrac{d}{ds} M \left( z \left( s \right) \right) + \left( \dot{\beta} \left( s \right) + \dfrac{\beta \left( s \right)}{s} \right) M \left( z \left( s \right) \right) = 0,
\end{equation}
with initial conditions $z \left( s_{0} \right) = z_{0} \in \sH$  and $\dot{z} \left( s_{0} \right) = \dot{z}_{0} \in \sH$.
\end{mdframed}

We assume that it admits a unique strong global solution $z : \left[ s_{0} , + \infty \right) \rightarrow \sH$, meaning that $z$ and ${\dot{z}}$ are absolutely continuous on every compact interval $[s_{0},S]$ and that $z$ satisfies \eqref{ds:second-order} almost everywhere, with the property that $s\mapsto M(z(s))$ is absolutely continuous on every compact interval $[s_{0},S]$. 

\begin{proposition}\label{prop:connection}
Let $s_{0} > 0$ and $\alpha > 1$. Then $z \colon \left[ s_{0} , + \infty \right) \to \sH$ is the trajectory solution of \eqref{ds:second-order} with the initial conditions $z \left( s_{0} \right) = z_{0} \textrm{ and } \dot{z} \left( s_{0} \right) = \dot{z}_{0}$ if and only if it is the trajectory solution of
\begin{equation}
\label{ds:first-order}
\dot{z} \left( s \right) + \beta \left( s \right) M \left( z \left( s \right) \right) + \dfrac{\alpha - 1}{s} \left( z \left( s \right) - v \right) = 0,
\end{equation}
with initial condition $z \left( s_{0} \right) = z_{0}$ and anchor point
$$v \coloneq z_0 + \dfrac{s_{0}}{\alpha-1}\left(\dot{z}_{0} + \beta \left( s_{0} \right) M \left( z_{0} \right) \right).$$
\end{proposition}
\begin{proof}
\item[\underline{``\emph{second-order}'' $\Mapsto$ ``\emph{first-order}''.}]
Let $z \colon \left[ s_{0} , + \infty \right) \to \sH$ be the trajectory solution of \eqref{ds:second-order} with initial conditions $z \left( s_{0} \right) = z_{0} \textrm{ and } \dot{z} \left( s_{0} \right) = \dot{z}_{0}$.
For almost every $s \geq s_{0}$ it holds
\begin{align*}
\dfrac{d}{ds} \left( s \dot{z} \left( s \right) + s \beta \left( s \right) M \left( z \left( s \right) \right) \right) 
= s \ddot{z} \left( s \right) + \dot{z} \left( s \right) + s \beta \left( s \right) \dfrac{d}{ds} M \left( z \left( s \right) \right) + \left( s \dot{\beta} \left( s \right) + \beta \left( s \right) \right) M \left( z \left( s \right) \right) = \left( 1 - \alpha \right) \dot{z} \left( s \right) .
\end{align*}
By integration from $s_{0}$ to $s$, we get for almost every  $s \geq s_{0}$
\begin{equation*}
s \dot{z} \left( s \right) - s_{0} \dot{z}_{0} + s \beta \left( s \right) M \left( z \left( s \right) \right) - s_{0} \beta \left( s_{0} \right) M \left( z_{0} \right) = \left( 1 - \alpha \right) \left( z \left( s \right) - z_{0} \right)
\end{equation*}
or, equivalently,
\begin{equation*}
\label{ds:}
\dot{z} \left( s \right) + \beta \left( s \right) M \left( z \left( s \right) \right) + \dfrac{\alpha - 1}{s} \left( z \left( s \right) - v \right) = 0 ,
\end{equation*}
where $v \coloneq z_{0} + \frac{s_{0}}{\alpha - 1} \left(\dot{z}_{0} + \beta \left( s_{0} \right) M \left( z_{0} \right) \right)$.

\item[\underline{``\emph{first-order}'' $\Mapsto$ ``\emph{second-order}''.}]
Starting with the trajectory solution $z \colon \left[ s_{0} , + \infty \right) \to \sH$ of the first-order system \eqref{ds:first-order} with initial condition $z \left( s_{0} \right) \coloneq z_{0} \in \sH$ and anchor point $v \coloneq z_{0} + \frac{s_{0}}{\alpha - 1} \left(\dot{z}_{0} + \beta \left( s_{0} \right) M \left( z_{0} \right) \right)$, it satisfies for almost every $s \geq s_{0}$
\begin{equation*}
s \dot{z} \left( s \right) + s \beta \left( s \right) M \left( z \left( s \right) \right) + \left( \alpha - 1 \right) \left( z \left( s \right) - v \right) = 0 .
\end{equation*}
In addition, $\dot{z} \left( s_{0} \right) = \dot{z}_{0}$. After differentiation, it yields for almost every $s \geq s_{0}$
\begin{equation*}
s \ddot{z} \left( s \right) + \dot{z} \left( s \right) + s \beta \left( s \right) \dfrac{d}{ds} M \left( z \left( s \right) \right) + \left( s \dot{\beta} \left( s \right) + \beta \left( s \right) \right) M \left( z \left( s \right) \right) + \left( \alpha - 1 \right) \dot{z} \left( s \right) = 0 ,
\end{equation*}
which is nothing else than \eqref{ds:second-order}.
\end{proof}

Theorem \ref{thm:beta:alpha:small} and the considerations in Subsection \ref{sec:anchor} lead to the following result, which we present using \emph{Big}-$\bO$ notation to provide the readers with a clearer overview and facilitate comparison with other works in the literature.
\begin{theorem}
\label{thm:second-order}
Let $s_{0} > 0$, $\alpha > 1$, $\beta \colon \left[ s_{0} , + \infty \right) \to \sR_{++}$ a continuously differentiable function and $z \colon \left[ s_{0} , + \infty \right)$ $\to \sH$ the trajectory solution of the dynamical system \eqref{ds:second-order} with initial conditions $z ( s_{0}) = z_{0}$ and $\dot{z} \left( s_{0} \right) = \dot{z_{0}}$. 
Then, the following statements are true:
\begin{enumerate}
\item if 
\begin{equation*}
\alpha - 2 < \inf_{s \geq s_{0}} \dfrac{s \dot{\beta} \left( s \right)}{\beta \left( s \right)},
\end{equation*}
then
\begin{align*}
\norm{\dot{z} \left( s \right)} = \bO \left(\dfrac{\beta \left( s \right)}{s^{\alpha}} \right) \quad \mbox{and} \quad \norm{M \left( z \left( s \right) \right)} = \bO \left(\dfrac{1}{s^{\alpha}}\right) \quad \mbox{as} \ s \rightarrow + \infty;
\end{align*}

\item if 
\begin{equation*}
-1 < \inf_{s \geq s_{0}} \dfrac{s \dot{\beta} \left( s \right)}{\beta \left( s \right)} \leq \sup_{s \geq s_{0}} \dfrac{s \dot{\beta} \left( s \right)}{\beta \left( s \right)} < \alpha - 2 
\end{equation*}
or
\begin{equation*}
\dfrac{s \dot{\beta} \left( s \right)}{\beta \left( s \right)} = \alpha -2 \quad \forall s \geq s_{0},
\end{equation*}
then
\begin{align*}
\norm{\dot{z} \left( s \right)} = \bO \left(\dfrac{1}{s} \right) \quad \mbox{and} \quad \norm{M \left( z \left( s \right) \right)} = \bO \left(\dfrac{1}{s \beta \left( s \right)}\right) \quad \mbox{as} \ s \rightarrow + \infty.
\end{align*}

\item 
the trajectory $z(s)$ converges strongly to $\proj_{\zer M} \left( z_{0} + \frac{s_{0}}{\alpha - 1} \left(\dot{z}_{0} + \beta \left( s_{0} \right) M \left( z_{0} \right) \right)\right)$ as $s \to + \infty$.
\end{enumerate}
\end{theorem}

\begin{remark}\label{rem33}
In \cite{Bot-Csetnek-Nguyen},  for the Fast Optimistic Gradient Descent Ascent (OGDA) dynamics
\begin{equation}	
\label{ds:fast-OGDA}
\ddot{z} \left( s \right) + \dfrac{\alpha}{s} \dot{z} \left( s \right) + \beta \left( s \right) \dfrac{d}{ds} M \left( z \left( s \right) \right) + \dfrac{1}{2} \left( \dot{\beta} \left( s \right) + \dfrac{\alpha \beta \left( s \right)}{s} \right) M \left( z \left( s \right) \right) = 0
\end{equation}
with initial time $s_{0} >0$, $\alpha \geq 2$, $\beta \colon \left[ s_{0} , + \infty \right) \to \sR_{++}$ a continuously differentiable and nondecreasing function,  and initial conditions $z ( s_{0}) = z_{0}$ and $\dot{z} \left( s_{0} \right) = \dot{z_{0}}$, it has been shown that if
\begin{equation*}
\sup_{s \geq s_{0}} \dfrac{s \dot{\beta} \left( s \right)}{\beta \left( s \right)} \leq \alpha - 2,
\end{equation*}
then 
$$\norm{\dot{z} \left( s \right)} = \bO \left(\dfrac{1}{s} \right) \quad \mbox{and} \quad \norm{M \left( z \left( s \right) \right)} = \bO \left(\dfrac{1}{s \beta \left( s \right)}\right) \quad \mbox{as} \ s \rightarrow +\infty,$$
without providing a statement regarding the convergence of the trajectory.  On the other hand, if
\begin{equation*}
\sup_{s \geq s_{0}} \dfrac{s \dot{\beta} \left( s \right)}{\beta \left( s \right)} < \alpha - 2,
\end{equation*}
then
$$\norm{\dot{z} \left( s \right)} = o \left(\dfrac{1}{s} \right) \quad \mbox{and} \quad \norm{M \left( z \left( s \right) \right)} = o \left(\dfrac{1}{s \beta \left( s \right)}\right) \quad \mbox{as} \ s \rightarrow +\infty,$$
and $z(s)$ converges weakly to a zero of $M$ as $s \rightarrow +\infty$.  If for every $s \geq s_{0}$ it holds
\begin{equation*}
s \dot{\beta} \left( s \right) = \left( \alpha - 2 \right) \beta \left( s \right) ,
\end{equation*}
that is, $\beta \left( s \right) = \beta_{0} s^{\alpha - 2}$ for some $\beta_{0} > 0$, then the dynamical system \eqref{ds:fast-OGDA} coincides with \eqref{ds:second-order}. According to Theorem \ref{thm:second-order}, besides the rates of convergence for the velocity and the operator norm, we obtain in this borderline case that $z(s)$ converges strongly to $\proj_{\zer M} \left( z_{0} + \frac{s_{0}}{\alpha - 1} \left(\dot{z}_{0} + \beta \left( s_{0} \right) M \left( z_{0} \right) \right)\right)$ as $s \to + \infty$.  This is particularly true when $\beta$ is constant and $\alpha=2$, which is of significant interest when deriving discrete explicit counterparts of \eqref{ds:fast-OGDA}.
\end{remark}

\section{Fixed point problem and discretization}\label{sec4}

In this section, we will introduce and analyze discrete time counterparts of the first-order dynamical system  \eqref{ds:Tikhonov} in the context of identifying the fixed points of nonexpansive operators.

\subsection{Problem formulation and overview of existing numerical methods}\label{sec41}

For  $T \colon \sH \to \sH$ a \emph{nonexpansive ($1$-Lipschitz continuous)} operator, i.e.,
$$\|T(x) - T(y)\| \leq \|x-y\| \quad \forall x,y \in \sH,$$
we consider the fixed point problem
\begin{mdframed}
\begin{equation}
\label{intro:pb:fix}
\textrm{ Find } x \in \sH \textrm{ such that } x = T \left( x \right) .
\end{equation}
\end{mdframed}

We denote the set of all \emph{fixed points} of $T$ by $\Fix T \coloneq \left\lbrace x \in \sH \colon x = T \left( x \right) \right\rbrace$ and assume throughout this section  that this set is nonempty. The set $\Fix T$ is convex and closed (see, for instance, \cite[Corollary 4.24]{Bauschke-Combettes:book}).

There are many motivations for studying this particular class of problems. It has long been known that various monotone inclusions and convex optimization problems can be represented by fixed point formulations using nonexpansive operators.  For a more detailed discussion of this connection,  the reader is referred to  \cite{Bauschke-Combettes:book, Dong-Cho-He-Pardalos-Rassias}.

To establish the connection between \eqref{intro:mono} and \eqref{intro:pb:fix}, we observe that $\Fix T = \zer (\Id-T)$ and, since $T$ is nonexpansive, $\Id-T$ satisfies (see \cite[Proposition 4.11]{Bauschke-Combettes:book}) the inequality
\begin{equation}
	\label{pre:coco}
	\left\langle x - y , \left( \Id - T \right) \left( x \right) - \left( \Id - T \right) \left( y \right) \right\rangle \geq \dfrac{1}{2} \left\lVert \left( \Id - T \right) \left( x \right) - \left( \Id - T \right) \left( y \right) \right\rVert ^{2} \geq 0 \quad \forall x , y \in \sH.
\end{equation}
This shows that $\Id-T$ is a continuous and monotone operator.  In other words, the fixed point problem \eqref{intro:pb:fix} can be equivalently formulated as a monotone equation of type \eqref{intro:mono}.

These premises allow to approach \eqref{intro:pb:fix} by the dynamical system \eqref{ds:q=1} for $M\coloneq\Id-T$. This gives a convergence rate of $\mathcal{O}\left(\frac{1}{t}\right)$ for $\left\lVert x(t) - T\left(x(t)\right) \right\rVert$ and ensures the strong convergence of $x(t)$ to $\proj_{\Fix T}\left(0\right)$ as $t \to +\infty$ (see also \cite{Bot-Grad-Meier-Staudigl}).This is a significant improvement over both convergence statements for the dynamical system
\begin{equation}\label{ds:fix}
\dot x(t) + (\Id-T)(x(t))=0
\end{equation}
with the initial condition $x(t_{0}) = x_0$, as provided in \cref{thm:cocoercive} (see, also, \cite[Theorems 11 and 16]{Bot-Csetnek}).

The most standard approach used to detect a fixed point of $T$ is the \emph{\BP iteration},  which reads for every $k \geq 0$
\begin{equation}
	\label{algo:BP}
	x_{k+1} \coloneq T \left( x_{k} \right) ,
\end{equation}
where $x_{0} \in \sH$ is a given starting point.  According to the \BP fixed point theorem, if $T$ is a \emph{contraction}, namely, $T$ is Lipschitz continuous with modulus $\kappa \in \left[ 0 , 1 \right)$, then the sequence $\seq{x_{k}}_{k \geq 0}$ generated by \eqref{algo:BP} converges strongly to the unique fixed point of $T$ with a linear convergence rate.   On the other hand, if $T$ is merely nonexpansive, then \BP might diverge. To see this, it is enough to choose $T = - \Id$ and $x_{0} \neq 0$.  A closer look at this example reveals that the \BP iteration not only fails to approach a fixed point of $T$,  but also generates a sequence that does not satisfy the \emph{asymptotic regularity property}.  We say that the sequence $\seq{x_{k}}_{k \geq 0}$ satisfies the asymptotic regularity property if the difference $x_{k} - T \left( x_{k} \right) \to 0$ as $k \to + \infty$. This property is crucial for guaranteeing the convergence of the iterates, as we will see later.

To overcome the restrictive contraction assumption on $T$, Krasnosel’ski\u{\i} proposed in \cite{Krasnoselskii} applying the \BP iteration \eqref{algo:BP} to the operator $\frac{1}{2} \Id + \frac{1}{2} T$ instead of $T$.  Building on this idea,  by iteratively considering a convex combination between $\Id$ and $T$,  the so-called \emph{\KM iteration} emerges. Specifically,  let $\seq{\alpha_{k}}_{k \geq 0}$ be a sequence in $\left[ 0 , 1 \right]$. The \KM iteration is defined for every $k \geq 0$ as
\begin{equation}
	\label{algo:KM}
	x_{k+1} \coloneq \alpha_{k} x_{k} + \left( 1 - \alpha_{k} \right) T \left( x_{k} \right) ,
\end{equation}
where, again, $x_{0} \in \sH$ is a given starting point.

A fundamental step in proving the convergence of the iterates of \eqref{algo:KM} is to demonstrate that $x_{k} - T \left( x_{k} \right) \to 0$ as $k \to + \infty$. This was first established by Browder and Petryshyn in \cite{Browder-Petryshyn} in the constant case $\alpha_{k} \equiv \alpha_{0} \in \left( 0 , 1 \right)$.  The extension to nonconstant sequences was achieved by Groetsch in \cite{Groetsch}, who proved that if $\sum_{k \geq 0} \alpha_{k} \left( 1 - \alpha_{k} \right) = + \infty$, then the asymptotic regularity property is satisfied. The weak convergence of the iterates has been studied in various settings in \cite{Bauschke-Combettes:book, Borwein-Reich-Shafrir, Groetsch,Ishikawa,Reich}. Techniques based on Tikhonov regularization to improve the convergence of the iterates from weak to strong have been recently explored in  \cite{Bot-Csetnek-Meier,Bot-Meier}.

It must be emphasized that the asymptotic regularity property does not automatically provide an explicit convergence rate.  Indeed, deriving a precise rate of convergence is much more complex, and research on this subject is still limited compared to the extensive results on the convergence of iterates.
The convergence of the \KM iteration,  expressed in terms of the fixed point residual $\norm{x_{k} - T \left( x_{k} \right)}$, was proved to be $o \left(\frac{1}{\sqrt{k}} \right)$ by Baillon and Bruck in \cite{Baillon-Bruck} for the case of a constant sequence $\seq{\alpha_{k}}_{k \geq 0}$.  For a nonconstant sequence, it was initially shown to be  $\bO \left(\frac{1}{\sqrt{k}} \right)$ in \cite{Cominetti-Soto-Vaisman,Liang-Fadili-Peyre}, and later on improved to $o \left(\frac{1}{\sqrt{k}}\right)$ in \cite{Davis-Yin:2016,Matsushita,Bravo-Cominetti}.  These results align with the statements in the continuous setting by Bo\c{t} and Csetnek (\cite{Bot-Csetnek}) mentioned earlier.  Contreras and Cominetti demonstrated in \cite{Contreras-Cominetti} that in the Banach space setting the lower bound of the  \KM iteration is $\bO \left(\frac{1}{\sqrt{k}} \right)$, whereas Maul\'en, Fierro, and Peypouquet proved in \cite{Maulen-Fierro-Peypouquet} that the rate of convergence of the fixed point residual of a general \emph{inertial \KM algorithm} is $o \left(\frac{1}{\sqrt{k}}\right)$.

In the following, we shall see that both \BP and \KM iterations are obtained by temporal discretizations of the same dynamical system \eqref{ds:fix}. 
The difference lies in the chosen step size. Specifically, we consider an explicit finite-difference scheme with step sizes $h_{k} > 0$
\begin{equation*}
	\dfrac{x_{k+1} - x_{k}}{h_{k}} + \left( \Id - T \right) \left( x_{k} \right) = 0 \quad \forall k \geq 0,
\end{equation*}
which is equivalent to
\begin{align}\label{algo:forward} 
	x_{k+1} = x_{k} - h_{k} \left( \Id - T \right) \left( x_{k} \right) = \left( 1 - h_{k} \right) x_{k} + h_{k} T \left( x_{k} \right) \quad \forall k \geq 0.
\end{align}
From this,  it is clear that \BP process is simply the scheme above with $h_{k} \equiv 1$, while taking $h_{k} = 1 - \alpha_{k}$ leads to the \KM iteration.

By considering convex combinations with a fixed anchor point $v \in \sH$, one obtains the \emph{Halpern iteration} \cite{Halpern}
\begin{equation}
	\label{algo:H}
	x_{k+1} \coloneq \alpha_{k} v + \left( 1 - \alpha_{k} \right) T \left( x_{k} \right) \quad \forall k \geq 0.
\end{equation}
The asymptotic regularity property of this iterative scheme has been studied in \cite{Wittmann,Xu:02}.  This method has recently been gaining increasing attention \cite{Lieder,Qi-Xu,Park-Ryu,Suh-Park-Ryu,He-Xu-Dong-Mei}. One reason for its attractiveness, besides guaranteeing strong convergence of the iteration, is its acceleration property, which has been demonstrated recently. Sabach and Shtern proved in \cite{Sabach-Shtern} for a general form of the Halpern iteration that the rate of convergence of the fixed point residual is of $\bO \left( \frac{1}{k} \right)$. Lieder obtained in \cite{Lieder} the same rate of convrgence  for the Halpern iteration \eqref{algo:H} with $\alpha_{k} \coloneq \frac{1}{k+2}$ for every $k \geq 0$,  while Park and Ryu showed in \cite{Park-Ryu} that the $\bO \left( \frac{1}{k} \right)$ convergence rate is optimal within a specific family of fixed point iterations.

This rate of convergence is particularly attractive when compared to the closely related \KM method.  Building on these insights, researchers have started combining this anchoring technique with other methods to accelerate convergence.  Notable contributions in this direction include works by Yoon and Ryu \cite{Yoon-Ryu:21},  Lee and Kim \cite{Lee-Kim},  and Tran-Dinh \cite{Tran-Dinh} and his coauthor Luo \cite{Tran-Dinh-Luo}.

Recently,  in \cite{Bot-Nguyen} (see also \cite{Bot-Csetnek-Nguyen}), explicit numerical methods with Nesterov momentum and correcting operator terms have been introduced.  These methods exhibit a convergence rate for the fixed point residual of $\bO \left( \frac{1}{k} \right)$, while the sequence of iterates weakly converges to a fixed point of the operators. These numerical methods were derived through a subtle discretization of the Fast Optimistic Gradient Descent Ascent (OGDA) dynamics \eqref{ds:fast-OGDA:intro}.

The main aim of this section is to show that the Halpern iteration serves as the discrete counterpart to the first-order dynamical system \eqref{ds:Tikhonov} with an anchor point $v \in \sH$ (see \eqref{ds:anchor}).  Additionally,  we will derive convergence rates for the vanishing of both the discrete velocity and the fixed point residual for various choices of the sequence $\{\alpha_k\}_{k \geq 0}$, and establish the strong convergence of the iterates to  $\proj_{\Fix T} \left( v \right)$.  Our approach will utilize novel proof techniques that are inspired by those used in the continuous time analysis.

\subsection{Convergence results of the discretization}\label{sec:algo}

To construct a numerical method,  we consider the dynamics given by \eqref{ds:Tikhonov} where an anchor point $v \in \sH$ is also added (or the dynamics \eqref{ds:anchor} with $\beta(\cdot) \equiv 1$)
\begin{equation*}
	\dot{x} \left( t \right) + \left( \Id - T \right) \left( x \left( t \right) \right) + \varepsilon \left( t \right) \left( x \left( t \right) - v \right) = 0,
\end{equation*}
with initial condition $x(t_{0}) = x_0 \in \sH$.

The explicit finite-difference scheme with step sizes $1$ gives
\begin{equation*}
x_{k+1} - x_{k} + \left( \Id - T \right) \left( x_{k} \right) + \varepsilon_{k} \left( x_{k+1} - v \right) = 0 \quad \forall k \geq 0
\end{equation*}
or, equivalently,
\begin{equation*}
x_{k+1} \coloneq \dfrac{\varepsilon_{k}}{1 + \varepsilon_{k}} v + \dfrac{1}{1 + \varepsilon_{k}} T \left( x_{k} \right) \quad \forall k \geq 0.
\end{equation*}
Here $\seq{\varepsilon_{k}}_{k \geq 0}$ is a positive, so-called \emph{regularization sequence} which mimics the properties \eqref{cond:eps} of the regularization function $\varepsilon \left( t \right)$ in the continuous case, and is therefore required to fulfill
\begin{equation}
\label{cond:var-k}
\lim_{k \to + \infty} \varepsilon_{k} = 0 \textrm{ and } \sum_{k \geq 0} \varepsilon_{k} = + \infty .
\end{equation}

For convenience,  we write for every $k \geq 0$
\begin{equation*}
\alpha_{k} \coloneq \dfrac{\varepsilon_{k}}{1 + \varepsilon_{k}} \in \left( 0 , 1 \right) \Leftrightarrow 1 - \alpha_{k} = \dfrac{1}{1 + \varepsilon_{k}} .
\end{equation*}
Then \eqref{cond:var-k} translate into the following condition.
\begin{mdframed}
The sequence $\seq{\alpha_{k}}_{k \geq 0} \subseteq (0,1)$ satisfies
\begin{equation}
	\label{cond:alpha-k}
	\lim_{k \to + \infty} \alpha_{k} = 0 \textrm{ and } \sum_{k \geq 0} \alpha_{k} = + \infty .
\end{equation}
\end{mdframed}

With this new regularization sequence, the update rule turns out to be nothing else but the Halpern fixed point iteration (\cite{Halpern}).
\begin{mdframed}
Let $v, x_{0} \in \sH$. For every $k \geq 0$ set
\begin{equation}
\label{algo:Halpern}
x_{k+1} \coloneq \alpha_{k} v + \left( 1 - \alpha_{k} \right) T \left( x_{k} \right) .
\end{equation}
\end{mdframed}

Next, we provide some preliminary estimates that will useful in the convergence analysis of \eqref{algo:Halpern}.

\begin{lemma}
\label{lem:numT}
Let $\seq{x_{k}} _{k \geq 0}$ be the sequences generated by the iterative scheme \eqref{algo:Halpern} and $\xi \in \Fix T$.  Then, the following inequalities are satisfied for every $k \geq 0$:
\begin{align}
\norm{x_{k} - \xi} & \leq D_{0} \left( x_{0} , v , \xi \right) , \label{numT:bound:ite} \\
\norm{T \left( x_{k} \right) - v} & \leq 2D_{0} \left( x_{0} , v , \xi \right) , \label{numT:bound:oper} \\
\dfrac{1}{2} \norm{x_{k+1} - \xi}^{2} & \leq \dfrac{1}{2} \left( 1 - \alpha_{k} \right) \norm{x_{k} -  \xi}^{2} + \alpha_{k} \scal{x_{k} - \xi,  v - \xi} \nonumber \\
& \quad + \alpha_{k} D_{0} ( x_{0} , v , \xi) \left(\frac{3}{2} D_{0} \left( x_{0} , v , \xi \right) \alpha_{k} + \norm{x_{k} - T \left( x_{k} \right)} \right) \label{numT:inq} .
\end{align}
\end{lemma}

\begin{proof}
We use an induction argument to prove the first inequality.  Obviously \eqref{numT:bound:ite} holds when $k = 0$.  Suppose that \eqref{numT:bound:ite} holds for $k = k_{0} \geq 0$.  It follows from \eqref{algo:Halpern} and the triangle inequality that
\begin{align*}
\norm{x_{k_{0}+1} - \xi} & = \norm{\alpha_{k_{0}} v + \left( 1 - \alpha_{k_{0}} \right) T \left( x_{k_{0}} \right) - \xi} \nonumber \\
& \leq \alpha_{k_{0}} \norm{v - \xi} + \left( 1 - \alpha_{k_{0}} \right) \norm{T \left( x_{k_{0}} \right) - \xi} \nonumber \\
& = \alpha_{k_{0}} \norm{v - \xi} + \left( 1 - \alpha_{k_{0}} \right) \norm{T \left( x_{k_{0}} \right) - T \left(\xi \right)} \nonumber \\
& \leq \alpha_{k_{0}}  \norm{v - \xi} + \left( 1 - \alpha_{k_{0}} \right) \norm{ x_{k_{0}} - \xi} \leq \max \seq{\norm{x_{0} - \xi} , \norm{v - \xi}} = D_{0} \left( x_{0} , v , \xi \right).
\end{align*}
This means that \eqref{numT:bound:ite} is satisfied for $k=k_{0}+1$,  so it is satisfied for every $k \geq 0$. The inequality \eqref{numT:bound:oper} follows immediately from this,  since for every $k \geq 0$ we have
\begin{align*}
\norm{T \left( x_{k} \right) - v} & \leq \norm{T \left( x_{k} \right) - T \left(\xi\right)} + \norm{T \left(\xi \right) - v} \nonumber \\
& \leq \norm{x_{k} - \xi} + \norm{T \left( \xi \right) - v} = \norm{x_{k} - \xi} + \norm{\xi - v} \leq 2 D_{0} \left( x_{0} , v , \xi \right).
\end{align*}
Now,  let $k \geq 0$ be fixed. By the definition of $x_{k+1}$ in \eqref{algo:Halpern}, we have
\begin{align}
\dfrac{1}{2} \norm{x_{k+1} - \xi}^{2} & = \dfrac{1}{2} \norm{\alpha_{k} v + \left( 1 - \alpha_{k} \right) T \left( x_{k} \right) - \xi}^{2} = \dfrac{1}{2} \norm{\alpha_{k} (v - \xi) + \left( 1 - \alpha_{k} \right) (T \left( x_{k} \right) - \xi)}^{2} \nonumber \\
& = \dfrac{1}{2} \left( 1 - \alpha_{k} \right) ^{2} \norm{T \left( x_{k} \right) - \xi}^{2} + \left( 1 - \alpha_{k} \right) \alpha_{k} \scal{T \left( x_{k} \right) - \xi , v - \xi} + \dfrac{1}{2} \alpha_{k}^{2} \norm{v - \xi}^{2} \nonumber \\
& = \dfrac{1}{2} \left( 1 - \alpha_{k} \right) ^{2} \norm{T \left( x_{k} \right) - T \left( \xi \right)}^{2} + \left( 1 - \alpha_{k} \right) \alpha_{k} \scal{x_{k} - \xi, v - \xi} \nonumber \\
& \quad - \left( 1 - \alpha_{k} \right) \alpha_{k} \scal{x_{k} - T \left( x_{k} \right) , v - \xi} + \dfrac{1}{2} \alpha_{k}^{2} \norm{v - \xi}^{2}.
\label{intermed413}
\end{align}
Using successively the nonexpansiveness of $T$, the Cauchy-Schwarz inequality, and the fact that $\alpha_{k} \in \left( 0 , 1 \right)$, we infer
\begin{align*}
& \dfrac{1}{2} \left( 1 - \alpha_{k} \right) ^{2} \norm{T \left( x_{k} \right) - T \left(\xi \right)}^{2} - \alpha_{k}^{2} \scal{x_{k} - \xi , v - \xi} - \left( 1 - \alpha_{k} \right) \alpha_{k} \scal{x_{k} - T \left( x_{k} \right) , v - \xi} + \dfrac{1}{2} \alpha_{k}^{2} \norm{v - \xi}^{2} \nonumber \\
\leq \: 	& \dfrac{1}{2} \left( 1 - \alpha_{k} \right) ^{2} \norm{x_{k} -  \xi}^{2} + \alpha_{k}^{2} \norm{x_{k} - \xi} \norm{v - \xi} + \left( 1 - \alpha_{k} \right) \alpha_{k} \norm{x_{k} - T \left( x_{k} \right)} \norm{v - \xi} + \dfrac{1}{2} \alpha_{k}^{2} \norm{v - \xi}^{2} \nonumber \\
\leq \: 	& \dfrac{1}{2} \left( 1 - \alpha_{k} \right) \norm{x_{k} -  \xi}^{2} + \alpha_{k}^{2} \norm{x_{k} - \xi} \norm{v - \xi} + \alpha_{k} \norm{x_{k} - T \left( x_{k} \right)} \norm{v - \xi} + \dfrac{1}{2} \alpha_{k}^{2} \norm{v - \xi}^{2}\nonumber \\
\leq \: 	& \dfrac{1}{2} \left( 1 - \alpha_{k} \right) \norm{x_{k} -  \xi}^{2} + \frac{3}{2} D_{0}^{2} \left( x_{0} , v , \xi \right)  \alpha_{k}^{2} + D_{0} \left( x_{0} , v , \xi \right) \alpha_{k} \norm{x_{k} - T \left( x_{k} \right)} ,
\end{align*}
and the announced statement follows by plugging this estimate into \eqref{intermed413}.
\end{proof}

In analogy to the analysis in the continuous time regime, we provide a result that allows us to derive convergence rates for the discrete velocity and the fixed point residual of the operator, as well as ensuring strong convergence for the sequence of iterates.

In the following, for $\left( x_{0} , v \right) \in \sH \times \sH$,  we set
\begin{equation*}
D_{1} \left( x_{0} , v; \Fix T \right) \coloneq \max \seq{\dist_{\Fix T} \left( x_{0} \right), \dist_{\Fix T}(v)}.
\end{equation*}

\begin{proposition}
\label{prop:numT}
Let $\seq{x_{k}} _{k \geq 0}$ be the sequences generated by the iterative scheme \eqref{algo:Halpern}.  Then, the following statement are true:
\begin{enumerate}
\item 
\label{prop:numT:rate}
for every $k \geq 1$ it holds
\begin{align}
\norm{x_{k+1} - x_{k}} & \leq \left( 1 - \alpha_{k} \right) \norm{x_{k} - x_{k-1}} + 2 D_{1} \left( x_{0} , v; \Fix T \right)  \abs{\alpha_{k} - \alpha_{k-1}} , \label{numT:vel} \\
\norm{x_{k+1} - T \left( x_{k+1} \right)} & \leq \left( 1 - \alpha_{k} \right) \norm{x_{k+1} - x_{k}} + 2 D_{1} \left( x_{0} , v; \Fix T \right) \alpha_{k} \label{numT:opn} ;
\end{align}
\item 
\label{prop:numT:conv}
if $\norm{x_{k} - T \left( x_{k} \right)} \to 0$ as $k \to + \infty$, then $\seq{x_{k}} _{k \geq 0}$ converges strongly to the fixed point of $T$ closest to $v$,  $\xi_{*} \coloneq \proj_{\Fix T} \left( v \right)$, as $k \to + \infty$.
\end{enumerate}
\end{proposition}
\begin{proof}
\begin{enumerate}[wide, labelwidth=!, labelindent=0pt]
\item Let $k \geq 1$ be fixed.
According to \eqref{algo:Halpern}, and by using the triangle inequality, we can derive that
\begin{align*}
\norm{x_{k+1} - x_{k}} & = \norm{\alpha_{k} v + \left( 1 - \alpha_{k} \right) T \left( x_{k} \right) - \alpha_{k-1} v - \left( 1 - \alpha_{k-1} \right) T \left( x_{k-1} \right)} \nonumber \\
& = \norm{\left( 1 - \alpha_{k} \right) \left( T \left( x_{k} \right) - T \left( x_{k-1} \right) \right) + \left( \alpha_{k} - \alpha_{k-1} \right) \left( v - T \left( x_{k-1} \right) \right)} \nonumber \\
& \leq \left( 1 - \alpha_{k} \right) \norm{T \left( x_{k} \right) - T \left( x_{k-1} \right)} + \abs{\alpha_{k} - \alpha_{k-1}} \norm{T \left( x_{k-1} \right) - v} \nonumber \\
& \leq \left( 1 - \alpha_{k} \right) \norm{x_{k} - x_{k-1}} + 2 D_{0} \left( x_{0} , v,  \xi \right) \abs{\alpha_{k} - \alpha_{k-1}} ,
\end{align*}
which yields further
\begin{align*}
\norm{x_{k+1} - T \left( x_{k+1} \right)} & = \norm{\alpha_{k} v + \left( 1 - \alpha_{k} \right) T \left( x_{k} \right) - T \left( x_{k+1} \right)} \nonumber \\
& \leq \left( 1 - \alpha_{k} \right) \norm{T \left( x_{k} \right) - T \left( x_{k+1} \right)} + \alpha_{k} \norm{v - T \left( x_{k+1} \right)} \nonumber \\
& \leq \left( 1 - \alpha_{k} \right) \norm{x_{k+1} - x_{k}} + 2 D_{0} \left( x_{0} , v,  \xi \right) \alpha_{k} .
\end{align*}	
The statements follow after taking the infimum over all $\xi \in  \Fix T$.

\item 
In the view of \eqref{numT:inq}, we will prove that
\begin{equation}
\label{numT:limsup}
\limsup_{k \to + \infty} \scal{x_{k} - \xi_{*} , v - \xi_{*}} \leq 0 .
\end{equation}
Recall that the sequence of iterates $\seq{x_{k}} _{k \geq 0}$ is bounded, see \cref{lem:numT}, therefore there exists a subsequence $\seq{x_{k_{n}}} _{n \geq 0}$ of $\seq{x_{k}} _{k \geq 0}$  that converges weakly to $\widehat{x}$ as $n \to + \infty$ and
\begin{equation*}
\limsup_{k \to + \infty} \scal{x_{k} - \xi_{*} , v - \xi_{*}} = \lim_{n \to + \infty} \scal{x_{k_{n}} - \xi_{*} , v - \xi_{*}} = \scal{\widehat{x} - \xi_{*} , v - \xi_{*}} .
\end{equation*}
Since,  according to the assumption, $\norm{x_{k_{n}} - T \left( x_{k_{n}} \right)} \to 0$ as $n \to + \infty$, 
the Browder’s demiclosedness principle (\cite[Theorem 4.27]{Bauschke-Combettes:book}) yields $\widehat{x} \in \Fix T$. Therefore, since $\xi_{*} \coloneq \proj_{\Fix T} \left( v \right)$, we have from the variational characterization of the projection that (\cite[Theorem 3.16]{Bauschke-Combettes:book})
\begin{equation*}
\scal{\widehat{x} - \xi_{*} , v - \xi_{*}} \leq 0 .
\end{equation*}
This yields \eqref{numT:limsup}, which allows us to deduce further that
\begin{align*}
	& \limsup_{k \to + \infty} \left( \scal{x_{k} - \xi_{*} , v - \xi_{*}} +  D_{0} ( x_{0} , v , \xi) \left(\frac{3}{2}  D_{0} ( x_{0} , v , \xi) \alpha_{k} + \norm{x_{k} - T \left( x_{k} \right)} \right)  \right) \nonumber \\
	= \ 	& \limsup_{k \to + \infty} \scal{x_{k} - \xi_{*} , v - \xi_{*}} \leq 0 .
\end{align*}
We are therefore in a position to use \cref{lem:lims-0} with
\begin{equation*}
	a_{k} \coloneq \dfrac{1}{2} \norm{x_{k} -  \xi_{*}}^{2} , \quad
	b_{k} \coloneq \scal{x_{k} - \xi_{*} , v - \xi_{*}} + D_{0} ( x_{0} , v , \xi) \left(\frac{3}{2}  D_{0} ( x_{0} , v , \xi) \alpha_{k} + \norm{x_{k} - T \left( x_{k} \right)} \right) , \quad
	d_{k} \coloneq 0,
\end{equation*} 
which yields the conclusion.
\qedhere
\end{enumerate}
\end{proof}

Relying on the estimates \eqref{numT:vel}-\eqref{numT:opn}, we can formulate further conditions on the sequence $\seq{\alpha_{k}}_{k \geq 0}$ which, in the light of \cref{prop:numT} (ii),  lead to the strong convergence of the sequence generated by the Halpern iteration \eqref{algo:Halpern}.

\begin{theorem}
\label{thm:Halpern-sc}
Let $\seq{x_{k}} _{k \geq 0}$ be the sequence generated by the iterative scheme \eqref{algo:Halpern}.
Assume that, in addition to \eqref{cond:alpha-k},  the sequence $\seq{\alpha_{k}}_{k \geq 0}$  satisfies either
\begin{equation}
\label{alpha-BV}
\mysum_{k \geq 0} \abs{\alpha_{k+1} - \alpha_{k}} < + \infty
\end{equation}
or
\begin{equation}
\label{alpha-lim}	
\lim\limits_{k \to + \infty} \dfrac{\abs{\alpha_{k+1} - \alpha_{k}}}{\alpha_{k+1}} = 0.
\end{equation}
Then, $\norm{x_{k} - T \left( x_{k} \right)} \to 0$ as $k \to + \infty$, and therefore the sequence $\seq{x_{k}} _{k \geq 0}$ strongly converges to $\xi_{*} \coloneq \proj_{\Fix T} \left( v \right)$ as $k \to + \infty$.
\end{theorem}
\begin{proof}
Let $\xi \in \Fix T$.  The statement follows by applying \cref{lem:lims-0} to the inequality \eqref{numT:vel} in two different ways, depending on the additional assumptions for $\seq{\alpha_{k}}_{k \geq 0}$.  Specifically,  for every $k \geq 1$ we set $a_{k} \coloneq \norm{x_{k} - x_{k-1}}$. If  \eqref{alpha-BV} holds, we choose $b_{k} \coloneq 0$ and $d_{k} \coloneq D_{1} \left( x_{0} , v; \Fix T \right) \abs{\alpha_{k} - \alpha_{k-1}}$, while if \eqref{alpha-lim} holds we choose $b_{k} \coloneq D_{1} \left( x_{0} , v; \Fix T \right) \frac{\abs{\alpha_{k} - \alpha_{k-1}}}{\alpha_{k}}$ and $d_{k} \coloneq 0$, repsectively.  In both cases we can conclude that $\norm{x_{k} - x_{k-1}} \to 0$ as $k \to + \infty$, which immediately gives $\norm{x_{k} - T \left( x_{k} \right)} \to 0$ as $k \to + \infty$ due to the inequality \eqref{numT:opn} and the assumption $\lim_{k \to + \infty} \alpha_{k} = 0$.   \cref{prop:numT} (ii) gives the desired conclusion.
\end{proof}

\begin{remark}\label{rem44}
The condition \eqref{alpha-BV} is  due to Wittmann \cite{Wittmann},  while \eqref{alpha-lim} was proposed by Xu in \cite{Xu:02}.  As pointed out in \cite[Examples 3.1 and 3.2]{Xu:03}, these two conditions are not comparable in the sense that neither of them implies the other.  Note also that the two conditions above can be interpreted as discrete versions of {\eqref{eps:BV} and \eqref{eps:lim}} respectively.
\end{remark}

To illustrate the theoretical results and to derive explicit convergence rates, we consider first the case $\alpha_{k} \coloneq \frac{\alpha-1}{\left( k+b \right) ^{q}}$ for every $k \geq 0$, where $0 < q < 1$, $\alpha > 1$ and $b >0$.
\begin{theorem}
\label{thm:algo:0<q<1}
Let $0 < q < 1$, $\alpha > 1$, $b >0$ with $\alpha -1 < b^q$.  For given $x_{0}, v \in \sH$, let $\seq{x_{k}} _{k \geq 0}$ be the sequence generated by
\begin{equation*}
x_{k+1} \coloneq \dfrac{\alpha - 1}{\left( k+b \right) ^{q}} v + \left( 1 - \dfrac{\alpha - 1}{\left( k+b \right) ^{q}} \right) T \left( x_{k} \right) \quad \forall k \geq 0.
\end{equation*}
Then,  the following statement are true:
\begin{enumerate}
\item 
{as $k \to + \infty$} it holds
\begin{align*}
\norm{x_{k+1} - x_{k}}  & \ {= \bO \left( \dfrac{1}{k+b+1} \right)} ; \nonumber \\
\norm{x_{k+1} - T \left( x_{k+1} \right)} & \ {
	= \bO \left( \dfrac{1}{\left( k+b \right) ^{q}} \right) } ;
\end{align*}
\item the sequence $\seq{x_{k}} _{k \geq 0}$ converges strongly to $\xi_{*} \coloneq \proj_{\Fix T} \left( v \right)$ as $k \to + \infty$.
\end{enumerate}
\end{theorem}
\begin{proof}
Since $0 < q < 1$, the function $t \mapsto t^{q}$ is nondecreasing and concave on $\left( 0 , + \infty \right)$. We thus have $0 \leq \left( k+b \right) ^{q} - \left( k+b-1 \right) ^{q} \leq q \left( k+b-1 \right)^{q-1}$ for every $k \geq 1$. Therefore, according to \eqref{numT:vel} in  \cref{prop:numT} (i),  we have for every $k \geq 1$ that
\begin{align*}
\norm{x_{k+1} - x_{k}} & \leq \left( 1 - \dfrac{\alpha - 1}{\left( k+b \right) ^{q}} \right) \norm{x_{k} - x_{k-1}} + 2 D_{1} \left( x_{0} , v; \Fix T \right) \abs{\dfrac{\alpha - 1}{\left( k+b \right) ^{q}} - \dfrac{\alpha - 1}{\left( k+b-1 \right) ^{q}}} \nonumber \\
& \leq \left( 1 - \dfrac{\alpha - 1}{\left( k+b \right) ^{q}} \right) \norm{x_{k} - x_{k-1}} + \dfrac{2q \left( \alpha - 1 \right) D_{1} \left( x_{0} , v; \Fix T \right)}{\left( k+b \right) ^{q} \left( k+b-1 \right)} ,
\end{align*}
and the statements follow from \cref{lem:rate:0<q<1} and the inequality \eqref{numT:opn}, respectively, after taking the infmum over all  $\xi \in \Fix T$.  As $\norm{x_{k} - T \left( x_{k} \right)} \to 0$ as $k \to +\infty$,  the strong convergence of the sequence of iterates follows from   \cref{prop:numT} (ii).
\end{proof}

There are several reasons for focusing on the case $q=1$.  One reason, of course, is to derive the discrete counterpart of the Theorem \ref{thm:Tikh:q=1} which gave the best rates of convergence in the continuous setting. Additionally, our approach in this critical case offers an alternative proof for the result by Lieder in \cite{Lieder}, which corresponds to the scenario where $\alpha_{k} = \frac{1}{k+2}$ for every $k \geq 0$.  Moreover, our method allows $v$ to vary within $\sH$ instead of fixing it as $x_{0}$.

\begin{theorem}
\label{thm:algo:q=1}
Let $\alpha > 1$.  For given $x_{0}, v \in \sH$, let $\seq{x_{k}} _{k \geq 0}$ be the sequence generated by
\begin{equation*}
x_{k+1} \coloneq \dfrac{\alpha - 1}{k + \alpha} v + \left( 1 - \dfrac{\alpha - 1}{k + \alpha} \right) T \left( x_{k} \right) \quad \forall k \geq 0.
\end{equation*}
Then,  the following statements are true:
\begin{enumerate}
\item 
as $k \to + \infty$ it holds
\begin{align*}
\norm{x_{k+1} - x_{k}} & = \bO \left( \dfrac{1}{k^{\min \left\lbrace \alpha - 1 , 1 \right\rbrace}} \right) , \\
\norm{x_{k+1} - T \left( x_{k+1} \right)} & =  \bO \left( \dfrac{1}{k^{\min \left\lbrace \alpha - 1 , 1 \right\rbrace}} \right) ;
\end{align*}
	
\item 
the sequence $\seq{x_{k}} _{k \geq 0}$ converges strongly to $\xi_{*} \coloneq \proj_{\Fix T} \left( v \right)$ as $k \to + \infty$.
\end{enumerate}
\end{theorem}

\begin{proof}
We will focus on rate of convergence for the fixed point residual.  The strong convergence of the sequence $\{x_k \}_{k \geq 0}$ is a direct consequence of  \cref{prop:numT} (ii).
	
For every $k \geq 1$, we have from \eqref{numT:vel} in   \cref{prop:numT} {(i)} that
\begin{align*}
\norm{x_{k+1} - x_{k}} & \leq \left( 1 - \dfrac{\alpha - 1}{k + \alpha} \right) \norm{x_{k} - x_{k-1}} + 2 D_{1} \left( x_{0} , v; \Fix T \right) \abs{\dfrac{\alpha - 1}{k + \alpha} - \dfrac{\alpha - 1}{k + \alpha - 1}} \nonumber \\
& \leq \left( 1 - \dfrac{\alpha - 1}{k + \alpha} \right) \norm{x_{k} - x_{k-1}} + \dfrac{2 \left( \alpha - 1 \right) D_{1} \left( x_{0} , v; \Fix T \right)}{\left( k + \alpha \right) \left( k + \alpha - 1 \right)} .
\end{align*}

For $\alpha \neq 2$, the estimates for the discrete velocity and the fixed point residual follow from \cref{lem:rate:q=1} and \eqref{numT:opn}.   For $\alpha = 2$,  \cref{lem:rate:q=1}
would give rates of convergence of $\bO \left( \frac{\log \left( k \right)}{k} \right)$ as $k \rightarrow +\infty$. In the following we will see that we can get rid of the $\log \left( k \right)$ term by a different approach, inspired by the continuous time analysis.  To fix the setting, the iteration \eqref{algo:Halpern} is in this case for every $k \geq 0$
\begin{equation}
\label{algo:alpha=2}
x_{k+1} \coloneq \dfrac{1}{k+2} v + \dfrac{k+1}{k+2} T \left( x_{k} \right),
\end{equation}
which also implies
\begin{equation}
\label{algo:Lieder}
T \left( x_{k} \right) - x_{k+1} = \dfrac{1}{k+1} \left( x_{k+1} - v \right) .
\end{equation}
For every $k \geq 1$, we consider the difference
\begin{align}
& \dfrac{1}{2} \left( k+2 \right) ^{2} \norm{x_{k+1} - x_{k}}^{2} - \dfrac{1}{2} \left( k+1 \right) ^{2} \norm{x_{k} - x_{k-1}}^{2} \nonumber \\
= \ & \scal{ \left( k+1 \right) \left( x_{k} - x_{k-1} \right) , \left( k+2 \right) \left( x_{k+1} - x_{k} \right) - \left( k+1 \right) \left( x_{k} - x_{k-1} \right) } \nonumber \\
& + \dfrac{1}{2} \norm{\left( k+2 \right) \left( x_{k+1} - x_{k} \right) - \left( k+1 \right) \left( x_{k} - x_{k-1} \right)}^{2} . \label{numT:pre}
\end{align}
By making use of the formulas \eqref{algo:alpha=2} and \eqref{algo:Lieder}, respectively, we observe that for every $k \geq 1$
\begin{align*}
& \left( k+2 \right) \left( x_{k+1} - x_{k} \right) - \left( k+1 \right) \left( x_{k} - x_{k-1} \right) \nonumber \\
= \ 	& v + \left( k+1 \right) T \left( x_{k} \right) - \left( k+2 \right) x_{k} - v - k T \left( x_{k-1} \right) + \left( k+1 \right) x_{k-1} \nonumber \\
= \ 	& - \left( k+1 \right) \Bigl( \left(x_{k} - T \left( x_{k} \right) \right) - \left(x_{k-1} - T \left( x_{k-1} \right) \right) \Bigr) - x_{k} + T \left( x_{k-1} \right) \nonumber \\
= \ 	& - \left( k+1 \right) \Bigl( \left( x_{k} - T \left( x_{k} \right) \right)- \left( x_{k-1} - T \left( x_{k-1} \right) \right) \Bigr) + \dfrac{1}{k} \left( x_{k} - v \right) .
\end{align*}
Plugging this relation into \eqref{numT:pre}, we deduce that for every $k \geq 1$
\begin{align*}
& \dfrac{1}{2} \left( k+2 \right) ^{2} \norm{x_{k+1} - x_{k}}^{2} - \dfrac{1}{2} \left( k+1 \right) ^{2} \norm{x_{k} - x_{k-1}}^{2} \nonumber \\
= \ & \scal{ \left( k+1 \right) \left( x_{k} - x_{k-1} \right) , - \left( k+1 \right) \Bigl( \left( x_{k} - T \left( x_{k} \right) \right) - \left(x_{k-1} - T \left( x_{k-1} \right) \right) \Bigr) + \dfrac{1}{k} \left( x_{k} - v \right) } \nonumber \\
& \ + \dfrac{1}{2} \norm{- \left( k+1 \right) \Bigl( \left(x_{k} - T \left( x_{k} \right) \right) - \left( x_{k-1} - T \left( x_{k-1} \right) \right) \Bigr) + \dfrac{1}{k} (x_{k}-v)}^{2} \nonumber \\
= \ & - \left( k+1 \right) ^{2} \scal{ x_{k} - x_{k-1} , \left( x_{k} - T \left( x_{k} \right) \right) - \left( x_{k-1} - T \left( x_{k-1} \right) \right)} + \dfrac{k+1}{k} \scal{ x_{k} - x_{k-1} , x_{k} - v } \nonumber \\
& + \dfrac{1}{2} \left( k+1 \right) ^{2} \norm{\left( x_{k} - T \left( x_{k} \right) \right) - \left( x_{k-1} - T \left( x_{k-1} \right) \right)}^{2} \nonumber \\
& - \dfrac{k+1}{k} \scal{ \left(x_{k} - T \left( x_{k} \right) \right) - \left( x_{k-1} - T \left( x_{k-1} \right) \right) , x_{k} - v }  + \dfrac{1}{2k^{2}} \norm{x_{k} - v}^{2} \nonumber \\
\leq \ & \dfrac{k+1}{k} \scal{ T \left( x_{k} \right) - T \left( x_{k-1} \right) , x_{k} - v } + \dfrac{1}{2k^{2}} \norm{x_{k} - v}^{2} ,
\end{align*}
where the last inequality comes from the $\frac{1}{2}$-cocoercivity of $\Id - T$, see \eqref{pre:coco}.  Furthermore,  by taking into consideration the definition of iteration $x_{k}$,  we have for every $k \geq 1$
\begin{align*}	
\dfrac{k+1}{k} \scal{ T \left( x_{k} \right) - T \left( x_{k-1} \right) , x_{k} - v } 
& = \scal{ T \left( x_{k} \right) - T \left( x_{k-1} \right) , T \left( x_{k-1} \right) - v } \nonumber \\
& = - \dfrac{1}{2} \norm{T \left( x_{k} \right) - T \left( x_{k-1} \right)}^{2} - \dfrac{1}{2} \norm{T \left( x_{k-1} \right) - v}^{2} + \dfrac{1}{2} \norm{T \left( x_{k} \right) - v}^{2} \nonumber \\
& \leq - \dfrac{1}{2} \norm{T \left( x_{k-1} \right) - v}^{2} + \dfrac{1}{2} \norm{T \left( x_{k} \right) - v}^{2}.
\end{align*}
Therefore, by also ultilizing \cref{lem:numT}, we conclude that for every $k \geq 1$
\begin{align*}
\dfrac{1}{2} \left( k+2 \right) ^{2} \norm{x_{k+1} - x_{k}}^{2} - \dfrac{1}{2} \left( k+1 \right) ^{2} \norm{x_{k} - x_{k-1}}^{2} 
& \leq \dfrac{1}{2} \norm{T \left( x_{k} \right) - v}^{2} - \dfrac{1}{2} \norm{T \left( x_{k-1} \right) - v}^{2} + \dfrac{1}{2k^{2}} \norm{x_{k} - v}^{2} \nonumber \\
& \leq \dfrac{1}{2} \norm{T \left( x_{k} \right) - v}^{2} - \dfrac{1}{2} \norm{T \left( x_{k-1} \right) - v}^{2} + \dfrac{2D_{1} \left( x_{0} , v; \Fix T \right)^2 }{k^{2}} .
\end{align*}
By telescoping,  we obtain for every $k \geq 1$
\begin{align*}
\dfrac{1}{2} \left( k+2 \right) ^{2} \norm{x_{k+1} - x_{k}}^{2} - 2 \norm{x_{1} - x_{0}}^{2}
& \leq \dfrac{1}{2} \norm{T \left( x_{k} \right) - v}^{2} - \dfrac{1}{2} \norm{T \left( x_{0} \right) - v}^{2} + 2D_{1} \left( x_{0} , v; \Fix T \right)^2  \sum_{i=1}^{k} \dfrac{1}{i^{2}} \nonumber \\
& \leq \dfrac{1}{2} \norm{T \left( x_{k} \right) - v}^{2} - \dfrac{1}{2} \norm{T \left( x_{0} \right) - v}^{2} + \dfrac{\pi^2 D_{1} \left( x_{0} , v; \Fix T \right)^2 }{3}.
\end{align*}
Using that $2 \left( x_{1} - x_{0} \right) = v + T \left( x_{0} \right) - 2x_{0}$,  thiy yields for every $k \geq 1$
\begin{align}
\left( k+2 \right) ^{2} \norm{x_{k+1} - x_{k}}^{2} & \leq \norm{T \left( x_{k} \right) - v}^{2} + \dfrac{2\pi^2 D_{1} \left( x_{0} , v; \Fix T \right)^2 }{3} + \norm{v + T \left( x_{0} \right) - 2x_{0}}^{2} - \norm{T \left( x_{0} \right) - v}^{2} \nonumber \\
& \leq \frac{2(6 +\pi^2)}{3} D_{1} \left( x_{0} , v; \Fix T \right)^2 + \norm{v + T \left( x_{0} \right) - 2x_{0}}^{2} - \norm{T \left( x_{0} \right) - v}^{2}.  \label{alpha=2:pre}
\end{align}
where last inequality follows from \eqref{numT:bound:oper} in \cref{lem:numT}.  This and \eqref{numT:opn} in  \cref{prop:numT} give the estimates in case $\alpha =2$.
\end{proof}

\begin{remark}\label{rem47}
From the estimates given in the case of $\alpha=2$, one can see that it is reasonable to choose $v \coloneq x_{0}$, as Lieder did in \cite{Lieder}, since the last two terms in the expression under the square root can be simplified, and therefore the constant depends only on the distance from $x_0$ to the set $\Fix T$. 
\end{remark}

\subsection{The averaged operator case}\label{subsec43}

The above investigations conducted for nonexpansive operator can be extended to $\theta$-averaged operators $T \colon \sH \to \sH$ with $\theta \in \left( 0 , 1 \right]$. This extension provides greater flexibility in the considered approaches. 

A single-valued operator $T \colon \sH \to \sH$ is said to be  \emph{$\theta$-averaged} if there exists a nonexpansive operator $N \colon \sH \to \sH$ such that 
$$T = \left( 1 - \theta \right) \Id + \theta N.$$
In this case,  for every $0 \leq \lambda \leq \frac{1}{\theta}$ the operator $T_{\lambda} \coloneq \left( 1 - \lambda \right) \Id + \lambda T$ is nonexpansive.  Indeed,  $T_{\lambda}$ can be written as
\begin{equation*}
	T_{\lambda} = \left( 1 - \lambda \right) \Id + \lambda T = \left( 1 - \lambda \right) \Id + \lambda \left( \left( 1 - \theta \right) \Id + \theta N \right) = \left( 1 - \lambda \theta \right) \Id + \lambda \theta N ,
\end{equation*}
which is obviously nonexpansive.  Moreover,  it holds $\norm{x - T \left( x \right)} = \frac{1}{\lambda} \norm{x - T_{\lambda} \left( x \right)}$ for every $x \in \sH$ and $\Fix T = \Fix T_{\lambda}$. This shows that, in order to obtain an iterative method that detects the fixed points of $T$,  one simply has to replace $T$ by $T_{\lambda}$ in the iterative scheme \eqref{algo:Halpern},  without affecting the convergence results.

\begin{mdframed}
Let $v, x_{0} \in \sH$ and $0 < \lambda \leq \frac{1}{\theta}$. For every $k \geq 0$ set
\begin{equation}
\label{algo:Tikhonov}
x_{k+1} \coloneq \alpha_{k} v + \left( 1 - \alpha_{k} \right) \left( \left( 1 - \lambda \right) x_{k} + \lambda T \left( x_{k} \right) \right) .
\end{equation}
\end{mdframed}

According to \cref{thm:Halpern-sc}, the sequence $\seq{x_{k}} _{k \geq 0}$ strongly converges to $\xi_{*} \coloneq \proj_{\Fix T} \left( v \right)$ as $k \to + \infty$, provided that $\seq{\alpha_{k}} _{k \geq 0}$ satisfies one of the conditions \eqref{alpha-lim} or \eqref{alpha-BV} in addition to \eqref{cond:alpha-k}.  Convergence rates for the discrete velocity and the fixed point residual can be straighforwardly derived from Theorem \ref{thm:algo:0<q<1} for $\alpha_{k} \coloneq \frac{\alpha-1}{\left( k+b \right) ^{q}}$ for every $k \geq 0$, where $0 < q < 1$, $\alpha > 1$ and $b >0$, and from Theorem \ref{thm:algo:q=1} for $\alpha_{k} \coloneq \frac{\alpha-1}{k+\alpha}$ for every $k \geq 0$, where $\alpha > 1$.

Regarding the strong convergence of the iterates to $\xi_{*} \coloneq \proj_{\Fix T} \left( v \right)$, we will see below that in the case $0 < \lambda < \frac{1}{\theta}$ this is guaranteed only by assuming that the sequence $\seq{\alpha_{k}} _{k \geq 0}$ satisfies \eqref{cond:alpha-k}. The proof adapts an idea used in \cite{Chidume-Chidume, Suzuki:06}.

\begin{theorem}\label{strongconvergenceaveraged}
Let $0 < \lambda < \frac{1}{\theta}$ and assume that the sequence $\seq{\alpha_{k}} _{k \geq 0}$ satisfies only \eqref{cond:alpha-k}.  For given $x_{0}, v \in \sH$, let $\seq{x_{k}} _{k \geq 0}$ be the sequence generated by \eqref{algo:Tikhonov}. Then, the sequence $\seq{x_{k}} _{k \geq 0}$ strongly converges to $\xi_{*} \coloneq \proj_{\Fix T} \left( v \right)$ as $k \to + \infty$.
\end{theorem}
\begin{proof}
Bydefinition,  there exists a nonexpansive operator $N \colon \sH \to \sH$ such that
\begin{equation*}
\left( 1 - \lambda \right) \Id + \lambda T = \left( 1 - \lambda \right) \Id + \lambda \left( \left( 1 - \theta \right) \Id + \theta N \right) = \left( 1 - \lambda \theta \right) \Id + \lambda \theta N .
\end{equation*}
In particular, the iteration \eqref{algo:Tikhonov} becomes
\begin{equation}
x_{k+1} = \alpha_{k} v + \left( 1 - \alpha_{k} \right) \left( \left( 1 - \lambda \theta \right) x_{k} + \lambda \theta N \left( x_{k} \right) \right) \quad \forall k \geq 0.
\end{equation}
From \cref{lem:numT},  we have for every $k \geq 0$
\begin{equation*}
\norm{x_{k}-\xi_{*}} \leq D_{0} \left( x_{0} , v , \xi_{*} \right) ,
\end{equation*}
which, by the triangle inequality and the fact that $\xi_{*} \in \Fix T = \Fix N$, leads to
\begin{equation}
\label{suff:bnd}	
\norm{N \left( x_{k} \right)} \leq \norm{N \left( x_{k} \right) - N \left( \xi_{*} \right)} + \norm{\xi_{*}} \leq \norm{x_{k} - \xi_{*}} + \norm{\xi_{*}} \leq D_{0} \left( x_{0} , v , \xi_{*} \right) + \norm{\xi_{*}} .
\end{equation}
These shows that both sequences $\seq{x_{k}}_{k \geq 0}$ and $\seq{N \left( x_{k} \right)}_{k \geq 0}$ are bounded.

For every $k \geq 0$,  we define
\begin{align*}
\theta_{k} 	& \coloneq \alpha_{k} + \left( 1 - \alpha_{k} \right) \lambda \theta \in \left( 0 , 1 \right) , \nonumber \\
w_{k} 		& \coloneq \dfrac{1}{\alpha_{k} + \left( 1 - \alpha_{k} \right) \lambda \theta} \left( \alpha_{k} v + \left( 1 - \alpha_{k} \right) \lambda \theta N \left( x_{k} \right) \right) .
\end{align*}
Obviously, the sequence $\seq{w_{k}}_{k \geq 0}$ is also bounded. Furthermore,  $1 - \theta_{k} = \left( 1 - \alpha_{k} \right) \left( 1 - \lambda \theta \right)$ and $\lim_{k \to + \infty} \theta_{k} = \lambda \theta \in \left( 0 , 1 \right)$.
For every $k \geq 0$, we have
\begin{align*}
w_{k} = \dfrac{1}{\theta_{k}} \left( \alpha_{k} v + \left( 1 - \alpha_{k} \right) \lambda \theta N \left( x_{k} \right) \right) = \dfrac{1}{\theta_{k}} \left( x_{k+1} - \left( 1 - \theta_{k} \right) x_{k} \right) 
\end{align*}
or, equivalently,
\begin{equation*}
x_{k+1} = \left( 1 - \theta_{k} \right) x_{k} + \theta_{k} w_{k} .
\end{equation*}
Some simple calculations give for every $k \geq 0$
\begin{align*}
w_{k+1} - w_{k} 
& = \dfrac{1}{\theta_{k+1}} \left( \alpha_{k+1} v + \left( 1 - \alpha_{k+1} \right) \lambda \theta N \left( x_{k+1} \right) \right) - \dfrac{1}{\theta_{k}} \left( \alpha_{k} v + \left( 1 - \alpha_{k} \right) \lambda \theta N \left( x_{k} \right) \right) \nonumber \\
& = \left( \dfrac{\alpha_{k+1}}{\theta_{k+1}} - \dfrac{\alpha_{k}}{\theta_{k}} \right) v + \dfrac{\left( 1 - \alpha_{k+1} \right) \lambda \theta}{\theta_{k+1}} \left( N \left( x_{k+1} \right) - N \left( x_{k} \right) \right) + \left( \dfrac{1 - \alpha_{k+1}}{\theta_{k+1}} - \dfrac{1 - \alpha_{k}}{\theta_{k}} \right) \lambda \theta N \left( x_{k} \right),
\end{align*}
therefore, by using that $\theta_{k} = \alpha_{k} + \left( 1 - \alpha_{k} \right) \lambda \theta \geq \left( 1 - \alpha_{k} \right) \lambda \theta$ and \eqref{suff:bnd},  we obtain the folowing estimates
\begin{align*}
\norm{w_{k+1} - w_{k}} - \norm{x_{k+1} - x_{k}} & \leq \abs{\dfrac{\alpha_{k+1}}{\theta_{k+1}} - \dfrac{\alpha_{k}}{\theta_{k}}} \norm{v} + \dfrac{\left( 1 - \alpha_{k+1} \right) \lambda \theta}{\theta_{k+1}} \norm{N \left( x_{k+1} \right) - N \left( x_{k} \right)} \nonumber \\
& \quad + \abs{\left( \dfrac{1 - \alpha_{k+1}}{\delta_{k+1}} - \dfrac{1 - \alpha_{k}}{\theta_{k}} \right)} \lambda \theta \norm{N \left( x_{k} \right)} - \norm{x_{k+1} - x_{k}} \nonumber \\
& \leq \abs{\dfrac{\alpha_{k+1}}{\theta_{k+1}} - \dfrac{\alpha_{k}}{\theta_{k}}} \norm{v} + \left( \dfrac{\left( 1 - \alpha_{k+1} \right) \lambda \theta}{\theta_{k+1}} - 1 \right) \norm{x_{k+1} - x_{k}} \nonumber \\
& \quad + \abs{\left( \dfrac{1 - \alpha_{k+1}}{\theta_{k+1}} - \dfrac{1 - \alpha_{k}}{\theta_{k}} \right)} \lambda \theta \norm{N \left( x_{k} \right)} \nonumber \\
& \leq \abs{\dfrac{\alpha_{k+1}}{\theta_{k+1}} - \dfrac{\alpha_{k}}{\theta_{k}}} \norm{v} + \abs{\left( \dfrac{1 - \alpha_{k+1}}{\theta_{k+1}} - \dfrac{1 - \alpha_{k}}{\theta_{k}} \right)} \lambda \theta \left( D_{0} \left( x_{0} , v , \xi_{*} \right) + \norm{\xi_{*}} \right) .
\end{align*}

This guarantees that $\limsup_{k \to + \infty} \left( \norm{w_{k+1}-w_{k}} - \norm{x_{k+1}-x_{k}} \right) \leq 0$, which, by using \cref{lem:Suzuki}, allows us to conclude that $\lim_{k \to + \infty} \norm{w_{k}-x_{k}} = 0$. Consequently,  $\lim_{k \to + \infty} \norm{x_{k+1}-x_{k}} = \lim_{k \to + \infty} \theta_k \norm{w_{k}-x_{k}} = 0$, which, by using \eqref{numT:opn}, further implies $\lim_{k \to + \infty} \norm{x_{k}-T(x_{k})} = 0$. The conclusion follows from  \cref{prop:numT} (ii).
\end{proof}

\begin{remark}\label{rmk:bcm}
Bo\c{t}, Csetnek, and Meier proposed in \cite{Bot-Csetnek-Meier} for $z_0 \in \sH$ the following iterative scheme
\begin{equation}
	\label{algo:BCM}
	z_{k+1} \coloneq \beta_{k} z_{k} + \lambda \left( T \left( \beta_{k} z_{k} \right) - \beta_{k} z_{k} \right) \quad \forall k \geq 0,
\end{equation}
with $\seq{\beta_{k}}_{k \geq 0} \subseteq (0,1], \lim_{k \rightarrow + \infty} \beta_k =1, \sum_{k \geq 0} (1-\beta_k) = +\infty$ and $\sum_{k \geq 1} |\beta_{k}-\beta_{k-1}| < +\infty$,  and $0 < \lambda \leq \frac{1}{\theta}$,  designed with the aim of generating a strongly convergent sequence $\seq{z_{k}}_{k \geq 0}$ to the fixed point of a $\theta$-averaged operator $T : \sH \rightarrow \sH$ with minimum norm.

By setting $x_{k} \coloneq \beta_{k} z_{k}$ for every $k \geq 0$, the iterative scheme \eqref{algo:BCM} can be rewritten as
\begin{equation}
	\label{algo:BCM-H}
	x_{k+1} \coloneq \beta_{k+1} \left((1-\lambda)x_{k} + \lambda T \left( x_{k} \right) \right) \quad \forall k \geq 0,
\end{equation}
which corresponds to \eqref{algo:Tikhonov} with $\alpha_k \coloneq 1- \beta_{k+1}$ for every $k \geq 0$,  and $v \coloneq 0$.  Note that the strong convergence of $\seq{x_{k}}_{k \geq 0}$ to  $\proj_{\Fix T} \left( 0 \right)$ is a direct consequence of Theorem \ref{thm:Halpern-sc} in the context of fulfilling condition \eqref{alpha-BV}.  Since $\lim_{k \rightarrow + \infty} \beta_k =1$, this says nothing more than that the sequence $\seq{x_{k}}_{k \geq 0}$ strongly converges to  $\proj_{\Fix T} \left( 0 \right)$ as $k \rightarrow +\infty$.  Furthermore,  if $0 < \lambda \leq \frac{1}{\theta}$, then, according to Theorem \ref{strongconvergenceaveraged},  the same conclusion follows even without assuming that $\sum_{k \geq 1} |\beta_{k}-\beta_{k-1}| < +\infty$.

In additon, the rates of convergence for $\norm{x_{k} - T \left( x_{k} \right)}$,  as provided by the Halpern iteration, can be transferred to the discrete velocity and the fixed point residual expressed in terms of the original sequence $\seq{z_{k}}_{k \geq 0}$, namely, by using the boundedness of the latter and that for every $k \geq 1$, the following holds
\begin{align*}
\norm{z_{k+1} -  z_{k}} & \leq (1- \beta_k) \norm{z_{k}} + \lambda \norm{x_{k} - T \left( x_{k} \right)} = \alpha_{k-1} \norm{z_{k}} + \lambda \norm{x_{k} - T \left( x_{k} \right)} ,
\end{align*}
and
\begin{align*}
\norm{z_{k} - T \left( z_{k} \right)}  \leq & \ \norm{z_{k} - x_{k}} + \norm{x_{k} - T \left( x_{k} \right)} + \norm{T \left( x_{k} \right) - T \left( z_{k} \right)} \nonumber\\
 \leq & \ 2 \norm{z_{k} - x_{k}} + \norm{x_{k} - T \left( x_{k} \right)} =  2 (1- \beta_k)  \norm{z_{k}} + \norm{x_{k} - T \left( x_{k} \right)} \nonumber \\
= & \ 2\alpha_{k-1}  \norm{z_{k}} + \norm{x_{k} - T \left( x_{k} \right)},
\end{align*}
respectively.
\end{remark}

\begin{remark} 
\label{rmk:Kim} 
For $A \colon \sH \rightrightarrows \sH$ a (set-valued) maximally monotone operator, the problem of finding a zero of $A$, this is an element in the set $\zer A \coloneq\{x \in \sH \colon 0 \in A(x) \}$, is for every $\eta >0$ equivalent to the problem of finding a fixed point of the resolvent $J_{\eta A} \colon \sH \rightarrow \sH, J_{\eta A} = (\Id + \eta A)^{-1}$, which is a single-valued and $\frac{1}{2}$-averaged operator.  The classical proximal point algorithm,  defined for a given $x_0 \in \sH$ as
\begin{equation*}
x_{k+1} = J_{\eta A} \left( x_{k} \right) \quad \forall k \geq 0,
\end{equation*}
produces a sequence $\seq{x_{k}}_{k \geq 0}$ that weakly converges to a zero of $A$ as $k \rightarrow +\infty$.  The fixed point residual $\left\lVert J_{\eta A} \left( x_{k} \right) - x_{k} \right\rVert$ vanishes with a rate of convergence of $\bO \left(\frac{1}{\sqrt{k}} \right)$ as $k \rightarrow +\infty$.

Gu and Yang have shown in  \cite{Gu-Yang} that for $\sH \coloneq \sR^{N}$, this rate of convergence can be improved to
\begin{equation*}
	\left\lVert J_{\eta A} \left( x_{k} \right) - x_{k} \right\rVert = \begin{dcases}
		\bO \left (\dfrac{1}{k} \right),	& \textrm{ if } N = 1 , \\
		\bO \left( \dfrac{1}{\sqrt{\left( 1 + \frac{1}{k} \right) ^{k} k}} \right), & \textrm{ if } N \geq 2 .
	\end{dcases} \quad \mbox{as} \ k \rightarrow +\infty.
\end{equation*}

In the setting of finite-dimensional Hilbert spaces,  Kim introduced in \cite{Kim} (see also \cite{Park-Ryu}) the following iterative scheme,  which, for given $y_{0} = x_{-1} = x_{0} \in \sH$,  for every $k \geq 0$,  reads
\begin{align}
	\begin{split}
		\label{algo:APPM}
		y_{k+1} 	& \coloneq J_{\eta A} \left( x_{k} \right) , \\
		x_{k+1}	& \coloneq y_{k+1} + \dfrac{k}{k+2} \left( y_{k+1} - y_{k} \right) - \dfrac{k}{k+2} \left( y_{k} - x_{k-1} \right) .
	\end{split}
\end{align}
Using the performance estimation approach,  he proved that this method has a fixed point residual  convergence rate of $\bO \left (\frac{1}{k} \right)$ as $k \rightarrow +\infty$.

One can note that, for every $k \geq 1$,  the second update formula in \eqref{algo:APPM} can be equivalently written as
\begin{align*}
	\left( k+2 \right) x_{k+1} & = 2(k+1)y_{k+1} - 2k y_{k} + k x_{k-1} , \\
	(k+1)x_{k} & = 2k y_{k} - 2 \left( k-1 \right) y_{k-1} + \left( k-1 \right) x_{k-2} , \\
	& \vdots \\
	3x_{2} & = 4y_{3} - 2y_{1} + x_{0} , \\
	2x_{1} & = 2y_{1}.
\end{align*}
Summing these relations, we obtain that for every $k \geq 0$
\begin{equation*}
	\left( k+2 \right) x_{k+1} + \left( k+1 \right) x_{k} = 2 \left( k+1 \right) y_{k+1} + x_{0} .
\end{equation*}
In the view of the first equation in \eqref{algo:APPM},  it yields that the sequence  $\seq{x_{k}}_{k \geq 0}$ generated by Kim's algorithm fulfills
\begin{equation*}
	x_{k+1} = \dfrac{1}{k+2} x_{0} + \dfrac{k+1}{k+2} \left( 2J_{\gamma A} \left( x_{k} \right) - x_{k} \right) \quad \forall k \geq 0.
\end{equation*}
Consequently,  according to Theorem \ref{thm:Halpern-sc} applied to averaged operators, even in infinite-dimensional Hilbert spaces, the sequence  $\seq{x_{k}}_{k \geq 0}$ strongly converges to $\proj_{\zer A}(x_0)$ as $k \rightarrow +\infty$. In addition,  from Theorem \ref{thm:algo:q=1} in case $\alpha=2$ it follows that $\norm{x_{k+1} - x_k}$ and $\left\lVert x_{k} - J_{\eta A} \left( x_{k} \right) \right\rVert$ converge to zero with a rate of convergence of $\bO \left (\frac{1}{k} \right)$ as $k \rightarrow +\infty$.
\end{remark}

\appendix
\section{Appendix}\label{secappendix}

In the appendix, we discuss the asymptotic behaviour of the trajectory generatd by \eqref{ds:simp} in case $M$ is an cocoercive operator (see \cite[Theorem 11]{Attouch-Bot-Nguyen} for a general formulation),  and outline additional results that play an essential role in the convergence of the continuous time models.

\subsection{The cocoercive case}
\label{missingproof}

\begin{theorem}
\label{thm:cocoercive}
Let $M \colon \sH \to \sH$ be a $\rho$-cocoercive operator for some $\rho > 0$, that is
\begin{equation*}
	\scal{M \left( x \right) - M \left( y \right) , x - y} \geq \rho \norm{M \left( x \right) - M \left( y \right)}^{2} \quad \forall x, y \in \sH,
\end{equation*}
with $\zer M \neq \emptyset$,  and $x \colon \left[ t_{0} , + \infty \right) \to \sH$ be the solution trajectory of
\begin{equation*}
\dot{x} \left( t \right) + M \left( x \left( t \right) \right) = 0,
\end{equation*}
with initial condition $x(t_{0}) \coloneq x_0 \in \sH$. Then, the following statements are true:
\begin{enumerate}
\item 
it holds
$
\left\lVert M \left( x \left( t \right) \right) \right\rVert = o \left( \frac{1}{\sqrt{t}}  \right) \textrm{ as } t \to + \infty ;
$
\item 
the solution trajectory $x(t)$ converges weakly to an element of $\zer M$ as $t \to +\infty$.
\end{enumerate}
\end{theorem}

\begin{proof}
Let $\xi \in \zer M$.  For every $t \geq t_{0}$, we consider
$$
\E \left( t \right) \coloneq \dfrac{1}{2} \left\lVert x \left( t \right) - \xi  \right\rVert ^{2}.
$$
Then, for almost every $t \geq t_{0}$ it holds
\begin{align*}
\dfrac{d}{dt} \E \left( t \right) & = \left\langle x \left( t \right) - \xi , \dot{x} \left( t \right) \right\rangle = - \left\langle x \left( t \right) - \xi,  M \left( x \left( t \right) \right) \right\rangle .
\end{align*}
By integration, we get for every $t \geq t_{0}$
\begin{align*}
\dfrac{1}{2} \left\lVert x \left( t \right) - \xi \right\rVert ^{2}  \leq \dfrac{1}{2} \left\lVert x \left( t \right) - \xi \right\rVert ^{2} + \int_{t_{0}}^{t} \left\langle x \left( s \right) - \xi , M \left( x \left( s \right) - M(\xi) \right) \right\rangle ds  \leq \dfrac{1}{2} \left\lVert x \left( t_{0} \right) - \xi \right\rVert ^{2}.
\end{align*}
This gives the boundedness of the trajectory.  On the other hand,  by invoking the $\rho$-cocoercivity of $M$ and that $M(\xi)=0$, and by letting $t$ converge to infinity, we obtain
\begin{equation}\label{eq:cocoer_100}
\rho \int_{t_{0}}^{+ \infty} \left\lVert M \left( x \left( s \right) \right) \right\rVert ^{2} ds \leq \int_{t_{0}}^{+ \infty} \left\langle x \left( s \right) - \xi, M \left( x \left( s\right) \right) \right\rangle du \leq \dfrac{1}{2} \left\lVert x \left( t_{0} \right) - \xi \right\rVert ^{2}  < + \infty.
\end{equation}
Therefore, $\liminf_{t \to + \infty} t \left\lVert M \left( x \left( t \right) \right) \right\rVert ^{2} = 0$. By computing the time derivative at almost every $t \geq t_{0}$ and by making use of the monotonicity of $M$, we get
\begin{align*}
\dfrac{d}{dt} \left( \dfrac{1}{2} t \left\lVert M \left( x \left( t \right) \right) \right\rVert ^{2} \right) & = \dfrac{1}{2}\left\lVert M \left( x \left( t \right) \right) \right\rVert ^{2} + t \left\langle M \left( x \left( t \right) \right) , \dfrac{d}{dt} \left( M \left( x \left( t\right) \right) \right) \right\rangle \nonumber \\
& = \dfrac{1}{2}\left\lVert M \left( x \left( t \right) \right) \right\rVert ^{2} - t \left\langle \dot{x} \left( t \right) , \dfrac{d}{dt} \left( M \left( x \left( t \right) \right) \right) \right\rangle \\
& \leq \dfrac{1}{2}\left\lVert M \left( x \left( t \right) \right) \right\rVert ^{2}.
\end{align*}

The right-hand side of the above inequality belongs to $\mathbb{L}^{1} \left( \left[ t_{0} , + \infty \right); \sH \right)$ thanks to \eqref{eq:cocoer_100}.  According to Lemma \ref{lem:lim-R}, $\lim_{t \to + \infty} t \left\lVert M \left( x \left( t \right) \right) \right\rVert ^{2} \in \mathbb{R}$ exists, and therefore it must be equal to $0$. In other words,
$$\left\lVert M \left( x \left( t \right) \right) \right\rVert = o \left( \dfrac{1}{\sqrt{t}} \right) \textrm{ as } t \to + \infty,
$$
which means that statement (i) is proven.

Using that for almost every $t \geq t_{0}$ it holds
\begin{align*}
\dfrac{d}{dt} \left( \dfrac{1}{2} \left\lVert x \left( t \right) - \xi  \right\rVert ^{2} \right)  = \left\langle x \left( t\right) - \xi, \dot{x} \left( t \right) \right\rangle = - \left\langle x \left( t \right) - \xi , M \left( x \left( t \right) \right) \right\rangle \leq 0,
\end{align*}
thanks to Lemma \ref{lem:lim-R},  we obtain that the limit $\lim_{t \to + \infty} \left\lVert x \left( t \right) - \xi \right\rVert \in \sR$ exists. This shows that the first condition in Opial's Lemma (see Lemma \ref{lem:Opial:cont}) is satisfied.

Using that $\lim_{t \to +\infty} M \left( x \left( t \right) \right) =0$ and the sequential closedness of the graph of $M$ in the weak $\times$ strong topology of $\sH \times \sH$,  it yields that every weak limit point of the trajectory $x(t)$ belongs to $\zer M$. This shows that the second condition in Opial's Lemma is also satisfied. 

Consequently, according to Lemma \ref{lem:Opial:cont}, the trajectory $x(t)$ converges weakly to an element in $\zer M$ as $t \rightarrow +\infty$, which proves (ii).
\end{proof}

\subsection{Weak convergence of the trajectory under integrability of the regularization function}

As seen in Proposition \eqref{prop:Tikh:converge}, the condition \eqref{cond:eps} required for the regularization function, in combination with the assumption that $\norm{M \left( x(t) \right)} \to 0$ as $t \to +\infty$,  suffice to ensure the strong convergence of the trajectory $x(\cdot)$ towards the minimum norm solution of \eqref{intro:mono}. In contrast, the following proposition demonstrates that assuming convergence rather than divergence for the integral of the regularization function results in the weak convergence of the trajectory $x(\cdot)$ to a solution of \eqref{intro:mono}.

\begin{proposition}
\label{prop:weak}
Let $x \colon \left[ t_{0} , + \infty \right) \to \sH$ be the trajectory solution of the first-order dynamical system with Tikhonov regularization \eqref{ds:Tikhonov}.  Regarding the regularization function $\varepsilon(\cdot)$, instead of \eqref{cond:eps}, suppose that
\begin{equation*}
\lim_{t \to + \infty} \varepsilon \left( t \right) = 0
\quad \textrm{ and } \quad
\int_{t_{0}}^{+ \infty} \varepsilon \left( t \right) dt < + \infty.
\end{equation*}
If $\norm{M \left( x \left( t \right) \right)} \to 0$ as $t \to + \infty$, then $x(t)$ converges weakly to an element of $\zer M$ as $t \to + \infty$.
\end{proposition}
\begin{proof}
Observing that statement \ref{lem:bnd:dphi} in Lemma \ref{lem:bnd} holds regardless of condition \eqref{cond:eps}. Hence, let $\xi \in \zer M$, we have for almost every $t \geq t_{0}$
\begin{equation*}
\dot{\varphi_{\xi}} \left( t \right) \leq \dfrac{\varepsilon \left( t \right)}{2} \norm{\xi}^{2} .
\end{equation*}
In particular, according to Lemma \ref{lem:lim-R}, the limit $\lim_{t \to + \infty} \varphi_{{\xi}}(t) \in \sR$ exists. This yields that the first condition in Opial's Lemma (see Lemma \ref{lem:Opial:cont}) is satisfied.

The demonstration that the second condition in Opial's Lemma is satisfied under the additional assumption $\norm{M \left( x(t) \right)} \to 0$ as $t \to +\infty$ has been presented in
the proof of Proposition \ref{prop:Tikh:converge}.

The trajectory $x(\cdot)$ converging weakly to a solution of \eqref{intro:mono} follows as a consequence of Lemma \ref{lem:Opial:cont}.
\end{proof}

\subsection{Auxiliary results used in the convergence analysis of the continuous time models}\label{subseca2}

The following result and its proof can be referenced in \cite[Lemma 5.1]{Abbas-Attouch-Svaiter}.

\begin{lemma}
\label{lem:lim-R}
Let $h \colon \left[t_{0} , + \infty \right) \to \sR$ be a locally absolutely continuous function that is bounded from below, and $g \in \sL^{1} \left( \left[t_{0} , + \infty \right); \sR \right)$ be such that
\begin{equation*}
\dot{h} \left( t \right) \leq g \left( t \right) \quad \mbox{for almost every} \ t \geq t_{0}.
\end{equation*}
Then the limit $\lim\limits_{t \to + \infty} h \left( t \right) \in \sR$ exists.
\end{lemma}

The continuous form of Opial’s Lemma (\cite{Opial}) serves as the primary tool utilized to establish the weak convergence of the trajectory.

\begin{lemma}
\label{lem:Opial:cont}
Let $\mathcal{S}$ be a nonempty subset of $\sH$ and $x \colon \left[ t_{0} , + \infty \right) \to \sH$.
Assume that
\begin{enumerate}
\item
\label{lem:Opial:cont:i}
for every $\xi \in \mathcal{S}$, $\lim\limits_{t \to + \infty} \left\lVert x \left( t \right) - \xi \right\rVert$ exists;

\item
\label{lem:Opial:cont:ii}
every weak sequential cluster point of the trajectory $x \left( t \right)$ as $t \to + \infty$ belongs to $\mathcal{S}$.
\end{enumerate}
Then $x(t)$ converges weakly to a point in $\mathcal{S}$ as $t \to + \infty$.
\end{lemma}

\begin{lemma}(\cite[Lemma A.4]{Attouch-Peypouquet-Redont})
\label{lem:eint}
Take $t_{0} > 0$, and let $h \in \sL^{1} \left( \left[ t_{0} , + \infty \right) \right)$ be nonnegative and continuous. Consider anondecreasing function $\gamma \colon \left[ t_{0} , + \infty \right) \to \sR_{+}$ such that $\lim_{t \to + \infty} \gamma \left( t \right) = + \infty$. Then,
\begin{equation*}
\lim\limits_{t \to + \infty} \dfrac{1}{\gamma \left( t \right)} \int_{t_{0}}^{t} \gamma \left( r \right) h \left( r \right) = 0 .
\end{equation*}
\end{lemma}

The following result, provided in \cite[Lemma 1]{Cominetti-Peypouquet-Sorin}, is employed in proving the strong convergence of the trajectory.
\begin{lemma}
\label{lem:limsup}
Let $h \colon \left[t_{0} , + \infty \right) \to \sR$ be a locally absolutely continuous function, $g \colon \left[t_{0} , + \infty \right) \to \sR$ a bounded function and $\varepsilon \colon \left[t_{0} , + \infty \right) \to \sR_+$ a locally integrable function such that
\begin{equation*}
\dot{h} \left( t \right) + \varepsilon \left( t \right) h \left( t \right) \leq \varepsilon \left( t \right) g \left( t \right) \quad \mbox{for almost every} \ t \geq t_{0}.
\end{equation*}
Then $h$ is also bounded and, if $\int_{t_{0}}^{+ \infty} \varepsilon \left( t \right) dt = + \infty$, then $\limsup_{t \to + \infty} h \left( t \right) \leq \limsup_{t \to + \infty} g \left( t \right)$.
\end{lemma}

Gronwall's Lemma (\cite{Gronwall}) in integral form for continuous functions will be utilized to establish the strong convergence of the trajectory to the minimal solution under a certain condition. 

\begin{lemma}
\label{lem:gronwall}
Let $h, g \colon \left[t_{0} , + \infty \right) \to \sR$ be continuous functions and $k \colon \left[t_{0} , + \infty \right) \to \sR$ a function with the property that its negative part is integrable on every closed and bounded subinterval of $\left[t_{0} , + \infty \right)$. If $h$ is non-negative and
\begin{equation*}
h\left( t \right) \leq k(t) +  \int_{t_{0}}^t g(u)h(u)du \quad \mbox{for every} \ t \geq t_{0},
\end{equation*}
then
\begin{equation*}
h\left( t \right) \leq k(t) +  \int_{t_{0}}^t k(u)g(u) \exp \left(\int_{u}^t g(r)dr \right) du \quad \mbox{for every} \ t \geq t_{0}.
\end{equation*}
If, in addition, $k$ is nondecreasing, then
\begin{equation*}
h\left( t \right) \leq k(t) \exp \left(\int_{t_{0}}^t g(u) du\right) \quad \mbox{for every} \ t \geq t_{0}.
\end{equation*}
\end{lemma}

The following lemma, corresponding to \cite[Lemma A.5]{Brezis}, is used to derive the covergence rates.

\begin{lemma}
\label{lem:Ou-Iang}
Let $g \colon \left[t_{0} , T_0 \right] \to \sR_{+}$ be an integrable function, and $c \geq 0$.
If $h \colon \left[t_{0} , T_0 \right] \to \sR$ is a continuous function such that
\begin{equation}
\label{Ou-Iang:inq}
\dfrac{1}{2} h^{2} \left( t \right) \leq \dfrac{1}{2} c^{2} + \int_{t_{0}}^{t} g \left( r \right) h \left( r \right) dr \quad \mbox{for every} \ t \in \left[t_{0} , T_0 \right],
\end{equation}
then it holds
\begin{equation*}
\abs{h \left( t \right)} \leq c + \int_{t_{0}}^{t} g \left( r \right) dr \quad \mbox{for every} \ t \in \left[t_{0} , T_0 \right] .
\end{equation*}
\end{lemma}

\subsection{Auxiliary results used in the convergence analysis of the discrete time models}\label{subseca3}

The following result by Xu (see \cite[Lemma 2.5]{Xu:02}) and will play a crucial role in the proof of the strong convergence of the sequence of iterates.

\begin{lemma}
\label{lem:lims-0}
Let $\seq{a_{k}}_{k \geq 0}$ be a sequence of nonnegative real numbers satisfying for every $k \geq 0$
\begin{equation}
\label{limsup-0:inq}
a_{k+1} \leq \left( 1 - \alpha_{k} \right) a_{k} + \alpha_{k} b_{k} + d_{k} ,
\end{equation}	
where $\seq{\alpha_{k}}_{k \geq 0}$, $\seq{b_{k}}_{k \geq 0}$, and $\seq{d_{k}}_{k \geq 0}$ are sequences of real numbers with the properties
\begin{enumerate}
\item $\seq{\alpha_{k}}_{k \geq 0} \subseteq \left[ 0 , 1 \right]$ and $\sum_{k \geq 0} \alpha_{k} = + \infty$;
\item $\limsup_{k \to + \infty} b_{k} \leq 0$;
\item $\sum_{k \geq 0} d_{k} < + \infty$.
\end{enumerate}
Then, $\lim_{k \to + \infty} a_{k} = 0$.
\end{lemma}

The following two lemmas are used to derive the convergence rates for the discrete time models.  For their proofs we refer to  \cite{Polyak:book}.
\begin{lemma}
\label{lem:rate:0<q<1}
Let $\seq{\xi_{k}}_{k \geq 1}$ be a sequence of nonnegative real numbers satisfying
\begin{equation}
\label{rate:0<q<1:inq}
\xi_{k+1} \leq \left( 1 - \dfrac{c}{\left( k+b \right) ^{q}} \right) \xi_{k} + \dfrac{d}{\left( k+b \right) ^{q} \left( k+b-1 \right)} \quad \forall k \geq 1,
\end{equation}
where $0 < q < 1$, $0 < c < b^q$,  and $d \geq 0$. Then,  it holds as $k \to + \infty$
\begin{equation*}
\xi_{k} = \bO \left( \dfrac{1}{k+b} \right).
\end{equation*}
\end{lemma}

\begin{lemma}
\label{lem:rate:q=1}
Let $\seq{\xi_{k}}_{k \geq 1}$ be a sequence of nonnegative real numbers satisfying
\begin{equation}
\label{rate:q=1:inq}
\xi_{k+1} \leq \left( 1 - \dfrac{c}{k+b} \right) \xi_{k} + \dfrac{d}{\left( k+b \right) \left( k+b-1 \right)}  \quad \forall k \geq 1,
\end{equation}
where $0 < c < b$ and $d \geq 0$. Then,  it holds as $ k \to + \infty$
\begin{equation*}
\xi_{k} = \begin{dcases}
\bO \left( \dfrac{1}{k+b} \right) & \textrm{ if } c > 1 \\
\bO \left( \dfrac{\log \left( k \right)}{k+b} \right) & \textrm{ if } c = 1 \\
\bO \left( \dfrac{1}{\left( k+b \right) ^{c}} \right) & \textrm{ if } c < 1 .
\end{dcases}
\end{equation*}
\end{lemma}

The following result is \cite[Lemma 2.2]{Suzuki:05} and plays an important role in the proof of the strong convergence of the Halpern iteration in the case of averaged operators.

\begin{lemma}
\label{lem:Suzuki}
Let $\seq{x_{k}}_{k \geq 0}$ and $\seq{w_{k}}_{k \geq 0}$ be bounded sequences in $\sH$ and $\seq{\theta_{k}}_{k \geq 0}$ be a sequence in $[0,1]$ with $0 < \liminf_{k \to + \infty} \theta_{k} \leq \limsup_{k \to + \infty} \theta_{k} < 1$.  If $x_{k+1} = \left( 1 - \theta_{k} \right) x_{k} + \theta_{k} w_{k}$ for every $k \geq 0$ and 
\begin{equation*}
\limsup_{k \to + \infty} \left( \norm{w_{k+1}-w_{k}} - \norm{x_{k+1}-x_{k}} \right) \leq 0,
\end{equation*}
then $\lim_{k \to + \infty} \norm{w_{k}-x_{k}} = 0$.
\end{lemma}

%
%
%
%


{\bf Acknowledgements.} The authors are thankful to Ern\"o Robert Csetnek (University of Vienna) for comments and suggestions which improved the quality of the manuscript.

\end{document}